\documentclass[final,leqno,a4paper]{amsart}
\usepackage{amsmath,amssymb,amsfonts,mathtools,tikz,arydshln,
  mathrsfs,algorithm,algorithmic,graphicx,subcaption,etoolbox,soul,csquotes}
\usepackage[utf8]{inputenc} 
\usepackage[T1]{fontenc}
\usepackage[toc,page]{appendix} 
\usepackage[space,noadjust,nocompress]{cite}

\theoremstyle{definition}

\newtheorem{theorem}{Theorem}[section]
\newtheorem{lemma}{Lemma}[section]

\newcommand\scalemath[2]{\scalebox{#1}{\mbox{\ensuremath{\displaystyle#2}}}}
\newcommand{\mesh}{\mathbb{T}} 
\newcommand{\cell}{\kappa}
\newcommand{\meshfaces}{\mathbb{F}} 
\newcommand{\face}{f}
\newcommand{\ipbf}[2]{a_h\left(#1,#2\right)} 

\newcommand{\eps}{\varepsilon} 
\newcommand{\jump}[1]{\left[\!\left[#1\right]\!\right]}
\newcommand{\av}[1]{\left\{\!\!\left\{#1\right\}\!\!\right\}}
\newcommand{\avv}[1]{\left\{\!\!\!\left\{#1\right\}\!\!\!\right\}}
 
\newcommand{\e}[1]{e^{#1}}
\newcommand{\I}{i}

\newcommand{\dd}{\delta_0}

\allowdisplaybreaks

\begin{document} 

\title[Optimization of Two-level Methods for DG applied to reaction
diffusion]{Optimization of two-level methods for DG discretizations of
reaction-diffusion equations} \author{Martin Jakob Gander \and José
Pablo Lucero Lorca}

\maketitle

\begin{abstract}
We analyze and optimize two-level methods applied to a symmetric
interior penalty discontinuous Galerkin finite element discretization
of a singularly perturbed reaction-diffusion equation. Previous
analyses of such methods have been performed numerically by Hemker
et. al. for the Poisson problem. Our main innovation is that we obtain
explicit formulas for the optimal relaxation parameter of the
two-level method for the Poisson problem in 1D, and very accurate
closed form approximation formulas for the optimal choice in the
reaction-diffusion case in all regimes. Our Local Fourier Analysis,
which we perform at the matrix level to make it more accessible to the
linear algebra community, shows that for DG penalization parameter
values used in practice, it is better to use \emph{cell} block-Jacobi
smoothers of Schwarz type, in contrast to earlier results suggesting
that \emph{point} block-Jacobi smoothers are preferable, based on a
smoothing analysis alone. Our analysis also reveals how the
performance of the iterative solver depends on the DG penalization
parameter, and what value should be chosen to get the fastest
iterative solver, providing a new, direct link between DG
discretization and iterative solver performance. We illustrate our
analysis with numerical experiments and comparisons in higher
dimensions and different geometries.

\smallskip
\noindent \textbf{Keywords.} Reaction-diffusion, Discontinuous
Galerkin, Interior Penalty, Finite Element Method, block-Jacobi,
Two-level, Multigrid, Optimization, Local Fourier Analysis
\end{abstract}

\section{Introduction} 

Reaction-diffusion equations are differential equations arising from
two of the most basic interactions in nature: reaction models the
interchange of a substance from one type to another, and diffusion its
displacement from a point to its neighborhood. Chemical reactors,
radiation transport, and even stock option prices, all have regimes
where their mathematical model is a reaction-diffusion equation with
applications ranging from engineering to biology and finance
\cite{Smoller1994, KanschatLucero16, Manteuffel1998, Fife1979,
Becherer2005}.

In this paper, we present and analyze two-level methods to solve a
symmetric interior penalty discontinuous Galerkin (SIPG)
discretization of a singularly perturbed reaction-diffusion
equation. Symmetric interior penalty methods \cite{Arnold2002,
Nitsche1971, Baker1977, Arnold1982, Wheeler1978} are particularly
interesting to solve these equations since by imposing boundary
conditions weakly they produce less oscillations near the boundaries
in singularly perturbed problems \cite{LewBuscaglia2008}. Using this
discretization, the reaction operator involves only volume integrals
with no coupling between cells. Therefore, all its contributions are
included inside the local subspaces when using \emph{cell}
block-Jacobi smoothers, which can then be interpreted as
non-overlapping Schwarz smoothers (see
\cite{FengKarakashian2001,Dryja2016,LuceroKanschat2020} and references
therein). On the other hand, also \emph{point} block-Jacobi smoothers
have been considered in the literature, which we study as well.

The SIPG method leaves two parameters to be chosen by the user. One is
the penalty parameter, which determines how discontinuous the solution
is allowed to be between cells, and the other is the relaxation used
for the stationary iteration. For classical finite element or finite
difference discretizations of Poisson problems, it is sufficient to
optimize the smoother alone by maximizing the damping in the high
frequency range to get best performance of the two and multilevel
method, which leads for a Jacobi smoother to the damping parameter
$\frac{2}{3}$ (see \cite{Zhou2009}). This is however different for
SIPG discretizations, as we show in Figure \ref{fig:diskmesh}
\begin{figure}
  \raggedleft
  \includegraphics[width=0.23\textwidth]{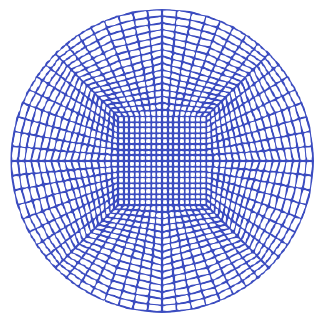}
  \includegraphics[width=0.7\textwidth]{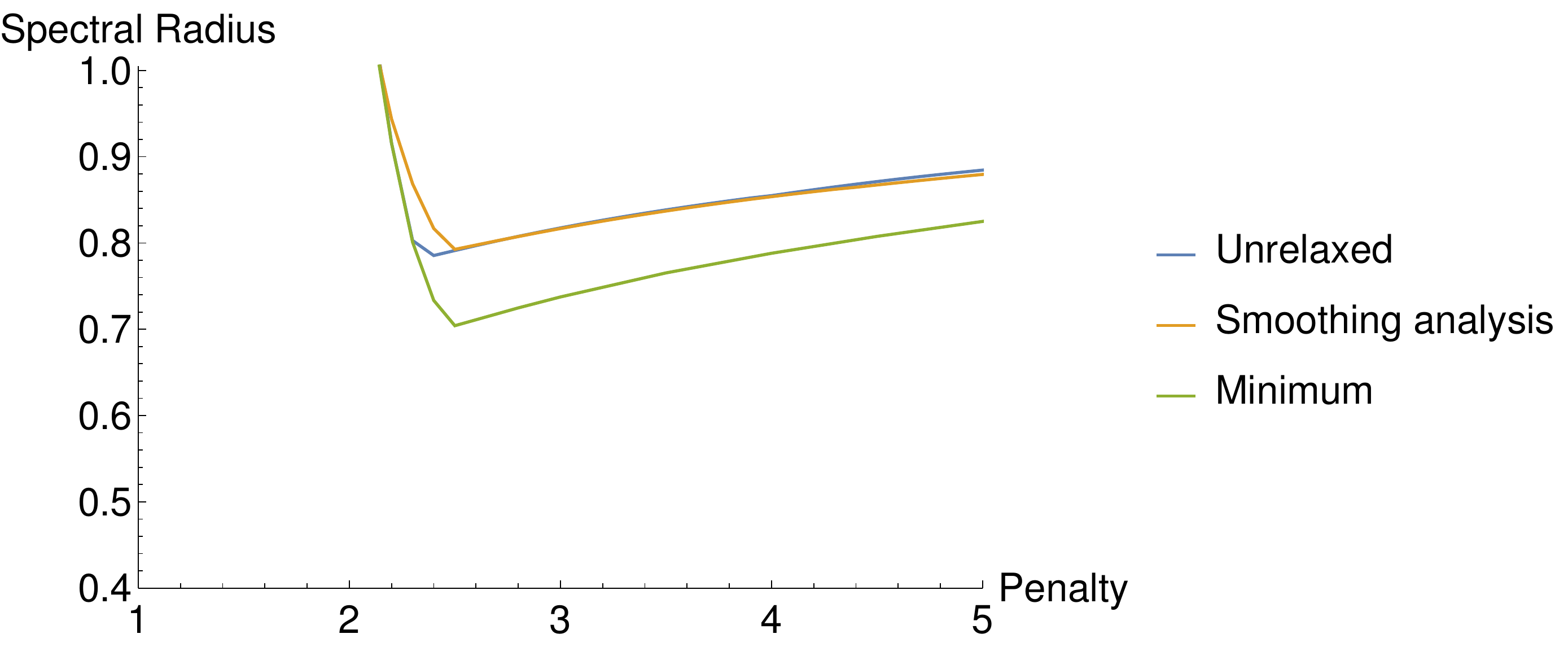}
  \caption{Left: circular domain and mesh used for the SIPG
    discretization of a Poisson problem. Right: spectral radius of the
    iteration operator as a function of the penalty parameter in SIPG
    using a \emph{cell} block-Jacobi smoother, without damping
    (Unrelaxed), with optimized damping from a 1D smoothing optimization
    alone (Smoothing analysis), and the numerically optimized two level
    process (Minimum).}
  \label{fig:diskmesh}
\end{figure}
for a Poisson problem on a disk discretized with SIPG on an irregular
mesh. We see that the best damping parameter depends on the
penalization parameter in SIPG, and can not be well predicted by a
smoothing analysis alone. Our goal here is to optimize the entire two
level process for such SIPG discretizations, both for Poisson and
singularly perturbed problems.

We apply Local Fourier Analysis (LFA), which has been widely used for
optimizing multigrid methods since its introduction in
\cite{Brandt1977}. This tool allows obtaining quantitative estimates
of the asymptotic convergence of numerical algorithms, and is
particularly useful for multilevel ones. Based on the Fourier
transform, the traditional LFA method is accurate for partial
differential equations if the influence of boundary conditions is
limited. It is well known \cite{Brandt1994}, that the method is exact
when periodic boundary conditions are used.

Previous Fourier analyses of such two-level methods for DG
discretizations have been performed for the Poisson equation by Hemker
et. al. (see \cite{Hemker2004,Hemker2003} and references therein), who
obtained numerically optimized parameters for \emph{point}
block-Jacobi smoothers. Our main results are first, explicit formulas
for the relaxation parameters of both \emph{point} and \emph{cell}
block-Jacobi smoothers for the Poisson equation and second, the
extension to the reaction-diffusion case, where we derive very
accurate closed form approximations of the optimal relaxation
parameters for the two-level process. Using our analytical results, we
can prove that for DG penalization parameter values used in practice,
it is better to use \emph{cell} block-Jacobi smoothers of Schwarz
type, in contrast to earlier results that suggested to use
\emph{point} block-Jacobi smoothers, based on a smoothing analysis
alone. Furthermore, our analysis reveals that there is an optimal
choice for the SIPG penalization parameter to get the fastest possible
two-level iterative solver. A further important contribution in our
opinion is that we present our LFA analysis using linear algebra tools
and matrices to make this important technique more accessible in the
linear algebra community.

\section{Model problem}\label{sec:model}

We consider the reaction-diffusion model problem
\begin{equation} \label{eqn:ContProb} 
  -\Delta u +\frac{1}{\eps} u = f \quad \text{in $\Omega$},\qquad
  u=0\quad\text{on $\partial\Omega$,}
\end{equation} 
where $\Omega \subset \mathbb{R}^{1,2,3}$ is a convex domain, $f$ is a
known source function and $\eps \in (0,\infty)$ is a parameter,
defining the relative size of the reaction term.

Using the $L^2(\Omega)$ space and the standard Sobolev space with zero
Dirichlet boundary conditions $H_0^1(\Omega)$, provided with their
respective inner products and norms, the weak form of problem
\eqref{eqn:ContProb} is: find $u \in H_0^1(\Omega)$ such that
\begin{equation}
  a(u,v) = (f,v)_{L^2(\Omega)},
\end{equation}
where $f \in L^2(\Omega)$ and the continuous bilinear form
$a(\cdot,\cdot)$ is defined by
\begin{equation}
  a(u,v) \coloneqq \int_\Omega \nabla u \cdot \nabla v dx +
  \frac{1}{\varepsilon} \int_\Omega u v dx =
  \left(u,v\right)_{H_0^1(\Omega)} + \frac{1}{\varepsilon}
  \left(u,v\right)_{L^2(\Omega)}.
\end{equation}
The bilinear form $a(u,v)$ is continuous and $H_0^1$-coercive
relatively to $L^2$ (see \cite[\S2.6]{DautrayLions1984}), i.e. there
exist constants $\gamma_a,C_a > 0$ such that
\begin{equation} 
  a(u,u) \ge \gamma_a \|u\|_{H_0^1(\Omega)}^2, \quad a(u,v) \le
  C_a\|u\|_{H_0^1(\Omega)}\|v\|_{H_0^1(\Omega)}.
\end{equation}
Note that even though $\gamma_a$ is independent of $\varepsilon$,
$C_a$ is not, which motivates our search for robust two-level methods
in the next section. From Lax-Milgram's theorem, the variational
problem admits a unique solution in $H_0^1(\Omega)$.

\subsection{Discretization}

We discretize the domain $\Omega$ using quadrilaterals or hexahedra,
constituting a mesh $\mesh$ with cells $\cell \in \mesh$ and faces
$\face \in \meshfaces$ using an SIPG finite element
discretization. Let $\mathbb{Q}_p(\cell)$ be the space of tensor
product polynomials with degree up to $p$ in each coordinate direction
with support in $\cell$. The discontinuous function space $V_h$ is
then defined as
\begin{equation} 
  V_h \coloneqq \bigl\{ v\in L^2(\Omega) \big| \forall \kappa,
  v_{|\cell} \in \mathbb{Q}_p(\cell)\bigr\}.
\end{equation}
Following \cite{Arnold1982}, we introduce the jump and average
operators $\jump{u} \coloneqq u^+ - u^-$ and $\av{u} \coloneqq
\frac{u^- + u^+}{2}$ and obtain the SIPG bilinear form
\begin{align}\label{eqn:SIPDG}
  \begin{aligned} 
    \ipbf{u}{v} \coloneqq& \int_\mesh \nabla u \cdot \nabla v dx +
    \frac{1}{\eps} \int_\mesh u v dx \\
    &+ \int_\meshfaces \left( \jump{u} \avv{\frac{\partial v}{\partial
          n}} + \avv{\frac{\partial u}{\partial n}} \jump{v} \right) ds +
    \int_\meshfaces \delta \jump{u} \jump{v} ds,
  \end{aligned}
\end{align} 
where the boundary conditions have been imposed weakly (i.e. Nitsche
boundary conditions \cite{Nitsche1971}) and $\delta \in \mathbb{R}$ is
a parameter penalizing the discontinuities at the interfaces between
cells. On the boundary there is only a single value, and we set
the value that would be on the other side to zero. In order for the
discrete bilinear form to be coercive, we must choose $\delta =
\delta_0/h$, where $h$ is the diameter of the cells and $\delta_0 \in
[1,\infty)$ is sufficiently large (see
\cite{KarakashianCollins2017}). Coercivity and continuity are proved
in \cite{Arnold1982} for the Laplacian under the assumption that
$\delta_0$ is sufficiently large, and these estimates are still valid
in the presence of the reaction term, since it is positive definite.

For our analysis, we will focus on a one-dimensional
problem\footnote{\label{HemkerFootnote}This is motivated by the
seminal work of P. W. Hemker \cite{Hemker2004} who stated: \emph{``we
study the one-dimensional equation, since this can be considered as an
essential building block for the higher dimensional case where we use
tensor product polynomials''}. We test however our analytical results
also in higher dimensions and on meshes which are not tensor products,
see Subsection \ref{HDSec}.}, with equally spaced nodes and cells with
equal size, see Fig. \ref{fig:MeshAndBasisFunctions}
\begin{figure}
  \centering
  \begin{tikzpicture} 
    \draw[-] (5,0) -- (1,0); 
    \draw[-] (5,0) -- (9,0); 
    \draw[dashed,-] (1,0) -- (0,0); 
    \draw[dashed,->] (9,0) -- (10,0) node[right]{$x$}; 
    \draw[-] (1,2) -- (1,-0.1) node[below]{$u_{j-2}^- u_{j-2}^+$};
    \draw[-] (3,2) -- (3,-0.1) node[below]{$u_{j-1}^- u_{j-1}^+$}; 
    \draw[-] (5,2) -- (5,-0.1) node[below]{$u_{j}^-u_{j}^+$}; 
    \draw[-] (7,2) -- (7,-0.1) node[below]{$u_{j+1}^- u_{j+1}^+$}; 
    \draw[-] (9,2) -- (9,-0.1) node[below]{$u_{j+2}^- u_{j+2}^+$};
    \draw[-] (1,0) -- (3,2) node[right]{$\phi_{j-1}$}; 
    \draw[-] (3,0) -- (1,2) node[right]{$\phi_{j-2}$}; 
    \draw[-] (3,0) -- (5,2) node[right]{$\phi_{j}$}; 
    \draw[-] (5,0) -- (3,2) node[left]{$\psi_{j-2}$};
    \draw[-] (5,0) -- (7,2) node[right]{$\phi_{j+1}$}; 
    \draw[-] (7,0) -- (5,2) node[left]{$\psi_{j - 1}$}; 
    \draw[-] (7,0) -- (9,2) node[left]{$\psi_{j+1}$}; 
    \draw[-] (9,0) -- (7,2) node[left]{$\psi_{j}$};
    \draw[-] (1.5,0.15) node[right]{$\kappa_{j-2}$}; 
    \draw[-] (3.5,0.15) node[right]{$\kappa_{j-1}$}; 
    \draw[-] (5.5,0.15) node[right]{$\kappa_{j}$}; 
    \draw[-] (7.5,0.15) node[right]{$\kappa_{j+1}$}; 
  \end{tikzpicture}
  \caption{Mesh for the discretization and finite element
    functions.}\label{fig:MeshAndBasisFunctions}
\end{figure}
for the mesh and finite element functions. We use the same kind of
basis and test functions and we denote them by $\phi_{j}=\phi_{j}(x)$
and $\psi_{j} = \psi_{j}(x)$ for decreasing and increasing linear
functions, respectively, with support in only one cell. The
coefficients accompanying each basis function are $u_j^+,u_j^- \in
\mathbb{R}$, where the superscript indicates if the nodal value is
evaluated from the left of the node ($^-$) or from the right ($^+$).

Any $v \in V_h$ can then be written as a linear combination of
$\phi_j(x)$ and $\psi_j(x)$,
\begin{align*}
  v &= \sum_{j \in J} u^+_j \phi_j(x) + u^-_j \psi_j(x) =
      \boldsymbol{u} \cdot \boldsymbol{\xi}^\intercal(x), \\
  \boldsymbol{u} &\coloneqq \left(\dots, u^+_{j-1}, u^-_{j-1},
                   u^+_{j}, u^-_{j}, u^+_{j+1}, u^-_{j+1}, \dots \right) \in
                   \mathbb{R}^{2J},\\
  \boldsymbol{\xi}(x) &\coloneqq \left( \dots, \phi_{j-1}(x),
                        \psi_{j-1}(x), \phi_{j}(x), \psi_{j}(x), \phi_{j+1}(x),\psi_{j+1}(x),
                        \dots \right),
\end{align*}
with $\phi_j(x), \psi_j(x) \in \mathbb{Q}_1(\kappa_j), j \in (1,J)$.
With this ordering, the SIPG discretization operator is
\newcommand{\edots}{\rotatebox{18}{$\ddots$}}
\renewcommand{\arraystretch}{2.}
\begin{align}\label{fig:DiscMat} A = \scalemath{0.75}{
  \left(
  \begin{array}{cccccccc} 
    \edots                        & \edots                        & \ipbf{\psi_{j-2}}{\psi_{j-1}} &                             &                             &                             \\
    \edots                        & \edots                        & \ipbf{\phi_{j-1}}{\psi_{j-1}} & \ipbf{\phi_{j-1}}{\phi_{j}} &                             &                             \\
    \ipbf{\psi_{j-1}}{\psi_{j-2}} & \ipbf{\psi_{j-1}}{\phi_{j-1}} & \ipbf{\psi_{j-1}}{\psi_{j-1}} & \ipbf{\psi_{j-1}}{\phi_{j}} & \ipbf{\psi_{j-1}}{\psi_{j}} &                             \\
                                  & \ipbf{\phi_{j}}{\phi_{j-1}}   & \ipbf{\phi_{j}}{\psi_{j-1}}   & \ipbf{\phi_{j}}{\phi_{j}}   & \ipbf{\phi_{j}}{\psi_{j}}   & \ipbf{\phi_{j}}{\phi_{j+1}} \\
                                  &                               &\ipbf{\psi_{j}}{\psi_{j-1}}    & \ipbf{\psi_{j}}{\phi_{j}}   & \edots                      & \edots \\
                                  &                               &                               & \ipbf{\phi_{j+1}}{\phi_{j}} & \edots                      & \edots
  \end{array}
                                                        \right)},
\end{align}
where the blank elements are zero. Using equation \eqref{eqn:SIPDG},
evaluating \eqref{fig:DiscMat} leads to
\renewcommand{\arraystretch}{1.7}
\begin{align} \label{eqn:DiscProb}
  A \boldsymbol{u} = 
  \frac1{h^2} \left(
  \begin{matrix}
    \ddots   & \ddots             & -\frac12                      &                              &                   &           \\ 
    \ddots   & \ddots             & \frac{h^2}{6 \eps}            &                   -\frac12   &                   &           \\[2mm]
    -\frac12 & \frac{h^2}{6 \eps} & \delta_0 + \frac{h^2}{3 \eps} &                   1-\delta_0 & - \frac12         &           \\[2mm]
             & - \frac12          & 1-\delta_0                    & \delta_0 +\frac{h^2}{3 \eps} & \frac{h^2}{6\eps} & - \frac12 \\
             &                    &                      -\frac12 &           \frac{h^2}{6 \eps} &            \ddots &    \ddots \\ 
             &                    &                               &                     -\frac12 &            \ddots &    \ddots
  \end{matrix}\right) \scalemath{0.85}{\left(\begin{matrix} 
    \vdots\\[2.75mm] 
    u_{j-1}^+\\[2.75mm] 
    u_{j}^- \\[2.75mm] 
    u_{j}^+ \\[2.75mm] 
    u_{j+1}^- \\[1mm] 
    \vdots
  \end{matrix}\right)} = \scalemath{0.85}{\left(\begin{matrix} 
    \vdots \\[2.75mm] 
    f_{j-1}^+\\[2.75mm] 
    f_{j}^- \\[2.75mm] 
    f_{j}^+ \\[2.75mm] 
    f_{j+1}^- \\[1mm] 
    \vdots
  \end{matrix}\right)} \eqqcolon \boldsymbol{f},
\end{align}
where \[\boldsymbol{f} = \left(\dots, f_{j-1}^+, f_{j}^-, f_{j}^+,
f_{j+1}^-, \dots\right) \in \mathbb{R}^{2J}\] is a vector, analogous
to $\boldsymbol{u}$, containing the coefficients of the representation
of the right hand side in $V_h$.  In the next section, we describe an
iterative two-level solver for the linear system \eqref{eqn:DiscProb}.

\section{Solver}

We solve the linear system \eqref{eqn:DiscProb} with a stationary
iteration of the form
\begin{equation}\label{eqn:richardson}
  \boldsymbol{u}^{(i+1)} = \boldsymbol{u}^{(i)} + M^{-1} \left(
    \boldsymbol{f} - A \boldsymbol{u}^{(i)} \right),
\end{equation}
where $M^{-1}$ is a two-level preconditioner, using first a cell-wise
nonoverlapping Schwarz (\emph{cell} block-Jacobi) smoother $D_c^{-1}$
(see \cite{FengKarakashian2001,Dryja2016}), i.e.
\renewcommand{\edots}{\rotatebox{10}{$\ddots$}}
\begin{align}\label{eqn:cellBJ}
  D_c^{-1} \boldsymbol{u} \coloneqq h^2 
  \left(\begin{array}{cccc}
          \edots &                              &                               &       \\
                 &\delta_0 + \frac{h^2}{3 \eps} & \frac{h^2}{6 \eps}            &       \\
                 &\frac{h^2}{6 \eps}            & \delta_0 + \frac{h^2}{3 \eps} &       \\
                 &                              &                               & \edots 
        \end{array}\right)^{-1}
  \left(\begin{array}{c}
          \vdots  \\
          u_{j}^+ \\
          u_{j}^- \\
          \vdots
        \end{array}\right).
\end{align}
This smoother takes only into account the relation between degrees of
freedom that are contained in each cell ($x_j^+$ and $x_j^-$ in
Fig. \ref{fig:Mesh}),
\begin{figure}
  \centering
  \begin{tikzpicture}[scale=1.]
    \draw[-] (0.0,0.0) -- (10.5, 0.0);
    \draw[-] (0.0,0.1) -- (0.0,-0.1) node[below]{$x_{1}^+ = 0$} ;
    \draw[-] (1.5,0.1) -- (1.5,-0.1) node[below]{$x_{1}^- x_{2}^+$};
    \draw[-] (3.0,0.1) -- (3.0,-0.1) node[below]{$\dots$};
    \draw[-] (4.5,0.1) -- (4.5,-0.1) node[below]{$x_{j-1}^-x_{j}^+$};
    \draw[-] (6.0,0.1) -- (6.0,-0.1) node[below]{$x_{j}^- x_{j+1}^+$};
    \draw[-] (7.5,0.1) -- (7.5,-0.1) node[below]{$\dots$};
    \draw[-] (9.0,0.1) -- (9.0,-0.1) node[below]{$x_{J-1}^- x_{J}^+$};
    \draw[-] (10.5,0.1) -- (10.5,-0.1) node[below]{$x_{J}^- = 1$};
  \end{tikzpicture}
  \caption{Mesh.}\label{fig:Mesh}
\end{figure}
i.e. we solve a local discrete reaction-diffusion problem consisting
of one cell, like a domain decomposition method with subdomains formed
by the cells.

Following \cite{Hemker2003}, we consider as well a \emph{point}
block-Jacobi smoother, consisting of a \emph{shifted} block
definition, i.e.
\renewcommand{\edots}{\rotatebox{10}{$\ddots$}}
\begin{align}
  D_p^{-1} \boldsymbol{u} \coloneqq h^2 
  \left(
  \begin{array}{cccc}\label{eqn:pointBJ}
    \edots &                              &                               &      \\
           &\delta_0 + \frac{h^2}{3 \eps} &1-\delta_0                     &      \\
           &1-\delta_0                    & \delta_0 + \frac{h^2}{3 \eps} &      \\
           &                              &                               &\edots
  \end{array}\right)^{-1}
  \left(
  \begin{array}{c}
    \vdots    \\
    u_{j}^-   \\
    u_{j+1}^+ \\
    \vdots
  \end{array}
  \right).
\end{align}
In this case, the smoother takes into account the relation between
degrees of freedom associated to a node ($x_j^-$ and $x_{j+1}^+$ in
Fig. \ref{fig:Mesh}). The domain decomposition interpretation in this
case is less clear than for $D_c$.

Let the restriction operator be defined as
\renewcommand{\edots}{\rotatebox{9}{$\ddots$}} 
\begin{align*}
  R \coloneqq \frac12 \left(
  \begin{array}{cccccccc}
    1    &   1/2   &   1/2   &         &         &         &         &         \\
         &   1/2   &   1/2   &    1    &         &         &         &         \\
         &         &         &         &  \edots &  \edots &  \edots &         \\
         &         &         &         &         &  \edots &  \edots &  \edots 
  \end{array}
  \right),
\end{align*}
and the prolongation operator be $P \coloneqq 2 R^\intercal$ (linear
interpolation), and set $A_0 := R A P$. Then the two-level
preconditioner $M^{-1}$, with one presmoothing step and a relaxation
parameter $\alpha$, acting on a residual $g$ is defined by Algorithm
\ref{alg:precit}.
\begin{algorithm}
  \caption{Two-level non-overlapping Schwarz preconditioned iteration.}
  \label{alg:precit}
  \begin{algorithmic}[1]
    \STATE compute $\boldsymbol{x}:= \alpha D^{-1} \boldsymbol{g}$,
    \STATE compute $\boldsymbol{y}:= \boldsymbol{x} + P A_0^{-1} R
    (\boldsymbol{g} - A \boldsymbol{x})$,
    \STATE obtain $M^{-1}\boldsymbol{g} = \boldsymbol{y}$.
  \end{algorithmic}
\end{algorithm}

\section{Local Fourier Analysis (LFA)}

In order to make the important LFA more accessible to the linear
algebra community, we work directly with matrices instead of
symbols. We consider a mesh as shown in Fig. \ref{fig:Mesh}, and
assume for simplicity that it contains an even number of
elements. Given that we are using nodal finite elements, a function $w
\in V_h$ is uniquely determined by its values at the nodes,
$\boldsymbol{w} = \left(\dots, w_{j-1}^+, w_{j}^-, w_{j}^+, w_{j+1}^-,
\dots \right)$. For the local Fourier analysis (LFA), we can picture
continuous functions that take the nodal values at the nodal points,
and since in the DG discretization there are two values at each node,
we consider two continuous functions, $w^+(x)$ and $w^-(x)$, which
interpolate the nodal values of $w$ to the left and right of the
nodes, respectively. We next represent these two continuous functions
as combinations of Fourier modes to get an understanding of how they
are transformed by the two grid iteration.

\subsection{LFA tools}\label{sec:lfatools}

For a uniform mesh with mesh size $h$, and assuming periodicity, we
can expand $w^-(x)$ and $w^+(x)$ into a finite Fourier series,
\begin{align*}
  w^+(x)=&\sum_{\widetilde{k}=-(J/2-1)}^{J/2} c_{\widetilde{k}}^{+}
           \e{\I 2 \pi \widetilde{k} x} = \sum_{k=1}^{J/2} c_{k-J/2}^{+} \e{\I 2
           \pi (k-J/2) x} + c_k^{+} \e{\I 2 \pi k x},\\
  w^-(x)=&\sum_{\widetilde{k}=-(J/2-1)}^{J/2} c_{\widetilde{k}}^{-} \e{\I 2 \pi
           \widetilde{k} x} = \sum_{k=1}^{J/2} c_{k-J/2}^{-} \e{\I 2 \pi (k-J/2)
           x} + c_k^{-} \e{\I 2 \pi k x}.
\end{align*}
Enforcing the interpolation condition for these trigonometric
polynomials at the nodes, $w_j^+ \coloneqq w^+(x_j^+)$ and $w_j^-
\coloneqq w^-(x_j^-)$, we obtain
\begin{align*}
  w^+_j =& \scalemath{0.9}{\sum_{k=1}^{J/2} c_{k-J/2}^+ \e{\I 2 \pi
           (k-J/2) x^+_j} + c_k^+ \e{\I 2 \pi k x^+_j} = \sum_{k=1}^{J/2}
           c_{k-J/2}^+ \e{\I 2 \pi (k-J/2) (j-1) h} + c_k^+ \e{\I 2 \pi k (j-1)
           h}}, \\
  w^-_j =& \scalemath{0.9}{\sum_{k=1}^{J/2} c_{k-J/2}^- \e{\I 2 \pi
           (k-J/2) x^-_j} + c_k^- \e{\I 2 \pi k x^-_j} = \sum_{k=1}^{J/2}
           c_{k-J/2}^- \e{\I 2 \pi (k-J/2) j h} + c_k^- \e{\I 2 \pi k j h}}.
\end{align*}
The representation for $\boldsymbol{w}^+$ and $\boldsymbol{w}^-$ as a
set of nodal values can therefore be written as
\begin{align*}
  \boldsymbol{w}^+ =& \left(
  \begin{smallmatrix} 
    w_1^+ \\
    \vdots \\
    w_j^+ \\
    \vdots \\
    w_{J}^+
  \end{smallmatrix} 
  \right) = 
  \scalemath{0.75}{
  \left(
  \begin{smallmatrix} 
    \displaystyle\sum_{k=1}^{J/2} c_{k-J/2}^+ + c_k^+ \\
    \vdots \\
    \displaystyle\sum_{k=1}^{J/2} c_{k-J/2}^+ \e{\I 2 \pi (k-J/2)
      (j-1) h} + c_k^+ \e{\I 2 \pi k (j-1) h} \\
    \vdots \\
    \displaystyle\sum_{k=1}^{J/2} c_{k-J/2}^+ \e{\I 2 \pi (k-J/2)
      (J-1) h} + c_k^+ \e{\I 2 \pi k (J-1) h}
  \end{smallmatrix} 
  \right)}, \\
\boldsymbol{w}^- =& \left(
  \begin{smallmatrix} 
    w_1^- \\
    \vdots \\
    w_j^- \\
    \vdots \\
    w_J^-
  \end{smallmatrix} 
  \right) = 
  \scalemath{0.75}{
  \left(
  \begin{smallmatrix} 
    \displaystyle\sum_{k=1}^{J/2} c_{k-J/2}^- \e{\I 2 \pi (k-J/2) h} +
    c_k^- \e{\I 2 \pi k h} \\
    \vdots \\
    \displaystyle\sum_{k=1}^{J/2} c_{k-J/2}^- \e{\I 2 \pi (k-J/2) j h}
    + c_k^- \e{\I 2 \pi k j h} \\
    \vdots \\
    \displaystyle\sum_{k=1}^{J/2} c_{k-J/2}^- \e{\I 2 \pi (k-J/2) J h}
    + c_k^- \e{\I 2 \pi k J h}
  \end{smallmatrix} 
  \right)}. 
\end{align*}
We thus write the Fourier representation as a matrix-vector
product and define two matrices $Q^+$ and $Q^-$, such that
$\boldsymbol{w}^+ = Q^+ \boldsymbol{c}^+$ and $\boldsymbol{w}^- = Q^-
\boldsymbol{c}^-$, where
\begin{align*}
  &Q^+ :=\scalemath{0.98}{\left(
  \begin{smallmatrix} 
    1                &        1             & \dots &              1              &          1             & \dots &   1   &           1                \\
    \vdots               &      \vdots          &\ddots &         \vdots              &    \vdots              &\ddots &\vdots &\vdots                      \\
    \e{- \I 2 \pi (1-J/2) (j-1) h} & \e{\I 2 \pi (j-1) h} & \dots & \e{\I 2 \pi (k-J/2) (j-1) h}& \e{\I 2 \pi k (j-1) h} & \dots &   1   & \e{\I 2 \pi (J/2) (j-1) h} \\
    \vdots               &      \vdots          &\ddots &         \vdots              &    \vdots              &\ddots &\vdots &\vdots                      \\
    \e{- \I 2 \pi (1-J/2) (J-1) h} & \e{\I 2 \pi (J-1) h} & \dots & \e{\I 2 \pi (k-J/2) (J-1) h}& \e{\I 2 \pi k (J-1) h} & \dots &   1   & \e{\I 2 \pi (J/2) (J-1) h} 
  \end{smallmatrix} 
  \right)},\\
  &Q^- := \scalemath{0.965}{\left(
  \begin{matrix} 
    \e{\I 2 \pi (1-J/2) h}   & \e{\I 2 \pi h}   & \dots & \e{\I 2 \pi (k-J/2) h}   & \e{\I 2 \pi k h}   & \dots &   1   & \e{\I 2 \pi (J/2) h}   \\
    \vdots         &      \vdots      &\ddots &         \vdots           &    \vdots          &\ddots &\vdots &\vdots   \\
    \e{\I 2 \pi (1-J/2) j h} & \e{\I 2 \pi j h} & \dots & \e{\I 2 \pi (k-J/2) j h} & \e{\I 2 \pi k j h} & \dots &   1   & \e{\I 2 \pi (J/2) j h} \\
    \vdots         &      \vdots      &\ddots &         \vdots           &    \vdots          &\ddots &\vdots &\vdots  \\
    \e{\I 2 \pi (1-J/2) J h} & \e{\I 2 \pi J h} & \dots & \e{\I 2 \pi (k-J/2) J h} & \e{\I 2 \pi k J h} & \dots &   1   & \e{\I 2 \pi (J/2) J h} \\
  \end{matrix} 
  \right)},
\end{align*}
and
\begin{align*}
  &\boldsymbol{c}^+ := \left( 
  \begin{matrix} 
    c^+_{1-J/2} & 
    c^+_1 & 
    \dots &
    c^+_{k-J/2} & 
    c^+_k & 
    \dots & 
    c^+_0 & 
    c^+_{J/2} 
  \end{matrix} 
            \right)^\intercal, \\
  &\boldsymbol{c}^- := \left( 
  \begin{matrix} 
    c^-_{1-J/2} & 
    c^-_1 & 
    \dots &
    c^-_{k-J/2} & 
    c^-_k & 
    \dots & 
    c^-_0 & 
    c^-_{J/2} 
  \end{matrix} 
            \right)^\intercal.
\end{align*}
An element in $V_h$ can then be represented by its nodal elements in a
stacked vector
\begin{align*}
  \check{\boldsymbol{w}} =  \left(\begin{smallmatrix}
      \boldsymbol{w}^+ \\ 
      \boldsymbol{w}^- \end{smallmatrix}\right)=
  \left(
  \begin{smallmatrix}
    Q^+ &     \\
    & Q^-
  \end{smallmatrix}
      \right)
      \left(
      \begin{smallmatrix}
        \boldsymbol{c}^+ \\
        \boldsymbol{c}^-
      \end{smallmatrix}
  \right)
  =: \check{Q} \check{\boldsymbol{c}}.
\end{align*}
We now reorder the vectors $\check{\boldsymbol{w}}$ and
$\check{\boldsymbol{c}}$ to obtain the new vectors $\boldsymbol{w}$
and $\boldsymbol{c}$ such that their elements are ordered from left to
right with respect to the mesh. To do so, we define an orthogonal
matrix $S$, such that $\boldsymbol{w} = S^\intercal
\check{\boldsymbol{w}}$ and $\check{\boldsymbol{w}} =
S\boldsymbol{w}$,
\renewcommand{\arraystretch}{1}
\begin{align*}
  S^\intercal \coloneqq \scalemath{0.7}{\left(
  \begin{array}{ccccc:cccc}
    1 &     &      &     &     &     &     &      &     \\
      &     &      &     &     &  1  &     &      &     \\
      &  1  &      &     &     &     &     &      &     \\
      &     &      &     &     &     &  1  &      &     \\
      &     &\cdots&     &     &     &     &      &     \\
      &     &      &     &     &     &     &\cdots&     \\
      &     &      &     &     &     &     &      &     \\
      &     &      &     &     &     &     &      &     \\
  \end{array}
  \right)},
\end{align*}
where the dashed line is drawn between the two columns in the middle
of the matrix. Finally, we define the reordered and scaled matrix $Q$
\begin{align*}
  \boldsymbol{w} = S^\intercal \check{Q} S \boldsymbol{c} =:
  \left(\sqrt{h}\right)^{-1} Q \boldsymbol{c}.
\end{align*}
The structure of Q is
\begin{align}\label{eqn:Q}
  Q = \sqrt{h} \scalemath{0.78}{\left(
  \begin{matrix} 
    \cdots &        &           \cdots           &                           &         \cdots        &                       & \cdots &        \\
           & \cdots &                            &           \cdots          &                       &         \cdots        &        & \cdots \\
    \cdots &        & \e{\I 2 \pi (k-J/2)(j-2)h} &                           & \e{\I 2 \pi k (j-2)h} &                       & \cdots &        \\
           & \cdots &                            &\e{\I 2 \pi (k-J/2)(j-1)h} &                       & \e{\I 2 \pi k (j-1)h} &        & \cdots \\
    \cdots &        & \e{\I 2 \pi (k-J/2)(j-1)h} &                           & \e{\I 2 \pi k (j-1)h} &                       & \cdots &        \\
           & \cdots &                            &\e{\I 2 \pi (k-J/2) j h}   &                       & \e{\I 2 \pi k j h}    &        & \cdots \\
    \cdots &        & \e{\I 2 \pi (k-J/2) j h}   &                           & \e{\I 2 \pi k j h}    &                       & \cdots &        \\
           & \cdots &                            &\e{\I 2 \pi (k-J/2)(j+1)h} &                       & \e{\I 2 \pi k (j+1)h} &        & \cdots \\
    \cdots &        & \e{\I 2 \pi (k-J/2)(j+1)h} &                           & \e{\I 2 \pi k (j+1)h} &                       & \cdots &        \\
           & \cdots &                            &\e{\I 2 \pi (k-J/2)(j+2)h} &                       & \e{\I 2 \pi k (j+2)h} &        & \cdots \\
    \cdots &        &           \cdots           &                           &         \cdots        &                       & \cdots &        \\
           & \cdots &                            &           \cdots          &                       &         \cdots        &        & \cdots 
  \end{matrix}
  \right)},
\end{align}
where the factor $\sqrt{h}$ is inserted such that $Q$ is unitary
(i.e. $Q^*=Q^{-1}$).

If we follow the same procedure for a coarser mesh, created by joining
neighboring cells together, the matrix $Q_0$, analogous to $Q$, picks
up the elements corresponding to the nodes contained in both the
coarse and fine meshes,
\begin{align*}
  Q_0 = \sqrt{2h} \left(
  \begin{matrix} 
    \cdots &        &           \cdots           &                       & \cdots &        \\
           & \cdots &                            &         \cdots        &        & \cdots \\
    \cdots &        & \e{\I 2 \pi (k-J/2)(j-2)h} &                       & \cdots &        \\
           & \cdots &                            & \e{\I 2 \pi k j h}    &        & \cdots \\
    \cdots &        & \e{\I 2 \pi (k-J/2) j h}   &                       & \cdots &        \\
           & \cdots &                            & \e{\I 2 \pi k (j+2)h} &        & \cdots \\
    \cdots &        &           \cdots           &                       & \cdots &        \\
           & \cdots &                            &         \cdots        &        & \cdots 
   \end{matrix}\right),
\end{align*}
where $j \ge 2$ is even and the factor $\sqrt{2h}$ is inserted such
that $Q_0$ is unitary. We next show that $Q$ renders $A$ and $D$ block
diagonal and $Q_0$ and $Q$ do the same for $R$ and $P$, albeit with
rectangular blocks. Therefore the study of the two grid iteration
operator is reduced to the study of a generic block. In order to prove
this result we need the following lemma.
\begin{lemma} \label{lem:blockdiagonal}
  Let $C \in \mathbb{R}^{2J \times 2J}$ be a block circulant matrix of
  the form
  \begin{equation*}
    C=\left(\begin{smallmatrix}
        C_{0}  & C_{1}  & C_{2}  & \dots  &   0    & \dots  & C_{-2} & C_{-1} \\
        C_{-1} & C_{0}  & C_{1}  & C_{2}  & \dots  &   0    & \dots  & C_{-2} \\
        C_{-2} & C_{-1} & C_{0}  & C_{1}  & C_{2}  & \dots  &   0    & \dots  \\
        \dots  & C_{-2} & C_{-1} & C_{0}  & C_{1}  & C_{2}  & \dots  & \dots  \\
        0      & \dots  & C_{-2} & C_{-1} & C_{0}  & C_{1}  & C_{2}  & \dots  \\
        \dots  &   0    & \dots  & C_{-2} & C_{-1} & C_{0}  & C_{1}  & \dots  \\
        C_{2}  & \dots  &   0    & \dots  & C_{-2} & C_{-1} & C_{0}  & \dots  \\
        C_{1}  & C_{2}  & \dots  & \dots  & \dots  & \dots  & \dots  & \dots
      \end{smallmatrix}\right),
  \end{equation*}
  where $C_j$ represents $(2 \times 2)$-blocks, and let $Q \in
  \mathbb{R}^{2J \times 2J}$ be the matrix which columns are discrete
  \emph{grid functions} as defined in \eqref{eqn:Q}, then the matrix
  $M=Q^* C Q$ is $(2 \times 2)$-block diagonal.
\end{lemma}
\begin{proof}
  We compute the block $(p,q)$ of $M$ to be
  \begin{equation*}
    M_{p,q} = \sum_{k=-(J/2-1)}^{J/2-1} \sum_{l=1}^J Q^*_{l,p} C_{k}
    Q_{((k+l-1)\%J)+1,q},
  \end{equation*}
  where we denote by $\% J$ equivalency modulo $J$, and a block $(m,n)$
  of $Q$ is
  \begin{equation*}
    Q_{m,n} =
    \begin{cases}
      \left(\begin{matrix}
          \e{\I 2 \pi ((n + 1)/2 - J/2) (m-1) h} & 0 \\
          0 & \e{\I 2 \pi ((n + 1)/2 - J/2) m h}
        \end{matrix}\right), &\text{if $n$ is odd,} \\
      \left(\begin{matrix}
          \e{\I 2 \pi (n/2) (m-1) h} & 0 \\
          0 & \e{\I 2 \pi (n/2) m h}
        \end{matrix}\right) ,& \text{if $n$ is even.}
    \end{cases} 
  \end{equation*}
  As before, we will use for the small blocks the notation $C_k =
  \left(\begin{smallmatrix} c_{11} & c_{12} \\ c_{21} &
      c_{22} \end{smallmatrix} \right)$.  We consider an off-diagonal block,
  i.e. $Q_{p,q}$, with $p \ne q$, and take an arbitrary $k$. Then if
  $p$ and $q$ are even, we have
  \begin{multline*}
    \sum_{l=1}^J Q^*_{l,p} C_{k} Q_{((k+l-1)\%J)+1,q} \\
    = \sum_{l=1}^J \left(
      \begin{smallmatrix}
        c_{11} \e{\frac{\I (((k+l-1)\%J)+1-1) \pi  q}{J}-\frac{\I (l-1) p \pi }{J}} & c_{12} \e{\frac{\I (((k+l-1)\%J)+1) \pi  q}{J}-\frac{\I (l-1) p \pi }{J}} \\
        c_{21} \e{\frac{\I (((k+l-1)\%J)+1-1) \pi  q}{J}-\frac{\I l p \pi }{J}}     & c_{22} \e{\frac{\I (((k+l-1)\%J)+1) \pi  q}{J}-\frac{\I l p \pi }{J}}
      \end{smallmatrix}
    \right) \\
    = \left(
      \begin{smallmatrix}
        c_{11} \e{\frac{\I 2 \pi}{J} \left(\left(\frac12 (p+(k-1) q)\right) \% J\right)} & c_{12} \e{\frac{\I 2 \pi}{J} \left(\left(\frac12 (p+k q)\right) \% J \right)} \\
        c_{21} \e{\frac{\I 2 \pi}{J} \left(\left(\frac12 (k-1) q\right) \% J\right)}     & c_{22} \e{\frac{\I 2 \pi}{J} \left(\left(\frac12 k q \right) \% J \right)}
      \end{smallmatrix}
    \right) \sum_{l=1}^J \e{\frac{\I 2 \pi}{J}  \left(\frac12 (q - p) l \right) \% J} = 0,
  \end{multline*}
  since we identify the sum of the roots of unity.  If $p$ and $q$ are
  odd, we have
  \begin{multline*}
    \sum_{l=1}^J Q^*_{l,p} C_{k} Q_{((k+l-1)\%J)+1,q} \\
    = \sum_{l=1}^J \left( 
      \scalemath{0.8}{
        \begin{smallmatrix} 
          c_{11} \e{\frac{\I 2 ((k+l-1)\%J)+1-1) \pi  \left(\frac{q+1}{2}-\frac{J}{2}\right)}{J}-\frac{\I 2 (l-1) \left(\frac{p+1}{2}-\frac{J}{2}\right) \pi }{J}} & c_{12} \e{\frac{\I 2 ((k+l-1)\%J)+1) \pi  \left(\frac{q+1}{2}-\frac{J}{2}\right)}{J}-\frac{\I 2 (l-1) \left(\frac{p+1}{2}-\frac{J}{2}\right) \pi }{J}} \\
          c_{21} \e{\frac{\I 2 ((k+l-1)\%J)+1-1) \pi  \left(\frac{q+1}{2}-\frac{J}{2}\right)}{J}-\frac{\I 2 l \left(\frac{p+1}{2}-\frac{J}{2}\right) \pi }{J}} & c{22} \e{\frac{\I 2 ((k+l-1)\%J)+1) \pi \left(\frac{q+1}{2}-\frac{J}{2}\right)}{J}-\frac{\I 2 l \left(\frac{p+1}{2}-\frac{J}{2}\right) \pi }{J}} 
        \end{smallmatrix}}
    \right) \\
    = \left(
      \begin{smallmatrix}
        c_{11} \e{\frac{\I 2 \pi}{J} \left(\left(\frac12(k+p+(k-1) q) \right)\% J \right)} & c_{12} \e{\frac{\I 2 \pi}{J} \left(\left(\frac12 (p+k (q+1)+1)\right) \% J \right)} \\
        c_{21} \e{\frac{\I 2 \pi}{J} \left(\left(\frac12 (k-1)(q+1)\right) \% J \right)}   & c_{22} \e{\frac{\I 2 \pi}{J} \left(\left(\frac12 k(q+1)\right)\%J\right)}
      \end{smallmatrix}
    \right) \sum_{l=1}^J \e{\frac{\I 2 \pi}{J}  \left(\frac12 (q - p) l \right) \% J} = 0,
  \end{multline*}
  since again we identify the sum of the roots of unity. If $p$ is odd
  and $q$ is even, we get
  \begin{multline*}
    \sum_{l=1}^J Q^*_{l,p} C_{k} Q_{((k+l-1)\%J)+1,q} \\
    = \sum_{l=1}^J \left(
      \begin{smallmatrix}
        c_{11} \e{\frac{\I ((k+l-1)\%J)+1-1) \pi  q}{J}-\frac{\I 2 (l-1) \left(\frac{p+1}{2}-\frac{J}{2}\right) \pi }{J}} & c_{12} \e{\frac{\I ((k+l-1)\%J)+1) \pi  q}{J}-\frac{\I 2 (l-1) \left(\frac{p+1}{2}-\frac{J}{2}\right) \pi }{J}} \\
        c_{21} \e{\frac{\I ((k+l-1)\%J)+1-1) \pi  q}{J}-\frac{\I 2 l \left(\frac{p+1}{2}-\frac{J}{2}\right) \pi }{J}} & c_{22} e^{\frac{\I ((k+l-1)\%J)+1) \pi  q}{J}-\frac{\I 2 l \left(\frac{p+1}{2}-\frac{J}{2}\right) \pi }{J}} 
      \end{smallmatrix}
    \right) \\
    \left(
      \begin{smallmatrix}
        c_{11} \e{-\frac{i \pi  (J-p-k q+q-1)}{J}} & c_{12} e^{-\frac{i \pi  (J-p-k q-1)}{J}} \\
        c_{21} \e{\frac{i (k-1) \pi  q}{J}}        & c_{22} e^{\frac{i k \pi  q}{J}}
      \end{smallmatrix}
    \right) \sum_{l=1}^J \e{\frac{\I 2 \pi}{J}  \left(\frac12 (q - p - 1 + J) l \right) \% J} = 0.
  \end{multline*}

  If $p$ is even and $q$ is odd, we get similarly
  \begin{multline*}
    \sum_{l=1}^J Q^*_{l,p} C_{k} Q_{((k+l-1)\%J)+1,q} = \\
    = \sum_{l=1}^J \left(
      \begin{smallmatrix}
        c_{11} \e{\frac{\I 2 ((k+l-1)\%J)+1-1) \pi \left(\frac{q+1}{2}-\frac{J}{2}\right)}{J}-\frac{\I (l-1) p \pi }{J}} & c_{12} \e{\frac{\I 2 ((k+l-1)\%J)+1) \pi \left(\frac{q+1}{2}-\frac{J}{2}\right)}{J}-\frac{\I (l-1) p \pi }{J}} \\
        c_{21} \e{\frac{\I 2 ((k+l-1)\%J)+1-1) \pi \left(\frac{q+1}{2}-\frac{J}{2}\right)}{J}-\frac{\I l p \pi }{J}} & c_{22} \e{\frac{\I 2 ((k+l-1)\%J)+1) \pi \left(\frac{q+1}{2}-\frac{J}{2}\right)}{J}-\frac{\I l p \pi }{J}}
      \end{smallmatrix}
    \right) \\
    = \left(
      \begin{smallmatrix}
        c_{11} \e{-\frac{\I 2 \pi}{J} \left(\left(\frac12 (J (k-1)-p+q-k (q+1)+1) \right) \% J \right)} & c_{12} \e{\frac{\I 2 \pi}{J} \left(\left(\frac12 (p+k (-J+q+1))\right) \% J \right)} \\
        c_{21} \e{-\frac{\I 2 \pi}{J} \left(\left(\frac12 (k-1)(J-q-1)\right) \% J \right)}             & c_{22} \e{-\frac{\I 2 \pi}{J} \left(\left(\frac12 k (J-q-1)\right) \% J \right)}
      \end{smallmatrix}
    \right) \\
    \sum_{l=1}^J \e{\frac{\I 2 \pi}{J} \left(\frac12 (q - p + 1 - J) l \right) \% J} = 0,
  \end{multline*}
  and thus $M$ is a $(2 \times 2)$-block diagonal matrix.
\end{proof}

Given that Lemma \ref{lem:blockdiagonal} ensures $M$ is block
diagonal, a generic block with block index $p,q$ can be computed as
follows:
\begin{align*}
  M = Q^* C Q \Longleftrightarrow\ & Q M = C Q \Longleftrightarrow (Q M)_{p,q} = (C Q)_{p,q}, \hspace{0.5cm} \forall p,q \\
  \Longleftrightarrow\ &Q_{p,q} M_{q} = \sum_{k=-(J/2-1)}^{J/2-1} C_{k} Q_{((k+p-1)\%J)+1,q}, \hspace{0.5cm} \forall p,q \\
  \Longleftrightarrow\ &M_{q} = (Q^*)_{q,p} \sum_{k=-(J/2-1)}^{J/2-1} C_{k} Q_{((k+p-1)\%J)+1,q} , \hspace{0.5cm} \forall p,q\\
  \Longleftrightarrow\ &\widetilde{M} = \widetilde{Q}^* \widetilde{C} Q_r,
\end{align*}
where $\widetilde{C} Q_r = \sum_{k=-(J/2-1)}^{J/2-1} C_{k} Q_{((k+p-1)\%J)+1,q}$, 
\begin{align*}
  Q_r =& \sqrt{\frac12} \left(
  \begin{matrix} 
  \e{\I 2 \pi (k-J/2)(j-2)h} &                           & \e{\I 2 \pi k (j-2)h} &                       \\
                             &\e{\I 2 \pi (k-J/2)(j-1)h} &                       & \e{\I 2 \pi k (j-1)h} \\
  \e{\I 2 \pi (k-J/2)(j-1)h} &                           & \e{\I 2 \pi k (j-1)h} &                       \\
                             &\e{\I 2 \pi (k-J/2) j h}   &                       & \e{\I 2 \pi k j h}    \\
  \e{\I 2 \pi (k-J/2) j h}   &                           & \e{\I 2 \pi k j h}    &                       \\
                             &\e{\I 2 \pi (k-J/2)(j+1)h} &                       & \e{\I 2 \pi k (j+1)h} \\
  \e{\I 2 \pi (k-J/2)(j+1)h} &                           & \e{\I 2 \pi k (j+1)h} &                       \\
                             &\e{\I 2 \pi (k-J/2)(j+2)h} &                       & \e{\I 2 \pi k (j+2)h}
  \end{matrix}
  \right),\\
  Q_l =& \sqrt{\frac12} \scalemath{0.98}{\left(
  \begin{matrix} 
  \e{-\I 2 \pi (k-J/2)(j-1)h} &                            & \e{-\I 2 \pi (k-J/2) j h} &                               \\
                              &\e{-\I 2 \pi (k-J/2) j h}   &                           & \e{-\I 2 \pi (k-J/2) (j+1) h} \\
  \e{-\I 2 \pi k (j-1)h}      &                            & \e{-\I 2 \pi k j h}       &                               \\
                              &\e{-\I 2 \pi k j h}         &                           & \e{-\I 2 \pi k (j+1) h}      
  \end{matrix}
  \right)},
\end{align*}
and the factor $\sqrt{\frac12}$ is chosen such that $Q_l I_{4 \times
8} Q_r = I_{4 \times 4}$, where $I_{4 \times 4}$ is the $4 \times 4$
identity matrix and
\[
  I_{4 \times 8} = \begin{pmatrix}
    0 & 0 & 1 & 0 & 0 & 0 & 0 & 0 & \\
    0 & 0 & 0 & 1 & 0 & 0 & 0 & 0 & \\
    0 & 0 & 0 & 0 & 1 & 0 & 0 & 0 & \\
    0 & 0 & 0 & 0 & 0 & 1 & 0 & 0 \end{pmatrix}.
\]

We have computed a generic block $\widetilde{M}$ in the block diagonal
of $M$. In the next subsection, we will work with blocks of size $4$
by $4$, given that we use a coarse correction with coarse cells formed
from $2$ adjacent fine cells with $2$ degrees of freedom each.

\subsection{Analysis of the SIPG operator and associated smoothers}
We extract a submatrix $\widetilde{A}$ containing the degrees of
freedom of two adjacent cells from the SIPG operator defined in
\eqref{eqn:DiscProb},
\begin{align*}
  \widetilde{A} = \left(
  \begin{array}{cccccccc}
    -\frac12 & 1-\delta_0 & \delta_0 + \frac{h^2}{3 \eps} & \frac{h^2}{6\eps}           & - \frac12                     &                             &                   &           \\
             & - \frac12  & \frac{h^2}{6 \eps}            & \delta_0 +\frac{h^2}{3\eps} & 1-\delta_0                    & -\frac12                    &                   &           \\
             &            & -\frac12                      & 1-\delta_0                  & \delta_0 + \frac{h^2}{3 \eps} & \frac{h^2}{6 \eps}          & - \frac12         &           \\
             &            &                               & - \frac12                   & \frac{h^2}{6 \eps}            & \delta_0 +\frac{h^2}{3\eps} & 1-\delta_0        & - \frac12   
  \end{array}
  \right).
\end{align*}

We can now begin the block-diagonalization,
\begin{align}
\begin{aligned}
  \widehat{A} &= Q_l \widetilde{A} Q_r \\
  =& \frac1{h^2}
  \left(
    \scalemath{0.9}{\begin{smallmatrix} 
        \dd+\frac{h}{3\eps}+\cos\left(2 \pi (k-J/2) h \right) & 1-\dd+\frac{h^2}{6\eps}\e{\I 2 \pi (k-J/2) h}           &                                                   &                                                   \\
        1-\dd+\frac{h^2}{6\eps}\e{-\I 2 \pi (k-J/2) h}        & \dd+\frac{h^2}{3\eps}+\cos\left(2 \pi (k-J/2) h \right) &                                                   &                                                   \\
                                                              &                                                         & \dd+\frac{h^2}{3\eps}-\cos\left(2 \pi k h \right) & 1-\dd+\frac{h^2}{6\eps}\e{\I 2 \pi k h}           \\
                                                              &                                                         & 1-\dd+\frac{h^2}{6\eps}\e{-\I 2 \pi k h}          & \dd+\frac{h^2}{3\eps}-\cos\left(2 \pi k h \right)                           
  \end{smallmatrix}}
  \right).
\end{aligned}
\end{align}
The same mechanism can be applied to the smoothers
\begin{align}
&\begin{aligned}
  \widetilde{D}_c = \left(
  \begin{array}{cccccccc}
       0     &      0     & \delta_0 + \frac{h^2}{3 \eps} & \frac{h^2}{6\eps}           &              0                &                             &                   &           \\
             &      0     & \frac{h^2}{6 \eps}            & \delta_0 +\frac{h^2}{3\eps} &              0                &              0              &                   &           \\
             &            &               0               &              0              & \delta_0 + \frac{h^2}{3 \eps} & \frac{h^2}{6 \eps}          &         0         &           \\
             &            &                               &              0              & \frac{h^2}{6 \eps}            & \delta_0 +\frac{h^2}{3\eps} &         0         &      0      
  \end{array}
  \right),
\end{aligned}\\
&\begin{aligned}
  \widehat{D}_c = Q_l \widetilde{D}_c Q_r = \frac1{h^2}
  \left(
  \begin{smallmatrix} 
    \dd+\frac{h}{3\eps}                      & \frac{h^2}{6\eps}\e{\I 2 \pi (k-J/2) h} &                                    &                                   \\
    \frac{h^2}{6\eps}\e{-\I 2 \pi (k-J/2) h} & \dd+\frac{h^2}{3\eps}                   &                                    &                                   \\
                                             &                                         & \dd+\frac{h^2}{3\eps}              & \frac{h^2}{6\eps}\e{\I 2 \pi k h} \\
                                             &                                         & \frac{h^2}{6\eps}\e{-\I 2 \pi k h} & \dd+\frac{h^2}{3\eps}                           
  \end{smallmatrix}
  \right),
\end{aligned}
\end{align}
and
\begin{align}
&\begin{aligned}
  \widetilde{D}_p = \left(
  \begin{array}{cccccccc}
       0     & 1-\delta_0 & \delta_0 + \frac{h^2}{3 \eps} &              0              &               0               &                             &                   &           \\
             &      0     &              0                & \delta_0 +\frac{h^2}{3\eps} & 1-\delta_0                    &              0              &                   &           \\
             &            &              0                & 1-\delta_0                  & \delta_0 + \frac{h^2}{3 \eps} &              0              &         0         &           \\
             &            &                               &              0              &               0               & \delta_0 +\frac{h^2}{3\eps} & 1-\delta_0        &     0       
  \end{array}
  \right),
\end{aligned}\\
&\begin{aligned}
  \widehat{D}_p &= Q_l \widetilde{D}_p Q_r = \frac1{h^2}
  \left(
  \begin{smallmatrix} 
    \dd+\frac{h}{3\eps} & 1-\dd                 &                       &                       \\
    1-\dd               & \dd+\frac{h^2}{3\eps} &                       &                       \\
                        &                       & \dd+\frac{h^2}{3\eps} & 1-\dd                 \\
                        &                       & 1-\dd                 & \dd+\frac{h^2}{3\eps}
  \end{smallmatrix}
  \right).
\end{aligned}
\end{align}
We continue with the analysis of the restriction, prolongation and
coarse operators.

\subsection{Analysis of the restriction, prolongation and coarse
operators}\label{subsec:coarse}

The same block-diagonalization is possible for the restriction and
prolongation operators. The calculation for the restriction gives
\begin{align}
  \begin{aligned}
    \widetilde{R} =& \frac12 \left(
      \begin{array}{cccccccc}
        1    &   1/2   &   1/2   &         &         &         &         &         \\
             &   1/2   &   1/2   &    1    &         &         &         &         \\
             &         &         &         &    1    &   1/2   &   1/2   &         \\
             &         &         &         &         &   1/2   &   1/2   &    1    
      \end{array}
    \right),
    \\
    \widehat{R} =& \frac12 {Q_l}_0 \widetilde{R} Q_r \\
    =& \frac1{2\sqrt{2}}
    \scalemath{0.8}{\left(
        \begin{array}{cccc}
          2 + \e{\I 2 \pi (k-J/2) h}                    & \e{\I 2 \pi (k - J/2) h}                        & (-1)^j \left(2 + \e{\I 2 \pi k h}\right) & (-1)^j \left(\e{\I 2 \pi k h}\right) \\
          (-1)^j \left(\e{-\I 2 \pi (k - J/2) h}\right) & (-1)^j \left(2 + \e{-\I 2 \pi (k-J/2) h}\right) & \e{-\I 2 \pi k h}                        & 2 + \e{-\I 2 \pi k h}
        \end{array}
      \right)},
  \end{aligned}
\end{align}
and for the prolongation operator we obtain
\begin{align}
  P = 2 R^\intercal, && \widehat{P} = Q_l \widetilde{P}
                        {Q_r}_0 = 2 \widehat{R}^*,
\end{align}
and finally for the coarse operator
\begin{align}\label{eqn:coarseoperator}
  \begin{aligned}
    Q_0^* A_0 Q_0 =& Q_0^* R A P Q_0 = Q_0^* R Q Q^* A Q Q^* P Q_0 \\
\implies \widehat{A}_0=& \widehat{R} \widehat{A} \widehat{P}
    = \frac1{H^2}
    \left(
      \begin{smallmatrix} 
        2 \dd+\frac{H^2}{3\eps}-\cos\left(2 \pi k H \right)           & (-1)^j \left(1-2\dd+\frac{H^2}{6\eps}\e{\I 2 \pi k H}\right) \\
        (-1)^j \left(1-2\dd+\frac{H^2}{6\eps}\e{-\I 2 \pi k H}\right) & 2 \dd+\frac{H^2}{3\eps}-\cos\left(2 \pi k H \right)
      \end{smallmatrix}
    \right),
  \end{aligned}
\end{align}
where $H=2h$. We notice that the coarse operator is different for $j$
even and $j$ odd; however, the matrices obtained for both cases are
similar, with similarity matrix $(-1)^j I$ where $I$ is the identity
matrix, and therefore have the same spectrum. In the rest of the paper
we assume $j$ is even, without loss of generality. This means that we
will be studying a node that is present in both the coarse and fine
meshes. We can now completely analyze the two grid iteration operator.

\subsection{Analysis of the two grid iteration operator}

The error reduction capabilities of Algorithm \ref{alg:precit} are
given by the spectrum of the iteration operator
\begin{align*}
  E = (I - P A_0^{-1} R A)(I - \alpha D^{-1} A),
\end{align*}
and we have shown that the 4-by-4 block Fourier-transformed operator
\begin{align*}
  \widehat{E}(k) = (I - \widehat{P}(k) \widehat{A}_0^{-1}(k)
  \widehat{R}(k) \widehat{A}(k))(I - \alpha \widehat{D}^{-1}(k)
  \widehat{A}(k))
\end{align*}
has the same spectrum. Then, we will focus on studying the spectral
radius $\rho\left(\widehat{E}(k)\right)$ in the next section, in order
to find the optimal relaxation parameter $\alpha_\text{opt}$.

\section{Study of optimal relaxation parameters} \label{sec:optrel}
We begin by recalling the study performed by Hemker et. al. for the
Poisson equation.

\subsection{Hemker et. al. results}
In \S4.1 of \cite{Hemker2004}, a smoothing analysis is performed,
which is an important first step in LFA studies. A comparison of the
spectrum of the \emph{point} block-Jacobi and \emph{cell} block-Jacobi
smoother with a relaxation parameter optimized only via a smoothing
analysis, they were obtained by Hemker et. al. is shown in Figure
\ref{fig:hsmoother}.
\begin{figure}
  \includegraphics[width=0.6\textwidth]{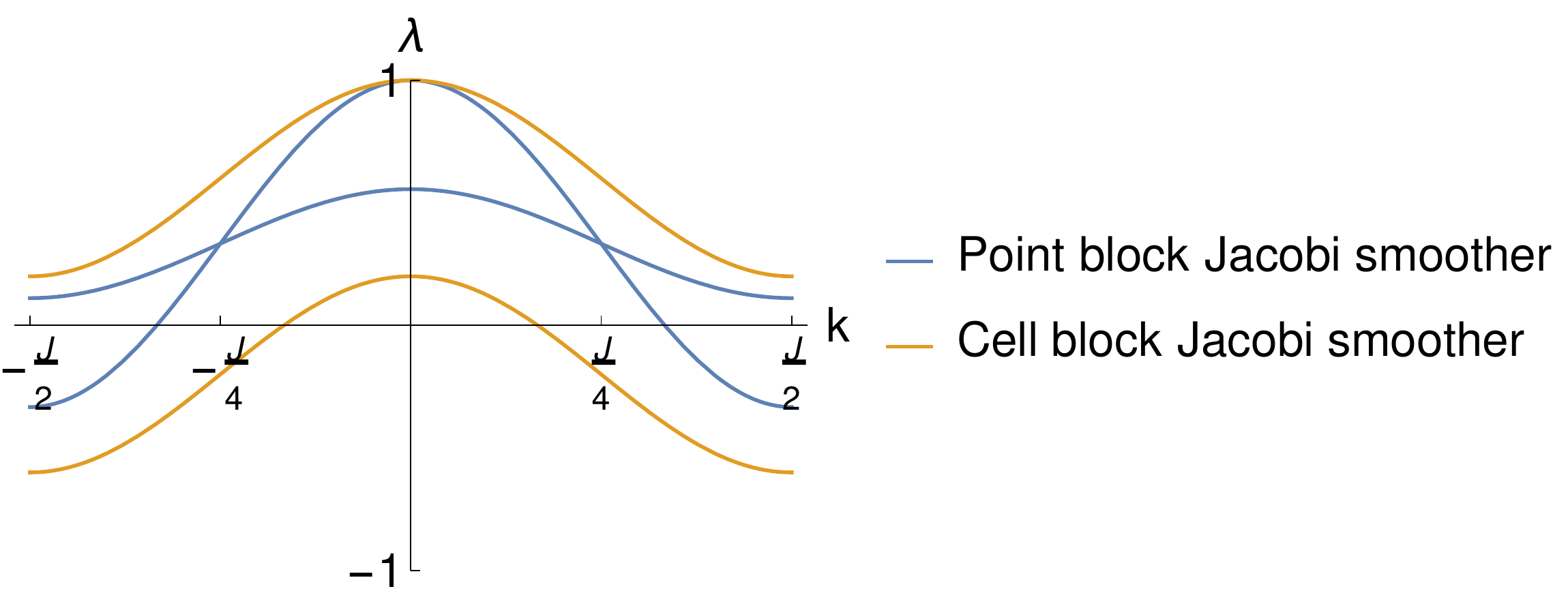}
  \caption{Spectrum of the \emph{point} block-Jacobi and \emph{cell}
    block-Jacobi smoothers for $\dd=2$, with optimized relaxation
    parameter without taking into account the coarse solver, following
    Hemker et. al. in \cite{Hemker2004}.}
  \label{fig:hsmoother}
\end{figure}
The smoothing analysis predicts an optimal relaxation parameter $4/5$
for the \emph{point} block-Jacobi smoother, and $2/3$ for the
\emph{cell} block-Jacobi smoother. We see that the smoothing
capabilities of the \emph{point} block-Jacobi smoother are better than
the \emph{cell} block-Jacobi smoother, since the upper half of the
spectrum corresponding to the higher frequencies is better damped
(equioscillation between $J/4$ and $J/2$).

In our study, we take into account the interaction of smoothing and
coarse correction when optimizing the relaxation parameter, in order
to get the best possible two level method, and we deduce explicit
formulas for the relaxation parameter. We will show that, for DG
penalization parameter values $\dd$ lower than a certain threshold
$\delta_c$, which we determine explicitly, the \emph{cell}
block-Jacobi smoother of Schwarz type leads to a more efficient
two-level method than the \emph{point} block-Jacobi smoother. This
threshold is higher than the frequently used DG penalization parameter
value $\dd=p(p+1)=2$ (where $p=1$ here is the polynomial degree). This
shows that, for these penalization regimes, it is of interest in
practice to use the \emph{cell} block-Jacobi smoother instead of the
\emph{point} block-Jacobi smoother which looks preferable based on the
smoothing analysis alone.

\subsection{Poisson equation}\label{sec:PoissonPoint}

We begin with the study of the Poisson equation, for which we can
completely quantify the optimal choice of the relaxation parameter in
the smoothing procedure to get the best error reduction in the two
level algorithm. The best choice is characterized by equioscillation
of the spectrum, in the sense that the absolute values of the maximum
and minimum eigenvalues of the error reduction operator are equal, and
is given in the following two Theorems.
\begin{theorem}[Optimal \emph{point} block-Jacobi two-level method]
  \label{thm:PoissonPoint}
  Let $A$ be the first order, nodal, SIPG discretization matrix of a
  1D Laplacian with periodic boundary conditions. The optimal relaxation
  parameter $\alpha_\text{opt}$, in order to maximize the error
  reduction of Algorithm \ref{alg:precit}, using a \emph{point}
  block-Jacobi smoother is given by
  \begin{equation}\label{AlphaOptBlockJacobiLaplace}
    \alpha_\text{opt}=\frac{(2 \dd-1)^2}{6 \dd^2-6 \dd+1}.
  \end{equation}
\end{theorem}
\begin{proof}
  We compute the spectrum of $\widehat{E}(k)$ and find its extrema for
  $-J/2 \le k \le J/2$. $\widehat{E}(k)$ has $4$ eigenvalues, two of
  which are zero since the coarse operator is of rank $2$. We focus on
  the non-zero eigenvalues $\lambda_+$ and $\lambda_-$, with $\lambda_+
  \ge \lambda_-$, shown as function of $k$ for several values of $\dd$
  in Figure \ref{fig:pointl},
  \begin{figure}
    \centering
    \begin{subfigure}{0.49\textwidth}
      \includegraphics[width=\textwidth]{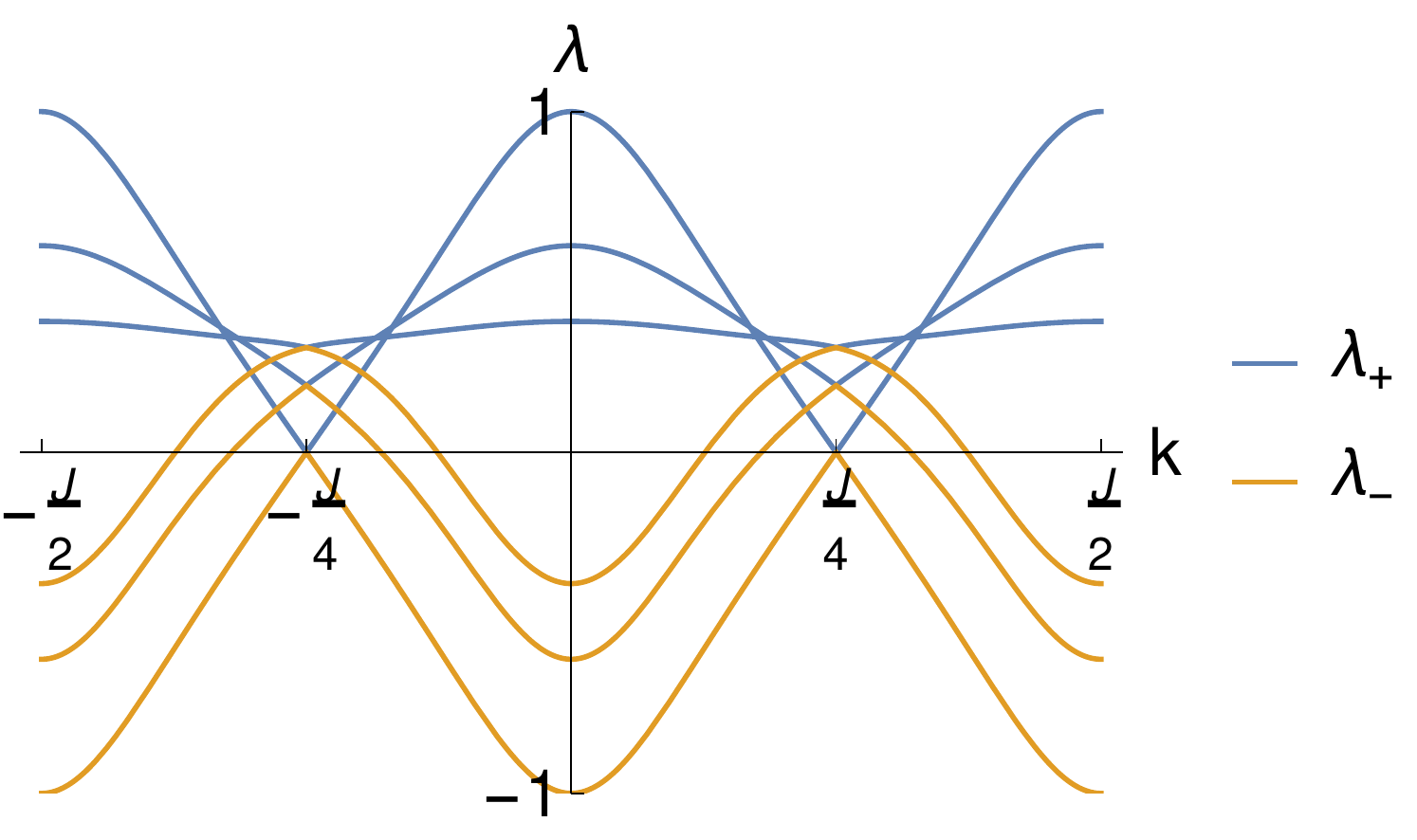}
      \caption{$\lambda_+$ and $\lambda_-$ for $\dd=1,1.2,2$ (in
        decreasing absolute value at $k=0$) using $\alpha_\text{opt}$.}
      \label{fig:pointl}
    \end{subfigure}
    \begin{subfigure}{0.49\textwidth}
      \includegraphics[width=\textwidth]{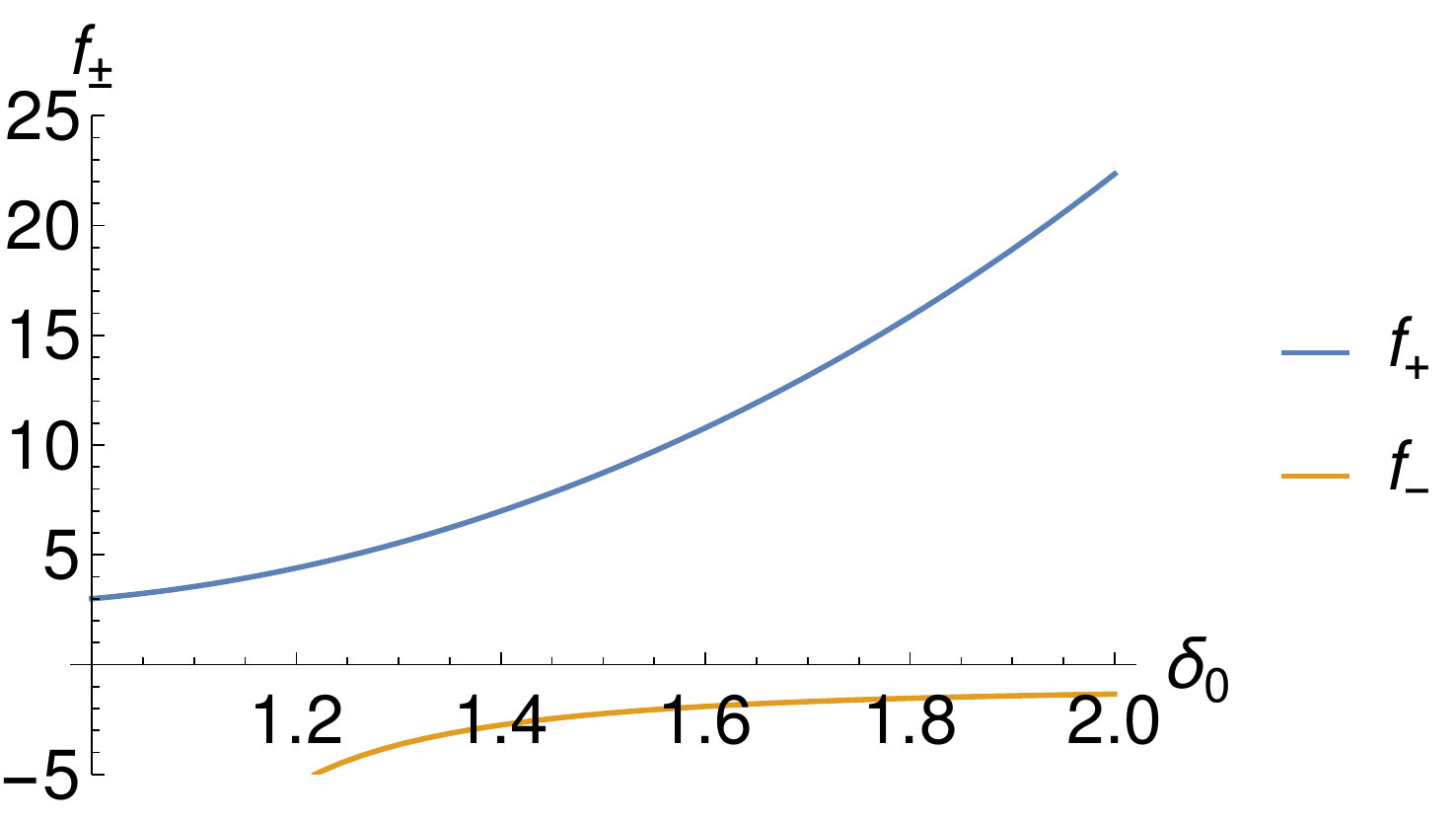}
      \caption{$f_+$ and $f_-$.}
      \label{fig:pointfpm}
    \end{subfigure}
    \caption{}
  \end{figure}
  \begin{equation}\label{eqn:pointrho}
    \scalemath{0.9}{\lambda_\pm = 1+\alpha\frac{-1+8 \dd-10
        \dd^2-\left(2 \dd^2-4 \dd+1\right) c_k \pm \sqrt{(c_k+1)(1-\dd)
          \left(c_k-f_-\right) \left(c_k-f_+\right)}}{(2 \dd-1) (4 \dd-c_k-1)}},
  \end{equation}
  where $c_k=\cos \left(\frac{4 \pi k}{J}\right)$ contains the
  dependence on $k$, and
  \begin{equation*}
    f_\pm(\dd) = \scalemath{0.77}{\frac{1-6\dd +8 \dd^2-8 \dd^3+4
        \dd^4 \pm \sqrt{1-8 \dd+16 \dd^2-48 \dd^3+120 \dd^4-160 \dd^5+128
          \dd^6-64 \dd^7+16 \dd^8}}{2 (\dd-1)}}.
  \end{equation*}
  
  The function $f_\pm(\dd)$ satisfies the following properties for
  $\dd \ge 1$, as one can see from a direct computation (see Figure
  \ref{fig:pointfpm}):
  \begin{enumerate}
  \item $f_+(\dd)$ is monotonically increasing,
    $\lim_{\dd\rightarrow 1} f_+(\dd) = 3$ and
    $\lim_{\dd\rightarrow \infty} f_+(\dd) \rightarrow \infty$,
    therefore
    $\left(c_k-f_+(\dd)\right) < 0$;
  \item $f_-(\dd)$ is monotonically increasing,
    $\lim_{\dd\rightarrow 1} f_-(\dd) \rightarrow -\infty$ and
    $\lim_{\dd\rightarrow \infty} f_-(\dd) = -1$, therefore
    $\left(c_k-f_+(\dd)\right) > 0$;
  \item $1 - \dd \le 0$ and $c_k + 1 \ge 0$, and thus with (1) and
    (2) we have
    $(c_k+1)(1-\dd)\left(c_k-f_-(\dd)\right)\left(c_k-f_+(\dd)\right)
    \ge 0$, and therefore $\lambda_{\pm}(\dd) \in \mathbb{R}$;
  \item $\lim_{\dd\rightarrow 1}
    (c_k+1)(1-\dd)\left(c_k-f_-(\dd)\right)\left(c_k-f_+(\dd)\right)
    = (c_k+1)\left(3-c_k\right)$, therefore $\lambda_+(\dd) =
    \lambda_-(\dd) \iff c_k = -1$, i.e. $k=J/4$.
  \end{enumerate}
  In order to obtain the extrema of $\lambda_\pm$ in $k$, we need to
  study $\frac{\partial \lambda_\pm}{\partial k}$, and since
  $\frac{\partial \lambda_\pm}{\partial k} = \frac{\partial
    \lambda_\pm}{\partial c_k} \frac{\partial c_k}{\partial k}$, we first
  compute
  \begin{align*}
    &\begin{aligned}
      &\frac{\partial \lambda_\pm}{\partial c_k} = \\
      &\alpha \bigg[\scalemath{0.9}{-1+9 \dd-28 \dd^2+64 \dd^3-64
        \dd^4+32 \dd^5 +\left(-3+23 \dd+64 \dd^3-64 \dd^4-56 \dd^2+32
          \dd^5\right) c_k}\\
      &\scalemath{0.85}{+\left(-3+15 \dd-12 \dd^2\right) c_k^2+(\dd-1)
        c_k^3 \pm 16 (1-\dd) \dd^2 \sqrt{(c_k+1)(1-\dd) \left(c_k-f_-\right)
          \left(c_k-f_+\right)}\bigg]\bigg/}\\
      &\left(\pm 2 (2 \dd-1) (-4 \dd+c_k+1)^2
        \sqrt{(c_k+1)(1-\dd)\left(c_k-f_-\right) \left(c_k-f_+\right)}\right).
    \end{aligned}
  \end{align*}
  
  We begin by looking for zeros of the numerator; separating the
  term with the square root and squaring both sides of the equation
  leads to
  \begin{align*}
    \begin{aligned}
      &(-4 \dd+c_k+1)^2 \\
      &\bigg[1-10 \dd+41 \dd^2-144 \dd^3+256 \dd^4-192 \dd^5+64 \dd^6\\
      &\left.+\left(128 \dd^6-384 \dd^5+512 \dd^4-368 \dd^3+148
          \dd^2-40 \dd+4\right) c_k\right.\\
      &\left.+\left(64 \dd^6-192 \dd^5+256 \dd^4-240 \dd^3+158
          \dd^2-52 \dd+6\right) c_k^2\right.\\
      &+\left(-16 \dd^3+36 \dd^2-24 \dd+4\right) c_k^3+\left(\dd^2-2
        \dd+1\right) c_k^4\bigg]=0.
    \end{aligned}
  \end{align*}
  This operation might add spurious roots to the original expression,
  so we analyze them individually.  The left hand side is a product of
  two factors, the second of which is a 4th degree polynomial in
  $c_k$. The application of the Cardano-Tartaglia formula leads to
  complex roots for $\dd \ge 1$, leaving only two real roots coming from
  the first factor, both at $c_k = -1 + 4 \dd$, but $\dd \ge 1$ and
  $|c_k| \le 1$, so there is no real root of $\frac{\partial
    \lambda_\pm}{\partial c_k}$. We deduce that $\frac{\partial
    \lambda_\pm}{\partial k}$ is zero only where $\frac{\partial
    c_k}{\partial k}=0$, i.e., $k=J/4,J/2$.

  We remark at this point that because the dependency on $k$ is
  contained in $c_k$, the eigenvalues at $k=0$ will be the same than at
  $k=J/2$, so it suffices to consider only the case $k=J/2$.
   
  We realize as well that the denominator vanishes for $c_k=-1$
  (i.e. $k = J/4$), and for the derivative when approaching this value,
  we get $\lim_{k \rightarrow J/4} \frac{\partial \lambda_\pm}{\partial
    k} = \lim_{k \rightarrow J/4} \frac{\partial \lambda_\pm}{\partial
    c_k} \frac{\partial c_k}{\partial k}$; multiplying and dividing by the
  factor $\sqrt{(c_k+1) (1-\dd) \left(c_k-f_-\right)
    \left(c_k-f_+\right)}$ we obtain
  \begin{align*}
    \lim_{k \rightarrow J/4} \frac{\partial \lambda_\pm}{\partial k}
    =& \scalemath{0.7}{\lim_{k \rightarrow J/4} \frac{\frac{\partial
       c_k}{\partial k}}{\sqrt{(c_k+1)(1-\dd)\left(c_k-f_-\right)
       \left(c_k-f_+\right)}} \lim_{k \rightarrow J/4} \frac{\partial
       \lambda_\pm}{\partial c_k} \sqrt{(c_k+1)(1-\dd)\left(c_k-f_-\right)
       \left(c_k-f_+\right)}} \\
    =& \scalemath{0.9}{
       \begin{cases} \displaystyle \frac{2 \sqrt{2} \pi}{\sqrt{\dd} J}
         \lim_{k \rightarrow J/4} \frac{\partial \lambda_\pm}{\partial c_k}
         \sqrt{(c_k+1)(1-\dd)\left(c_k-f_-\right) \left(c_k-f_+\right)}, & k
         \rightarrow (J/4)^+, \\ 
         \displaystyle -\frac{2 \sqrt{2}
           \pi}{\sqrt{\dd} J} \lim_{k \rightarrow J/4} \frac{\partial
           \lambda_\pm}{\partial c_k} \sqrt{(c_k+1)(1-\dd)\left(c_k-f_-\right)
           \left(c_k-f_+\right)}, & k \rightarrow (J/4)^-,
       \end{cases}} \\
    =& 
       \begin{cases} 
         \displaystyle \pm \frac{\sqrt{2} \alpha \pi}{(2 \dd - 1)
           \sqrt{\dd} J}, & k \rightarrow (J/4)^+, \\
         \displaystyle \mp \frac{\sqrt{2} \alpha \pi}{(2 \dd - 1)
           \sqrt{\dd} J}, & k \rightarrow (J/4)^-,
       \end{cases}
  \end{align*}
  therefore at $k=J/4$, $\lambda_+$ has a minimum and $\lambda_-$ has
  a maximum as observed in Fig. \ref{fig:pointl}.
  
  In order to determine if the extremum at $k=J/2$ is a minimum or a
  maximum, we compute the second derivative,
  \begin{equation*}
    \left.\frac{\partial^2 \lambda_+}{\partial k^2}\right|_{k=J/2} =
    \frac{8 \pi ^2 \alpha (1-2 \dd (2 (\dd-2) \dd+3))}{(2 \dd-1)^3 (2
      (\dd-1) \dd+1) J^2} < 0 \ \Longleftrightarrow \ 1 - 6 \dd + 8 \dd^2 -
    4 \dd^3 < 0,
  \end{equation*}
  which always holds for $\dd \ge 1$, and thus at $k=J/2$, $\lambda_+$
  has a maximum. Similarly, for $\lambda_-$, we find
  \begin{equation*}
    \left.\frac{\partial^2 \lambda_-}{\partial k^2}\right|_{k=J/2} =
    \frac{8 \pi ^2 \alpha (2 \dd (2 (\dd-1) \dd+1)-1)}{(2 \dd (\dd (2
      \dd-3)+2)-1) J^2} < 0 \ \Longleftrightarrow\ -1 + 2 \dd - 4 \dd^2 + 4
    \dd^3 < 0,
  \end{equation*}
  which never holds for $\dd \ge 1$, and thus at $k=J/2$, $\lambda_-$
  has a minimum, as we can see in Fig. \ref{fig:pointl}.
  
  To minimize the spectral radius, due to the monotonicity of the
  eigenvalues in the parameter $\alpha$, we can minimize the absolute
  value of $\lambda_\pm$ by just centering the eigenvalue distribution
  around zero. Using the explicit formulas for the extrema, this is
  achieved by equioscillation when the relaxation parameter
  $\alpha_\text{opt}$ satisfies $\lambda_+\big|_{k=J/2} =
  -\lambda_-\big|_{k=J/2}$, which gives
  \eqref{AlphaOptBlockJacobiLaplace}.
\end{proof}

\begin{theorem}[Optimal \emph{cell} block-Jacobi two-level
  method]\label{thm:PoissonCell}
  Let $A$ be the first order, nodal, SIPG discretization matrix of a
  1D Laplacian with periodic boundary conditions. The optimal relaxation
  parameter $\alpha_\text{opt}$, in order to maximize the error
  reduction of Algorithm \ref{alg:precit} using a \emph{cell}
  block-Jacobi smoother is given by
  \begin{align*}
    \alpha_\text{opt}=
    \begin{cases}
      \frac{\dd(2\dd-1)}{2\dd^2-1}, &\text{ for } 1 \le \dd \le
      \widetilde{\dd}_+, \\
      \frac{2 \dd^2 (2 \dd-1)}{\dd \left| 2 \dd^2-4 \dd+1 \right| +2
        \dd^3+4 \dd^2-5 \dd+1}, &\text{ for } \widetilde{\dd}_+ \le \dd \le
      \widetilde{\dd}_-, \\
      \frac{2\dd^2}{2\dd^2+\dd-1}, &\text{ for } \widetilde{\dd}_- \le
      \dd,
    \end{cases}
  \end{align*}
  where $\widetilde{\dd}_+=\frac{1}{12} \left(8+\sqrt[3]{152-24
      \sqrt{33}}+2 \sqrt[3]{19+3 \sqrt{33}}\right)= 1.41964\dots$ and
  $\widetilde{\dd}_-=3/2$.
\end{theorem}
\begin{proof}
  As in the proof of Theorem \ref{thm:PoissonPoint}, we compute the
  spectrum of $\widehat{E}(k)$ and find its extrema for $-J/2 \le k
  \le J/2$. Again $\widehat{E}(k)$ has $4$ eigenvalues, two of which are
  zero.

  The non-zero eigenvalues $\lambda_+$ and $\lambda_-$ are real, with
  $\lambda_+ \ge \lambda_-$, and are given by
  \begin{align}\label{eqn:cellrho}
    \lambda_\pm = 1+\alpha \left(\frac{2+\dd\left(c_k-4\dd-1\right)
    \pm \sqrt{\left(\dd^2-2\right) \left(c_k- f_- \right)
    \left(c_k-f_+\right)}}{\dd\left(4\dd-c_k-1\right)}\right),
  \end{align}
  where $c_k=\cos \left(\frac{4 \pi k}{J}\right)$ and $f_\pm(\dd) =
  \frac{\dd\left(4 \dd^2-7 \dd +2 \right) \pm 2 \sqrt{(2 \dd - 3) (4
      \dd^3 - 8 \dd^2 + 4 \dd - 1)}}{\dd^2-2}$, (see
  Figs. \ref{fig:celll1}, \ref{fig:celll2} and \ref{fig:celll3}).
  \begin{figure}
    \centering
    \begin{subfigure}{0.49\textwidth}
      \includegraphics[width=\textwidth]{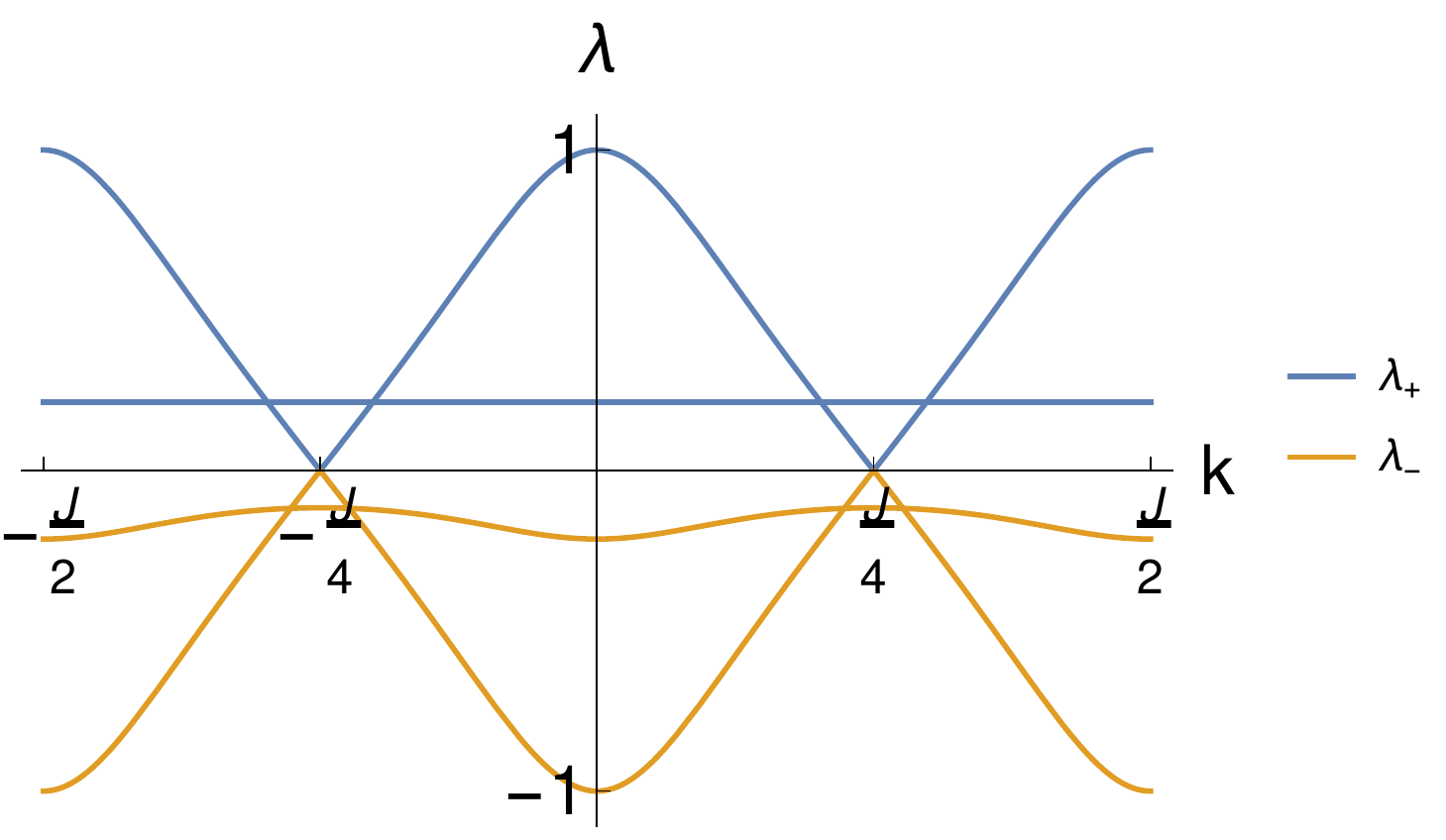}
      \caption{$\lambda_+$ and $\lambda_-$ for
        $\dd=1,\widetilde{\dd}_+$ (in decreasing absolute value at $k=0$)
        using $\alpha_\text{opt}$.}
      \label{fig:celll1}
    \end{subfigure}
    \begin{subfigure}{0.49\textwidth}
      \includegraphics[width=\textwidth]{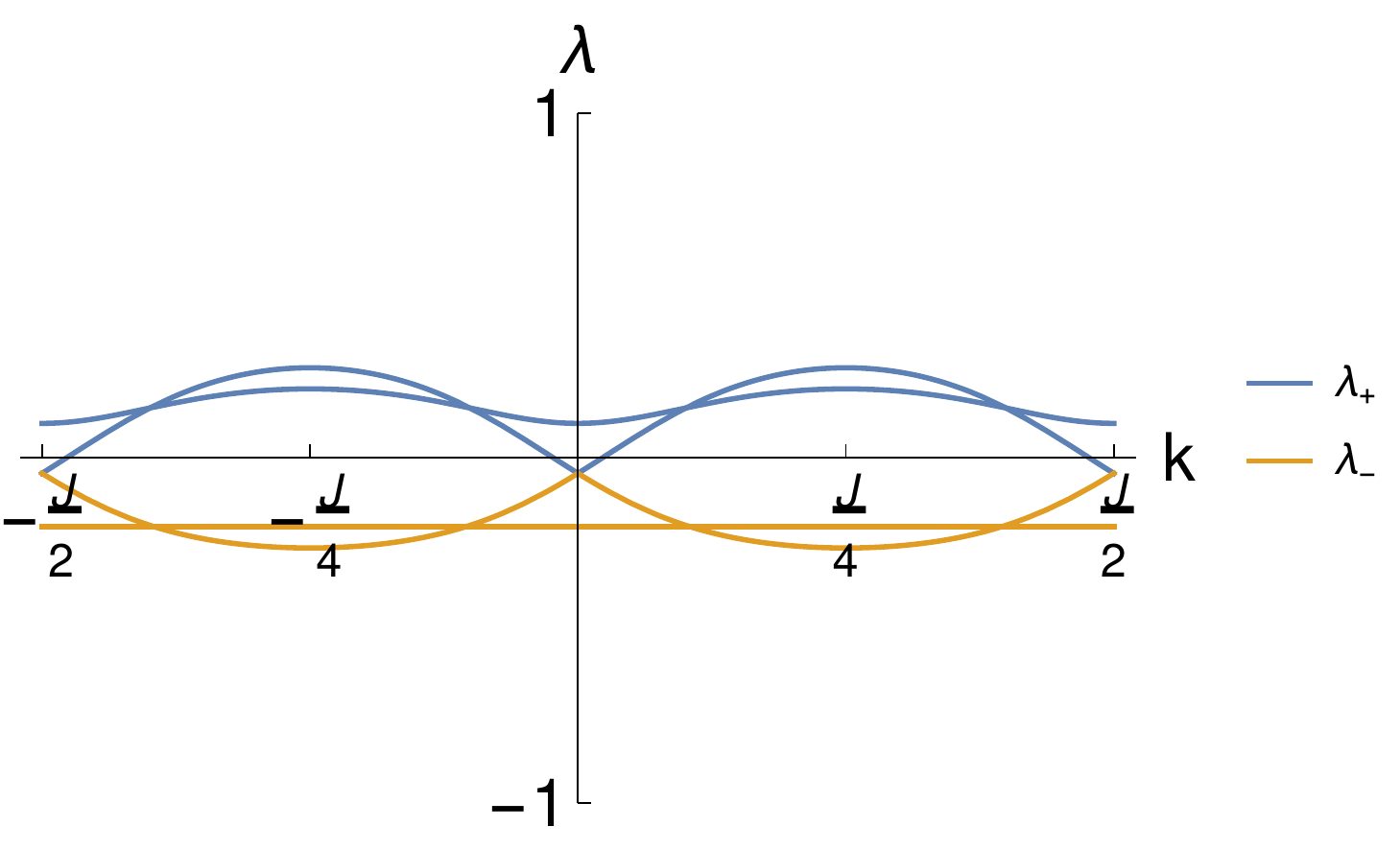}
      \caption{$\lambda_+$ and $\lambda_-$ for
        $\dd=\widetilde{\dd}_-,\frac{2+\sqrt{2}}{2}$ (in decreasing absolute
        value at $k=0$) using $\alpha_\text{opt}$.}
      \label{fig:celll2}
    \end{subfigure} \\
    \begin{subfigure}{0.49\textwidth}
      \includegraphics[width=\textwidth]{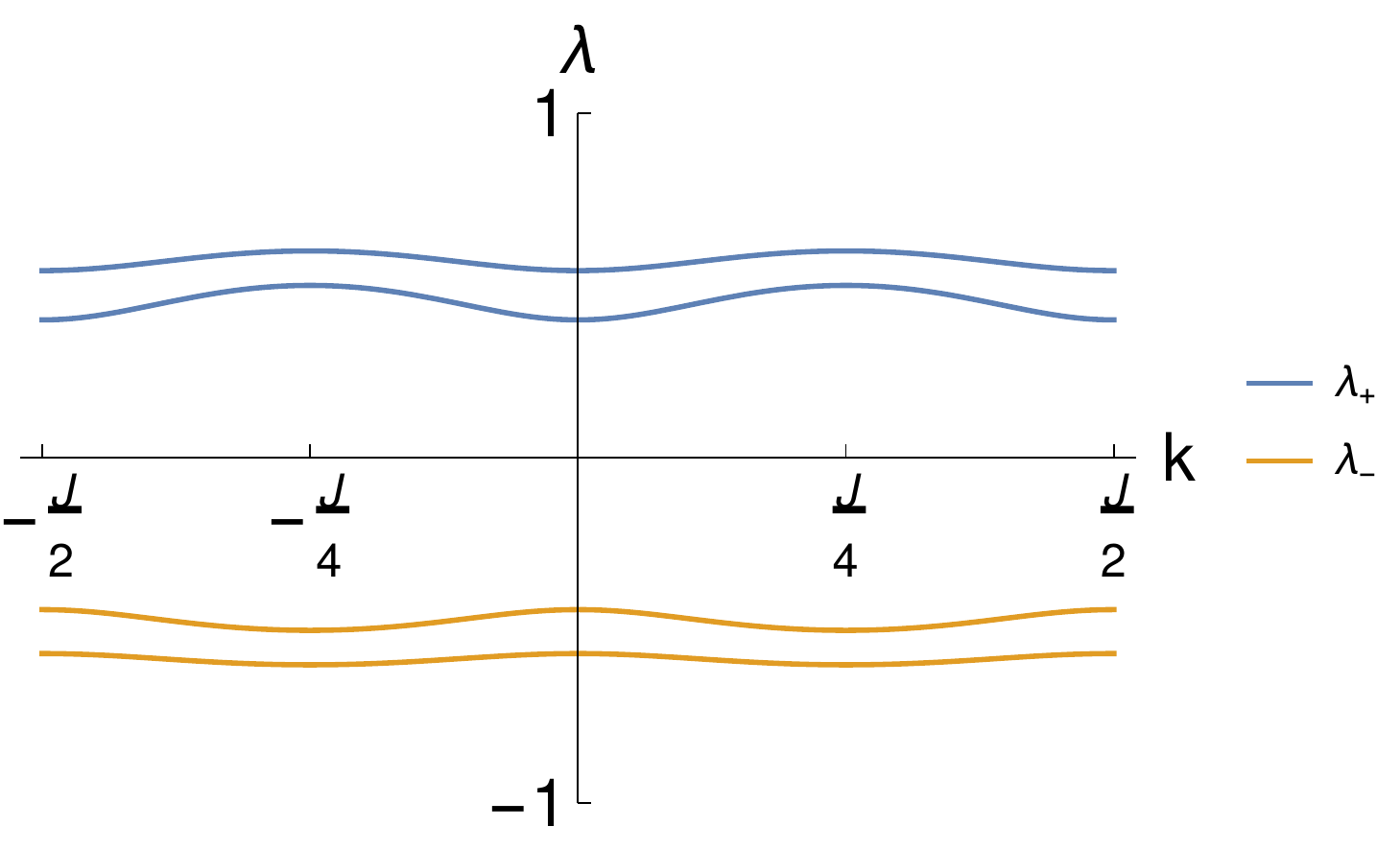}
      \caption{$\lambda_+$ and $\lambda_-$ for $\dd=3,4$ (in increasing
        absolute value at $k=0$) using $\alpha_\text{opt}$.}
      \label{fig:celll3}
    \end{subfigure}
    \begin{subfigure}{0.49\textwidth}
      \includegraphics[width=\textwidth]{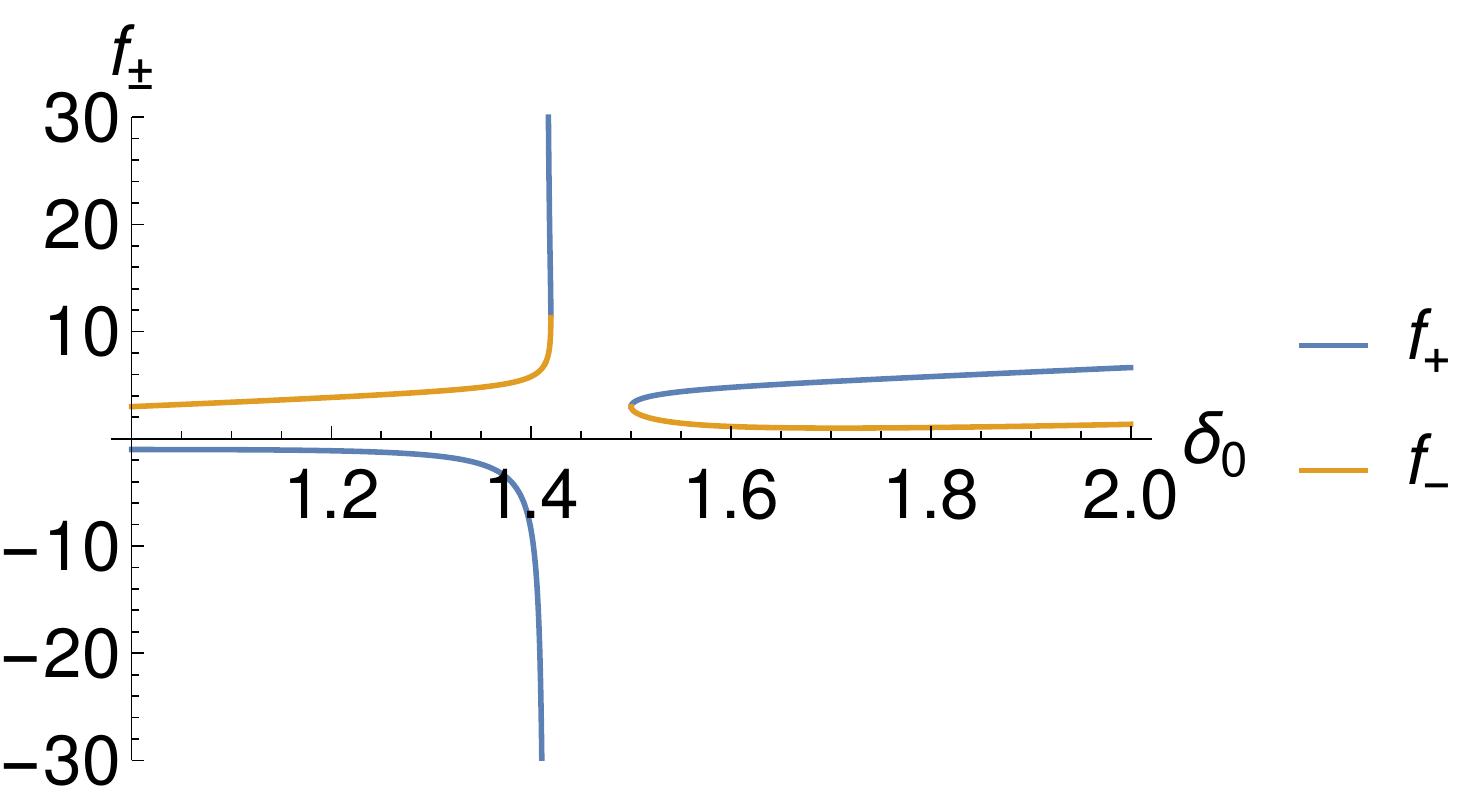}
      \caption{$f_+$ and $f_-$.}
      \label{fig:cellfpm}
    \end{subfigure}
    \caption{}
  \end{figure}
  A direct computation shows for $\dd \ge 1$ that (see Fig. \ref{fig:cellfpm})
  \begin{enumerate}
  \item $f_+=-1 \iff \dd = 1$,
  \item $f_-=1 \iff \dd=\frac{2 + \sqrt{2}}{2}$,
  \item $f_\pm \notin \mathbb{R} \iff \dd \in (\sqrt{2},\frac{2 +
      \sqrt{2}}{2})$,
  \item elsewhere $|f_\pm| > 1$.
  \end{enumerate}
  To find the extrema of $\lambda_\pm$ in $k$, we compute again the
  derivative $\frac{\partial \lambda_\pm}{\partial k} = \frac{\partial
    \lambda_\pm}{\partial c_k} \frac{\partial c_k}{\partial k}$ and obtain
  \begin{align}\label{eqn:dl11}
    &\begin{aligned}
      \frac{\partial \lambda_\pm}{\partial c_k} = \scalemath{0.8}{
        \alpha \frac{6 - 26 \dd + 50 \dd^2 - 24 \dd^3 + \left(6 \dd^2-10
            \dd+2\right) c_k \mp \sqrt{4 (\dd-1)^2 \left(\dd^2-2\right)
            \left(c_k-f_-\right) \left(c_k-f_+\right)}}{\pm \dd (-4 \dd+c_k+1)^2
          \sqrt{\left(\dd^2-2\right) \left(c_k-f_-\right)
            \left(c_k-f_+\right)}}}.
    \end{aligned}
  \end{align}
  We now look for roots of the numerator
  \begin{multline}\label{eqn:dl12}
    6 - 26 \dd + 50 \dd^2 - 24 \dd^3 + \left(6 \dd^2-10 \dd+2\right)
    c_k \\
    \mp \sqrt{4 (\dd-1)^2 \left(\dd^2-2\right) \left(c_k-f_-\right)
      \left(c_k-f_+\right)}=0.
  \end{multline}
  We first note that if $f_-=f_+=f$, i.e. $(2 \dd - 3) (4 \dd^3 - 8
  \dd^2 + 4 \dd - 1)=0$, we have
  \begin{multline*}
    6 - 26 \dd + 50 \dd^2 - 24 \dd^3 \pm f \sqrt{4 (\dd-1)^2
      \left(\dd^2-2\right)} \\ 
    + \left(6 \dd^2-10 \dd+2 \mp \sqrt{4
        (\dd-1)^2 \left(\dd^2-2\right)} \right) c_k = 0.
  \end{multline*}
  The factor multiplying the $c_k$ has roots,
  \begin{gather}
    \begin{gathered}\label{eqn:dlroots}
      6 \dd^2-10 \dd+2 \mp \sqrt{4 (\dd-1)^2 \left(\dd^2-2\right)} = 0
    \end{gathered}\\
    \begin{aligned}
      \implies & \left(-6 \dd^2+10 \dd-2 \right)^2 = 4 (\dd-1)^2
      \left(\dd^2-2\right) \\
      \iff & 8 \dd^4-28 \dd^3+32 \dd^2-14 \dd+3 = 0 \\
      \iff & (2 \dd - 3) (4 \dd^3 - 8 \dd^2 + 4 \dd - 1) = 0,
    \end{aligned}\nonumber
  \end{gather}
  where we might have added spurious roots to the original expression
  by squaring both sides, so we analyze them individually. We see that
  this is the same condition for $f_-=f_+=f$. There are, therefore,
  $\widetilde{\dd}_\pm$ such that $\frac{\partial \lambda_\pm}{\partial
    k}=0$ independently of $k$. Such $\widetilde{\dd}_\pm$ are found by
  obtaining the real roots of the polynomial from equation
  \eqref{eqn:dlroots},
  \begin{align*} 
    \widetilde{\dd}_+=&\frac{1}{12} \left(8+\sqrt[3]{152-24 \sqrt{33}}
                        + 2 \sqrt[3]{19+3 \sqrt{33}}\right) = 1.41964\dots , \\
    \widetilde{\dd}_-=&\frac32.
  \end{align*}
  We now take equation \eqref{eqn:dl12} and compute the roots with
  respect to $c_k$,
  \begin{multline*}
    \left(6 - 26 \dd + 50 \dd^2 - 24 \dd^3 + \left(6 \dd^2-10
        \dd+2\right) c_k\right)^2 =\\
    4 (\dd-1)^2 \left(\dd^2-2\right) \left(c_k-f_-\right)
    \left(c_k-f_+\right);
  \end{multline*}
  a simplification gives
  \begin{align*}
    c_k^2+(2-8 \dd) c_k + \left(16 \dd^2-8 \dd+1\right) = 0,
  \end{align*}
  which has two roots that are equal to $c_k = -1 + 4 \dd$, but $\dd
  \ge 1$, so there is no real root of $\frac{\partial
    \lambda_\pm}{\partial c_k}$. We deduce from this and the chain rule,
  that $\frac{\partial \lambda_\pm}{\partial k}$ is zero only where
  $\frac{\partial c_k}{\partial k}=0$, hence the roots are located at $k
  = J/4, J/2$ (i.e. $c_k=1,-1$ ), except when $\lambda_+$ or $\lambda_-$
  do not depend on $k$.

  We remark at this point that because the dependency on $k$ is
  contained in $c_k$, the eigenvalues at $k=0$ will be the same than at
  $k=J/2$. In what follows, we will only analyze the case $k=J/2$.

  We see that the denominator of \eqref{eqn:dl11} has roots at
  \begin{enumerate}
  \item $\dd = \sqrt{2}$, but given that $|c_k|\le 1$ we have
    \begin{equation*}
      \lim_{\dd \rightarrow \sqrt{2}} \left(\dd^2-2\right)
      \left(c_k-f_-\right) \left(c_k-f_+\right) = -4 (-50 + 35 \sqrt{2} +
      (-7 + 5 \sqrt{2}) c_k) \ne 0 ;
    \end{equation*}
    since $f_\pm$ contains the term $(\dd^2-2)$ in the denominator.
  \item $\dd=1,c_k=-1$ i.e. $k = J/4$, 
    \begin{align*}
      \lim_{\substack{\dd \rightarrow 1 \\
      k \rightarrow J/4}} \frac{\partial \lambda_\pm}{\partial k} =&
                                  \lim_{\substack{\dd \rightarrow 1 \\
      k \rightarrow J/4}} \frac{\partial \lambda_\pm}{\partial c_k}
      \frac{\partial c_k}{\partial k} =\pm \frac{2 \alpha}{(3 -
      c_k)^{\frac32} \sqrt{1 + c_k}} \left(- \frac{4 \pi s_k}{J} \right) \\
      =& \begin{cases}
        \mp \frac{\sqrt{2} \alpha \pi}{J}, & k \rightarrow (J/4)^+\\
        \pm \frac{\sqrt{2} \alpha \pi}{J}, & k \rightarrow (J/4)^-
      \end{cases},
    \end{align*}
    where $s_k=\sin\left(\frac{4 \pi k}{J}\right)$, hence there is a
    minimum for $\lambda_+$ and a maximum for $\lambda_-$;
  \item $\dd=\frac{2+\sqrt{2}}{2},c_k=1$, where 
    \begin{multline*}
      \lim_{\substack{\dd \rightarrow 1 \\ 
          k \rightarrow J/2}} \frac{\partial \lambda_\pm}{\partial k}
      = \lim_{\substack{\dd \rightarrow 1 \\
          k \rightarrow J/2}} \frac{\partial \lambda_\pm}{\partial
        c_k} \frac{\partial c_k}{\partial k} \\
      = \scalemath{0.9}{\lim_{k \rightarrow J/2} \left(\frac{2
            \alpha}{c_k - 3 - 2 \sqrt{2}} \left( 1 - \sqrt{2} \pm \frac{c_k -
              5}{\sqrt{(c_k-1)\left(\left(2\sqrt{2}-1\right)c_k - 7 -
                  2\sqrt{2}\right)}} \right) \right) \left(- \frac{4 \pi s_k}{J}
        \right)} \\
      = \begin{cases}
        \pm \frac{4 \alpha \pi (2 - \sqrt{2})}{(2 + \sqrt{2}) J}, & k
        \rightarrow (J/2)^+\\
        \mp \frac{4 \alpha \pi (2 - \sqrt{2})}{(2 + \sqrt{2}) J}, & k
        \rightarrow (J/2)^-
      \end{cases},
    \end{multline*}
    therefore it is a minimum for $\lambda_+$ and a maximum for
    $\lambda_-$.
  \end{enumerate}
  
  Thus, in the following we will assume that $\dd \ne 1$ and $\dd \ne
  \frac{2+\sqrt{2}}{2}$.
  
  In order to determine if the extremum at $k=J/4$ is a minimum or a
  maximum we compute the second derivative,
  \begin{equation*}
    \left.\frac{\partial^2 \lambda_+}{\partial k^2}\right|_{k=J/4} < 0
    \Longleftrightarrow \frac{4 \pi^2 \alpha \left(1 - 4 \dd + 8 \dd^2 - 4
        \dd^3\right)}{\dd^3 J^2 (\dd-1) \left(2\dd-1\right)} < 0
    \Longleftrightarrow 1 - 4 \dd + 8 \dd^2 - 4 \dd^3 < 0.
  \end{equation*}
  The only real root of this polynomial is $\widetilde{\dd}_+$, and we
  conclude that at $k=J/4$, for $\dd<\widetilde{\dd}_+$, $\lambda_+$ has
  a minimum, and conversely, for $\dd>\widetilde{\dd}_+$ it has a maximum.
  For the second eigenvalue, we get
  \begin{equation*}
    \left.\frac{\partial^2 \lambda_-}{\partial k^2}\right|_{k=J/4} < 0
    \Longleftrightarrow \frac{4 \pi^2 \alpha \left(2\dd-3\right)}{\dd J^2
      (\dd-1) \left(2\dd-1\right)} < 0 \Longleftrightarrow 2 \dd - 3 < 0,
  \end{equation*}
  and we conclude that at $k=J/4$, for $\dd<\widetilde{\dd}_-$,
  $\lambda_-$ has a maximum, and conversely, for $\dd>\widetilde{\dd}_-$
  it has a minimum.
  
  Similarly, at $k=J/2$, we find
  \begin{gather*}
    \left.\frac{\partial^2 \lambda_+}{\partial k^2}\right|_{k=J/2} < 0 \\
    \Longleftrightarrow\frac{8 \pi ^2 \alpha \left(\frac{(2 \dd-1) (d
          (2 \dd (3 \dd-7)+9)-2)}{\left| 1-2 (\dd-1) \dd (2 \dd-3)\right|
        }+\dd-1\right)}{\dd (J-2 d J)^2} < 0\\
    \Longleftrightarrow2-13 \dd+32 \dd^2 -34 \dd^3+12 \dd^4+ (\dd-1)
    \left| -4 \dd^3+10 \dd^2-6 \dd+1\right| < 0 \\
    \Longleftrightarrow\begin{cases}
      -1 + 4 \dd - 8 \dd^2 + 4 \dd^3 < 0  & \text{ if } \dd < \frac{2+\sqrt{2}}2, \\
      -2 + 9 \dd - 14 \dd^2 + 6 \dd^3 < 0 & \text{ if } \dd = \frac{2+\sqrt{2}}2, \\
      2 \dd - 3 < 0                       & \text{ if } \dd > \frac{2+\sqrt{2}}2,
    \end{cases}\\
    \Longleftrightarrow-1 + 4 \dd - 8 \dd^2 + 4 \dd^3 < 0,
  \end{gather*}
  and we conclude that at $k=J/2$, for $\dd<\widetilde{\dd}_+$,
  $\lambda_+$ has a maximum, and conversely, for $\dd>\widetilde{\dd}_+$
  it has a minimum. And finally,
  \begin{gather*}
    \left.\frac{\partial^2 \lambda_-}{\partial k^2}\right|_{k=J/2} < 0 \\
    \Longleftrightarrow-2+13 \dd-32 \dd^2+34 \dd^3-12 \dd^4+ (\dd-1)
    \left| -4 \dd^3+10 \dd^2-6 \dd+1\right| < 0 \\
    \Longleftrightarrow\begin{cases}
      3 - 2 \dd < 0                      & \text{ if } \dd < \frac{2+\sqrt{2}}2, \\
      2 - 9 \dd + 14 \dd^2 - 6 \dd^3 < 0 & \text{ if } \dd = \frac{2+\sqrt{2}}2, \\
      1 - 4 \dd + 8 \dd^2 - 4 \dd^3 < 0  & \text{ if } \dd > \frac{2+\sqrt{2}}2, 
    \end{cases}\\
    \Longleftrightarrow3 - 2 \dd < 0,
  \end{gather*}
  and we conclude that at $k=J/2$, for $\dd>\widetilde{\dd}_-$,
  $\lambda_-$ has a maximum, and conversely, for $\dd<\widetilde{\dd}_-$
  it has a minimum.
  
  In order to minimize the spectral radius we have to center again the
  eigenvalue distribution around zero, using the explicit formulas
  developed above. The result thus follows from the solution of
  \begin{align*}
    \begin{cases}
      \lambda_+\big|_{k=J/2} = -\lambda_-\big|_{k=J/2}, \text{ for } 1
      \le \dd \le \widetilde{\dd}_+, \\
      \lambda_+\big|_{k=J/4} = -\lambda_-\big|_{k=J/2}, \text{ for }
      \widetilde{\dd}_+ \le \dd \le \widetilde{\dd}_-, \\
      \lambda_+\big|_{k=J/4} = -\lambda_-\big|_{k=J/4}, \text{ for }
      \widetilde{\dd}_- \le \dd.
    \end{cases}
  \end{align*}
\end{proof}

\begin{figure}
  \includegraphics[width=0.5\textwidth]{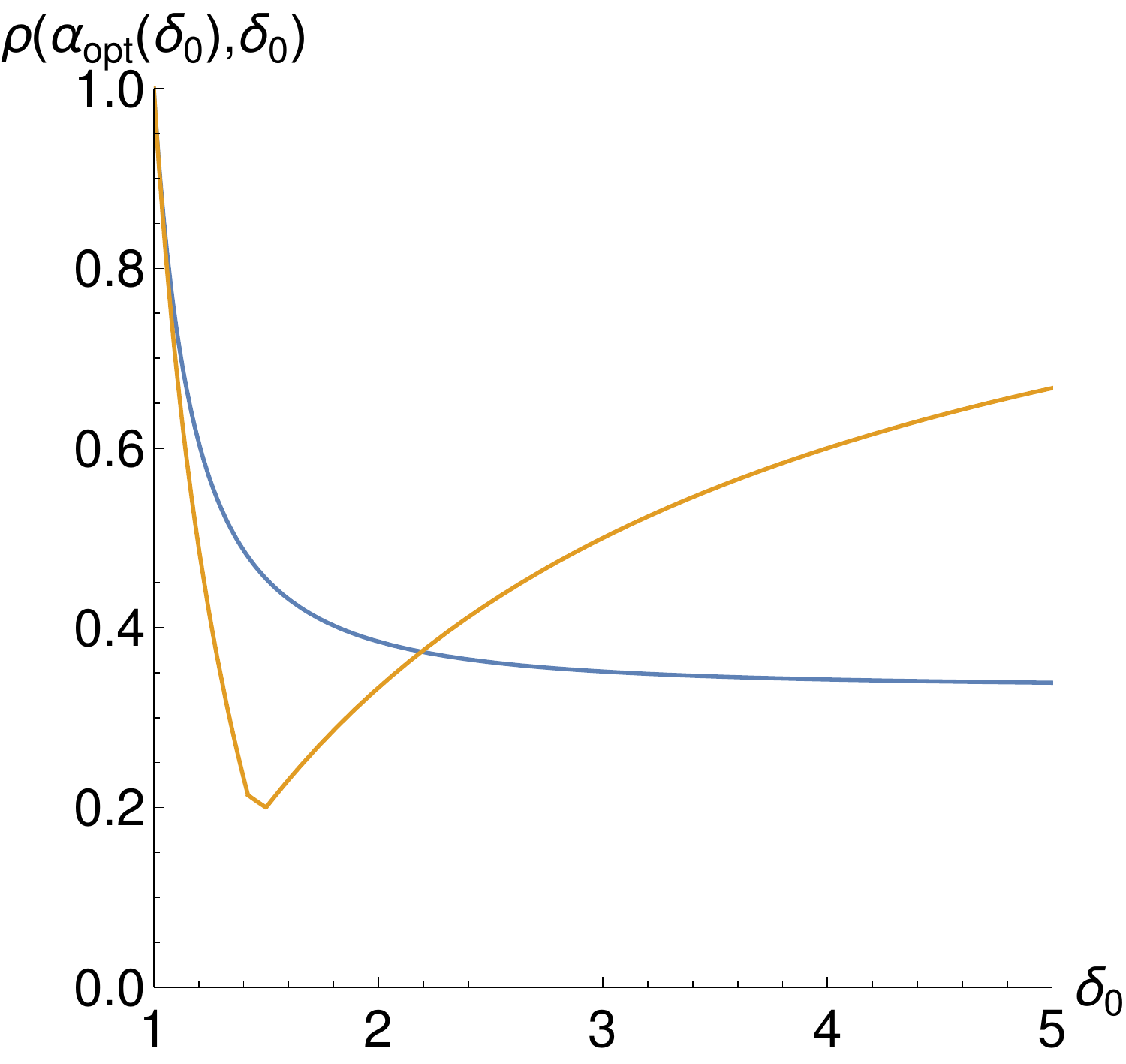}
  \caption{Spectral radius
    $\rho\left(\alpha_\text{opt}(\dd), \dd\right)$ of the iteration
    operator of Algorithm \ref{alg:precit} using an optimal relaxation
    parameter, for a \emph{point} block-Jacobi smoother (blue) and a
    \emph{cell} block-Jacobi smoother (orange) as function of the
    penalization parameter $\dd$.}
\label{fig:PoissonTheo}
\end{figure}
Figure \ref{fig:PoissonTheo} shows the contraction factor as function
of the penalization parameter $\dd$ for the \emph{point} block-Jacobi
and \emph{cell} block-Jacobi two-level methods using the best
relaxation parameter $\alpha_{\text{opt}}$ from Theorem
\ref{thm:PoissonPoint} and \ref{thm:PoissonCell}. We see that the
\emph{cell} block-Jacobi smoother outperforms the \emph{point}
block-Jacobi smoother for values of $\dd\le \delta_c = 1 + \frac{1}{6}
\sqrt[3]{54 - 6 \sqrt{33}} + \sqrt[3]{\frac14 + \frac{\sqrt{33}}{36}}
\approx 2.19149$. For larger penalization parameters $\dd$ the
\emph{point} block-Jacobi two-level method converges faster. This can
be understood intuitively as follows: the more we penalize the jumps,
the more important the face terms in the bilinear form become and,
after a threshold, a preconditioner that takes into account all the
terms containing this penalization begins performing better than a
preconditioner which does not.

It should be noted that even though large values of $\dd$ are a better
choice when using the \emph{point} block-Jacobi smoother, this also
means that the discretization of the coarse space will be harder to
invert, since according to equation \eqref{eqn:coarseoperator} the
penalty is doubled.

We can also observe that we obtain the best performance for
$\dd={\dd}_-=\frac32$, shown in Figure \ref{fig:PoissonTheo} as the
minimum of the orange curve. This shows that the penalization
parameter in SIPG has a direct influence on the two-level solver, and
there is an optimal choice $\dd={\dd}_-$ for best
performance. Choosing other values for $\dd$ can make the solver
slower by an order of magnitude, even if the best relaxation parameter
is chosen!

\subsection{Reaction-diffusion equation}

We now use LFA to study the more general reaction-diffusion case. The
computations become substantially more involved, but we will still be
able to \emph{center} the spectrum to derive relaxation parameter
values that lead to very effective two-level methods, even though we
can not formally prove optimality as in the simpler case of the
Poisson equation in the previous subsection. We will however provide
numerical evidence for the optimality in Section \ref{NumSec}.  For
the reaction-diffusion case, we see from the elements in the matrices
shown in \S\ref{sec:lfatools} that the key physical parameter is
\begin{equation}
  \gamma = \frac{\eps}{h^2} = \eps J^2.
\end{equation}
When $\eps$ becomes small, i.e. the reaction dominated case, the mesh
size needs to resolve boundary layers, and we then need
$h\sim\sqrt{\eps}$ \cite[\S1.3.2]{Gie2018} (see also
\cite{Kopteva2010} and references therein), which implies that
$\gamma$ is of order $1$. When $\eps$ is not small however, the mesh
size does not depend on $\eps$, and thus $\gamma$ can become large. We
therefore need a two-level method which is robust for a large range of
physical values $\gamma$.
  
\subsubsection{\emph{Point} block-Jacobi smoother}

By direct calculation, the eigenvalues of the iteration operator of
Algorithm \ref{alg:precit} for the reaction-diffusion equation case
using a \emph{point} block-Jacobi smoother are of the form
\begin{gather}\label{eqn:rdpoint}
  \lambda_\pm = \frac{c_1+c_2 x+c_3 x^2 \pm \sqrt{c_4 + c_5 x + c_6
      x^2 + c_7 x^3 + c_8 x^4 + c_9 x^5}} {c_{10} + c_{11} x + c_{12} x^2},
\end{gather}
where $x=\cos\left(\frac{4 \pi k}J \right)$, and the
$c_1,\dots,c_{12}$, depending on $\dd$, are defined in Appendix
\ref{apx:pointc}.  Figure
\ref{fig:pointepst}
\begin{figure}
  \centering
  \begin{subfigure}{0.49\textwidth}
    \includegraphics[width=\textwidth]{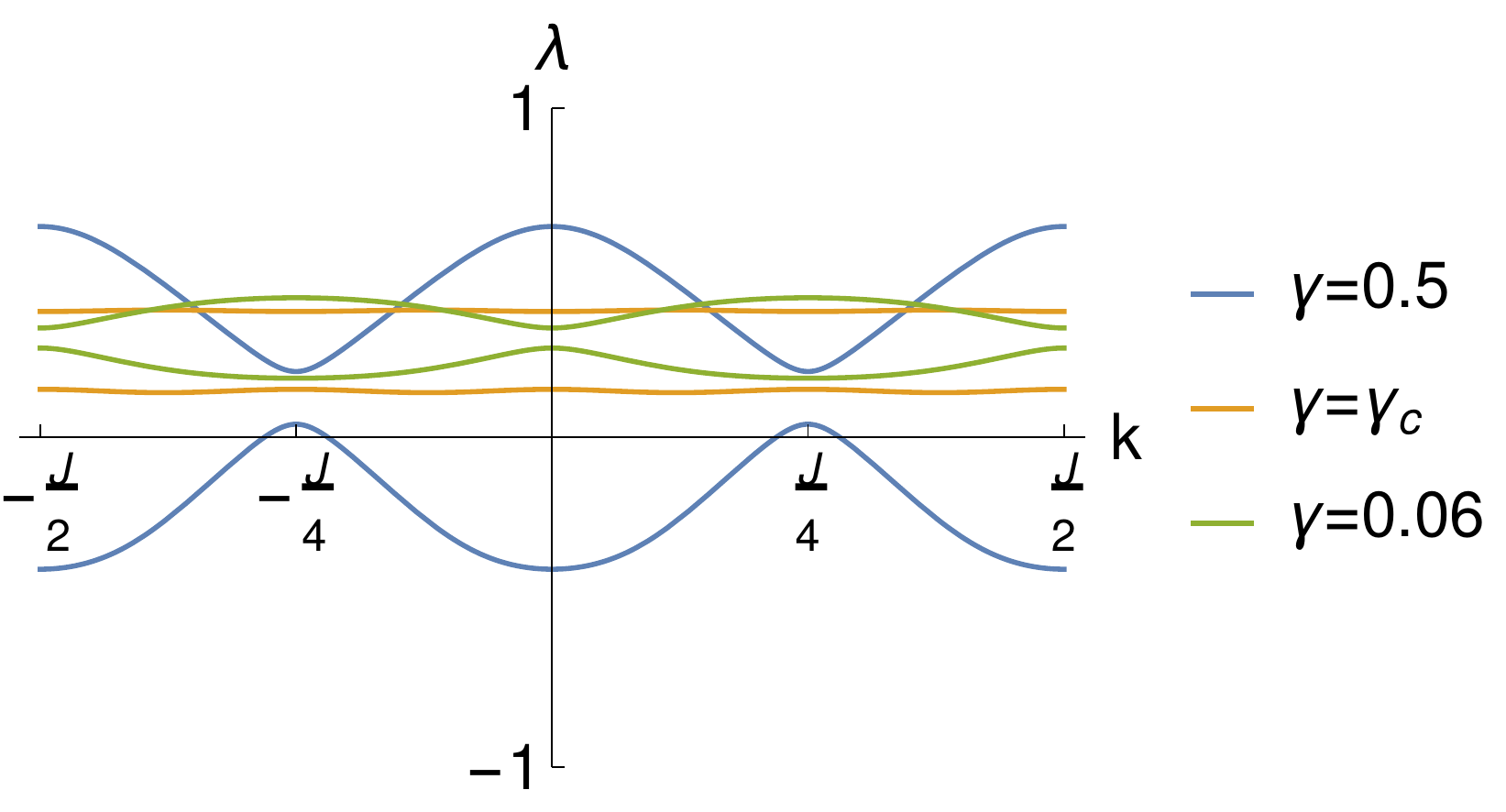}
    \caption{$\dd=1,\alpha=1$.}
    \label{fig:pointepst}
  \end{subfigure}\\
  \begin{subfigure}{0.49\textwidth}
    \includegraphics[width=\textwidth]{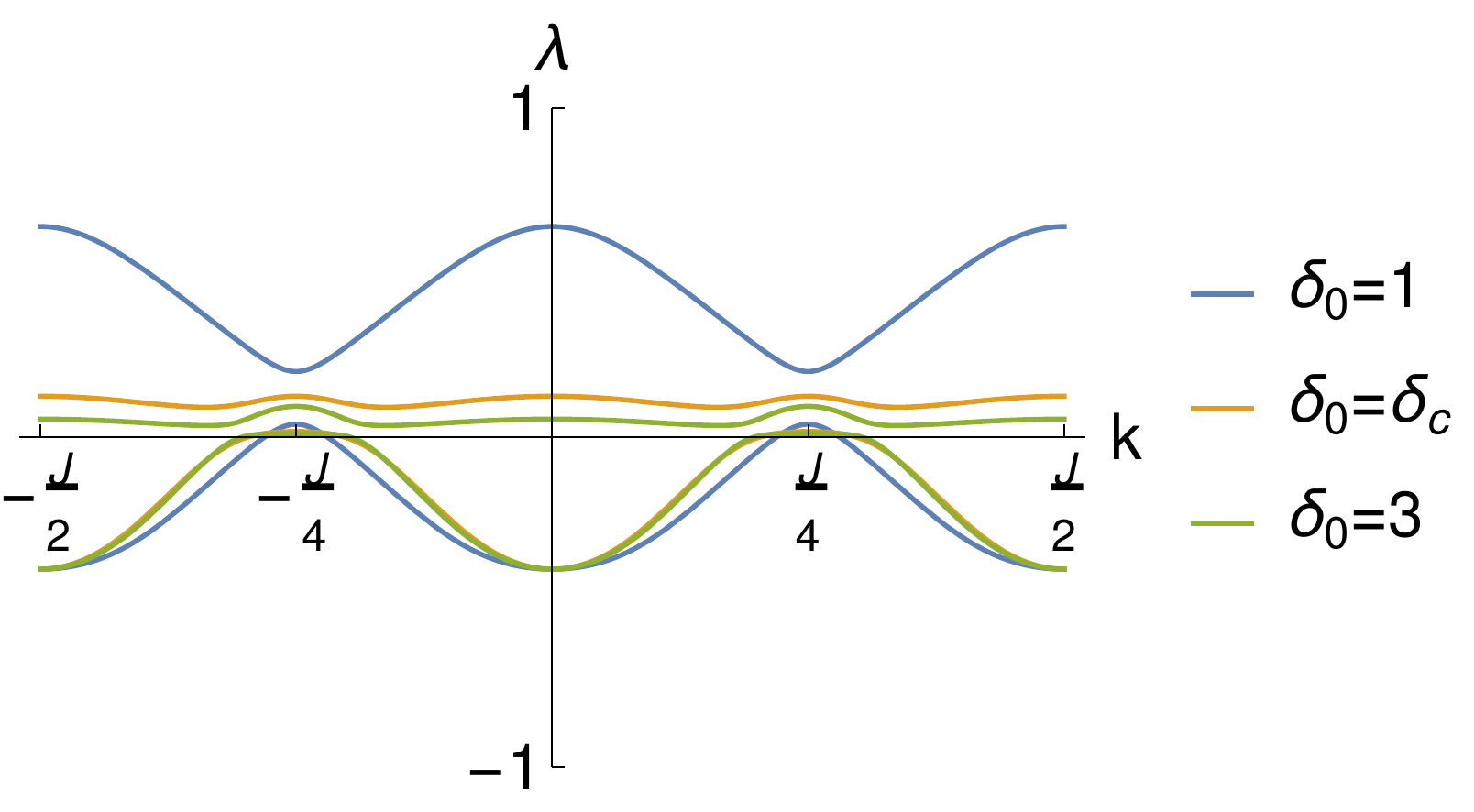}
    \caption{$\gamma=0.5>\gamma_c,\alpha=1$.}
    \label{fig:pointepst2}
  \end{subfigure}
  \begin{subfigure}{0.49\textwidth}
    \includegraphics[width=\textwidth]{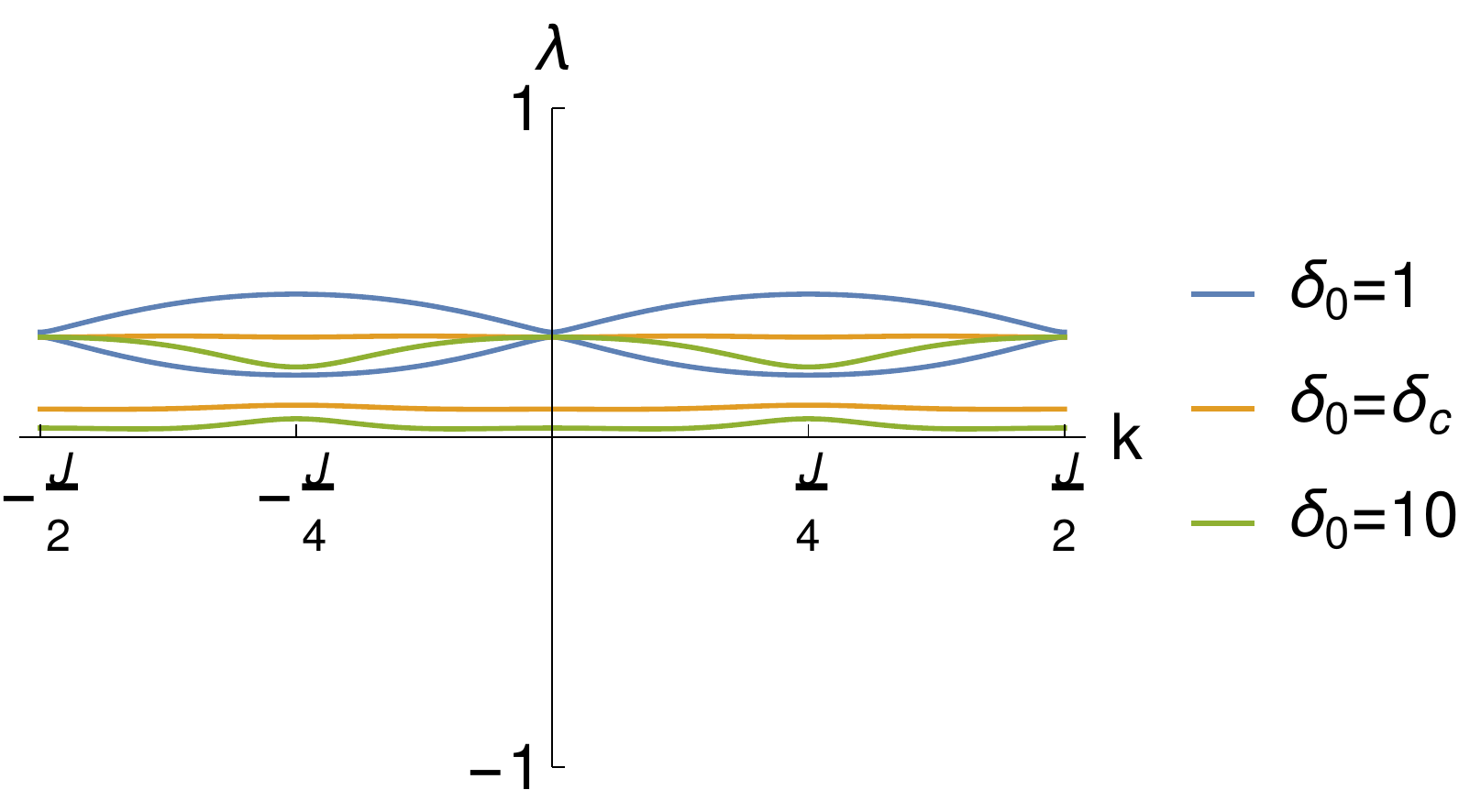}
    \caption{$\gamma=0.05<\gamma_c,\alpha=1$.}
    \label{fig:pointepst3}
  \end{subfigure}
  \caption{Spectrum of the iteration operator of algorithm
    \eqref{alg:precit} using a \emph{point} block-Jacobi smoother for a
    varying stabilization parameter $\dd$ of the SIPG method and reaction
    scaling $\gamma$.}
\end{figure}
shows the spectrum for penalization parameter $\dd=1$. We see that
there is a threshold on the physical parameter $\gamma$ where the
frequency $k$, at which the maximum absolute value of the eigenvalues
determining the spectral radius occurs, changes from $J/2$ to
$J/4$. The critical $\gamma$ can be computed by solving
$\lambda_+(\gamma)\big|_{k=J/2} = \lambda_+(\gamma)\big|_{k=J/4}$, and
it is given by
\begin{equation}
  \gamma_c(\dd) = \frac{1}{3 \left(\sqrt{4 (\dd-1) \dd+5}+(3-2
      \dd)\right)}.
\end{equation}
Similarly, Figures \ref{fig:pointepst2} and \ref{fig:pointepst3} show
the spectrum for $\gamma=0.5$ and $\gamma=0.05$. We see that there is
a threshold on $\dd$ where the frequency $k$, at which the maximum
absolute value of $\lambda_+$ occurs, changes from $J/2$ to $J/4$. The
critical $\dd$ can be computed as well by solving
$\lambda_+(\dd)\big|_{k=J/2} = \lambda_+(\dd)\big|_{k=J/4}$, and it is
given by
\begin{equation}
  \delta_c^+ = \scalemath{0.85}{\frac{-5 + 9 \gamma \left(6 \gamma^2+8
        \gamma+1\right) +\sqrt{\left(3 \gamma+1\right) \left(3 \gamma \left(12
            \gamma \left(3 \gamma \left(3 \gamma \left(3
                  \gamma+7\right)+20\right)+25\right)+53\right)+10\right)}} {6 \gamma
      (12 \gamma+5)}}
\end{equation}
for $\gamma > \gamma_c$, and
\begin{equation}
  \delta_c^- = \frac{1 + 2 \gamma \left(6 \gamma-11\right)-\sqrt{4
      \gamma \left(2 \gamma + 1\right) \left(3 \gamma \left(6 \gamma +
          7\right)+1\right)+1}}{8 \gamma \left(6 \gamma - 1\right)}
\end{equation}
for $\gamma \le \gamma_c$. This allows us to obtain
$\alpha_\text{opt}$ for different regimes: the equations to be solved
to minimize the spectral radius are
\begin{align}
  \begin{cases}
    \lambda_+\big|_{k=\frac{J}4} + \lambda_-\big|_{k=\frac{J}4} = 0 &
    \text{ for } \gamma \le \gamma_c,\dd \le \delta_c, \\
    \lambda_+\big|_{k=\frac{J}2} + \lambda_-\big|_{k=\frac{J}2} = 0 &
    \text{ for } \gamma \le \gamma_c,\dd > \delta_c \text{ or } \gamma >
    \gamma_c,\dd \le \delta_c, \\
    \lambda_+\big|_{k=\frac{J}4} + \lambda_-\big|_{k=\frac{J}2} = 0 &
    \text{ for } \gamma > \gamma_c,\dd > \delta_c,
  \end{cases}
\end{align}
which leads to the corresponding optimal relaxation parameters
\begin{align}
  \alpha_\text{opt}=
  \begin{cases} 
    \frac{8 \left(3 \gamma+1\right) \left(2 \dd \gamma+1\right)
      \left(3 (2 \dd-1) \gamma + 1\right)}{\left(12 \dd \gamma+5\right)
      \left(12 (2 \dd-1) \gamma^2+8 \dd \gamma + 1\right)}, & \text{ for }
    \gamma \le \gamma_c,\dd \le \delta_c^-, \\
    \frac{8 \left(3 \gamma + 1\right) \left(3 (2 \dd-1)
        \gamma+1\right)^2}{\left(6 \gamma + 1\right) \left(9 \gamma \left(4 (6
          (\dd-1) \dd+1) \gamma+8 \dd-5\right)+5\right)},&
    \begin{aligned} 
      &\text{ for } \gamma \le \gamma_c,\dd > \delta_c^- \\ 
      &\text{ or } \gamma > \gamma_c,\dd \le \delta_c^+,
    \end{aligned}\\
    \frac{4 \left(3 \gamma+1\right) \left(2 \dd \gamma+1\right)
      \left(3 (2 \dd-1) \gamma+1\right)}{\gamma \left(108 \dd (2 \dd-1)
        \gamma^2+6 (\dd (6 \dd+19)-8) \gamma+19 \dd+9\right)+2}, & \text{ for
    } \gamma > \gamma_c,\dd > \delta_c^+.
  \end{cases}
\end{align}
Figure \ref{fig:rdaoptpoint}
\begin{figure}
\includegraphics[width=0.49\textwidth]{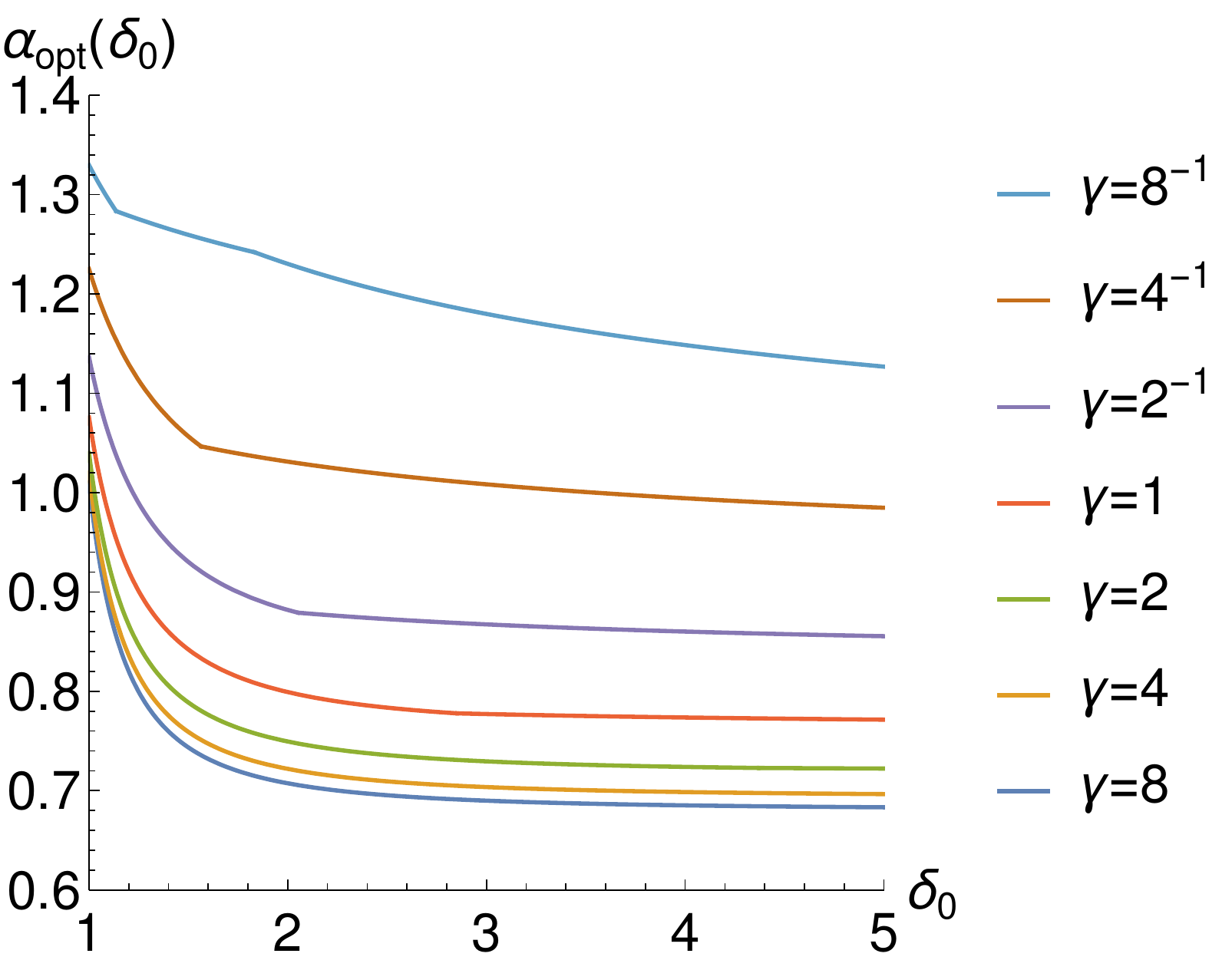}
\includegraphics[width=0.49\textwidth]{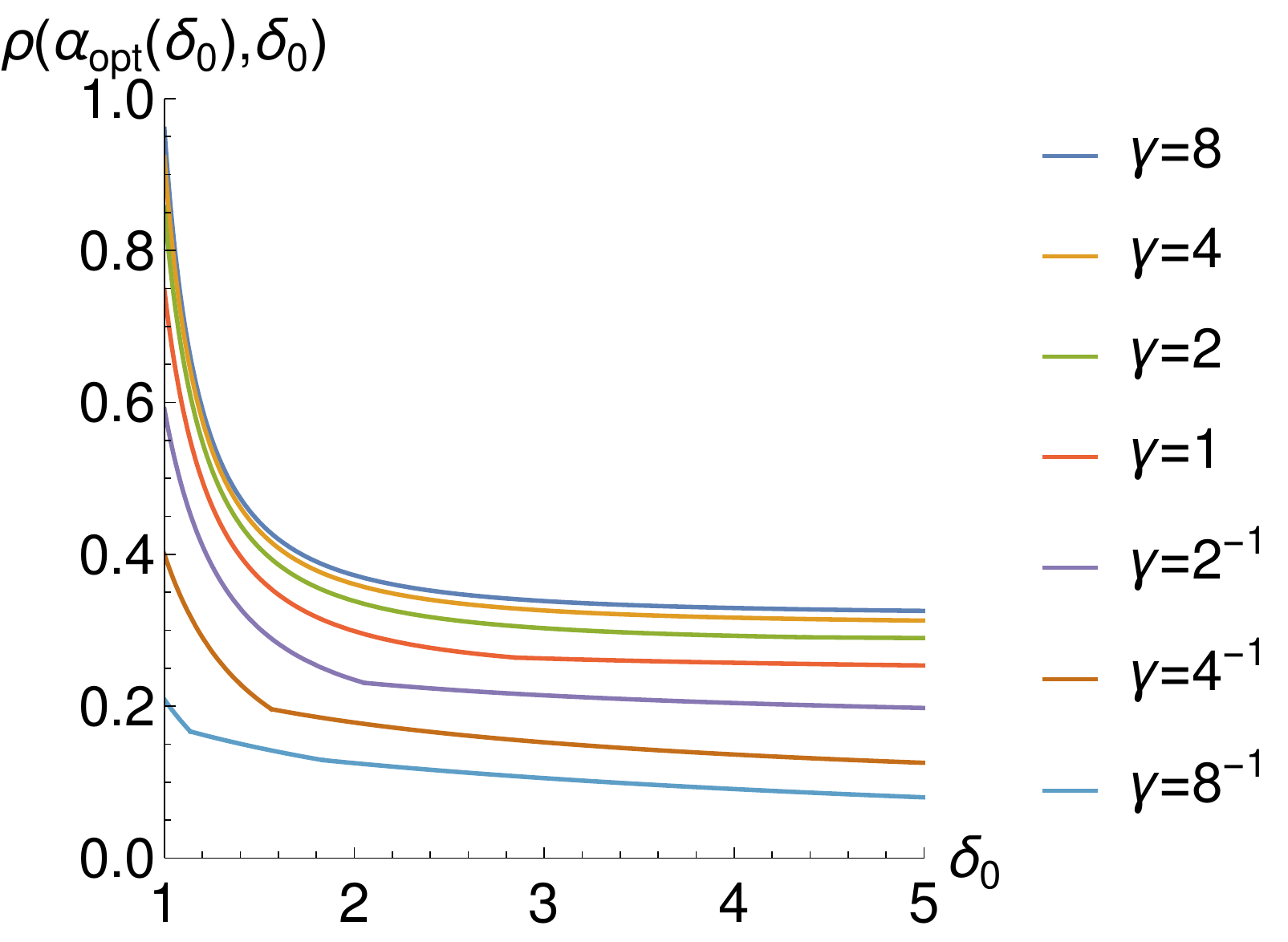}
\caption{Optimized relaxation parameter $\alpha_\text{opt}(\dd)$ and
  corresponding convergence factor of Algorithm \ref{alg:precit} using a
  \emph{point} block-Jacobi smoother as function of the stabilization
  parameter $\dd$ of the SIPG method for different reaction scalings
  $\gamma=\frac{\eps}{h^2}$.}
\label{fig:rdaoptpoint}
\end{figure}
shows the behavior of $\alpha_\text{opt}$ and the corresponding
convergence factor of the two-level method as a function of $\dd$ for
several values of the reaction scaling $\gamma =
\frac{\eps}{h^2}$. Note that $\lim_{\gamma \rightarrow \infty}
\delta_c^+ \rightarrow \infty$ and $\lim_{\gamma \rightarrow \infty}
\alpha_\text{opt} \rightarrow \frac{(2 \dd-1)^2}{6 \dd^2-6 \dd+1}$
(from the second expression), consistent with Theorem
\ref{thm:PoissonPoint}. We see from the right plot in Figure
\ref{fig:rdaoptpoint} that the \emph{point} block-Jacobi two-level
method is convergent for all $\dd > 1$ with the optimal choice
$\alpha_\text{opt}$, and the convergence factor remains below about
$0.4$ for penalization $\dd$ above 2, even when the reaction scaling
$\gamma$ becomes large, so the method is robust for large $\gamma$. We
also see from the left plot in Figure \ref{fig:rdaoptpoint} that
overrelaxation is needed (i.e. $\alpha_\text{opt}>1$), for typical
values of $\dd$ around 2, when $\gamma$ becomes small, but for
$\gamma$ large we need underrelaxation (i.e. $\alpha_\text{opt}<1$).

\subsubsection{\emph{Cell} block-Jacobi smoother}
By direct calculation, the eigenvalues of the iteration operator of
Algorithm \ref{alg:precit} for the reaction-diffusion equation case
using a \emph{cell} block-Jacobi smoother are of the form
\begin{equation}\label{eqn:rdcell}
  \lambda_\pm = \frac{c_1+c_2 x+c_3 x^2 \pm \sqrt{c_4 + c_5 x + c_6
      x^2 + c_7 x^3 + c_8 x^4}} {c_9 + c_{10} x + c_{11} x^2},
\end{equation}
where $x=\cos\left(\frac{4 \pi k}J \right)$, and the $c_1, \dots,
c_{11}$, depending on $\dd$, are defined in Appendix
\ref{apx:cellc}. Figures \ref{fig:cellepst}, \ref{fig:cellepst2},
\ref{fig:cellepst3} and \ref{fig:cellepst4}
\begin{figure}
  \centering
  \begin{subfigure}{0.49\textwidth}
    \includegraphics[width=\textwidth]{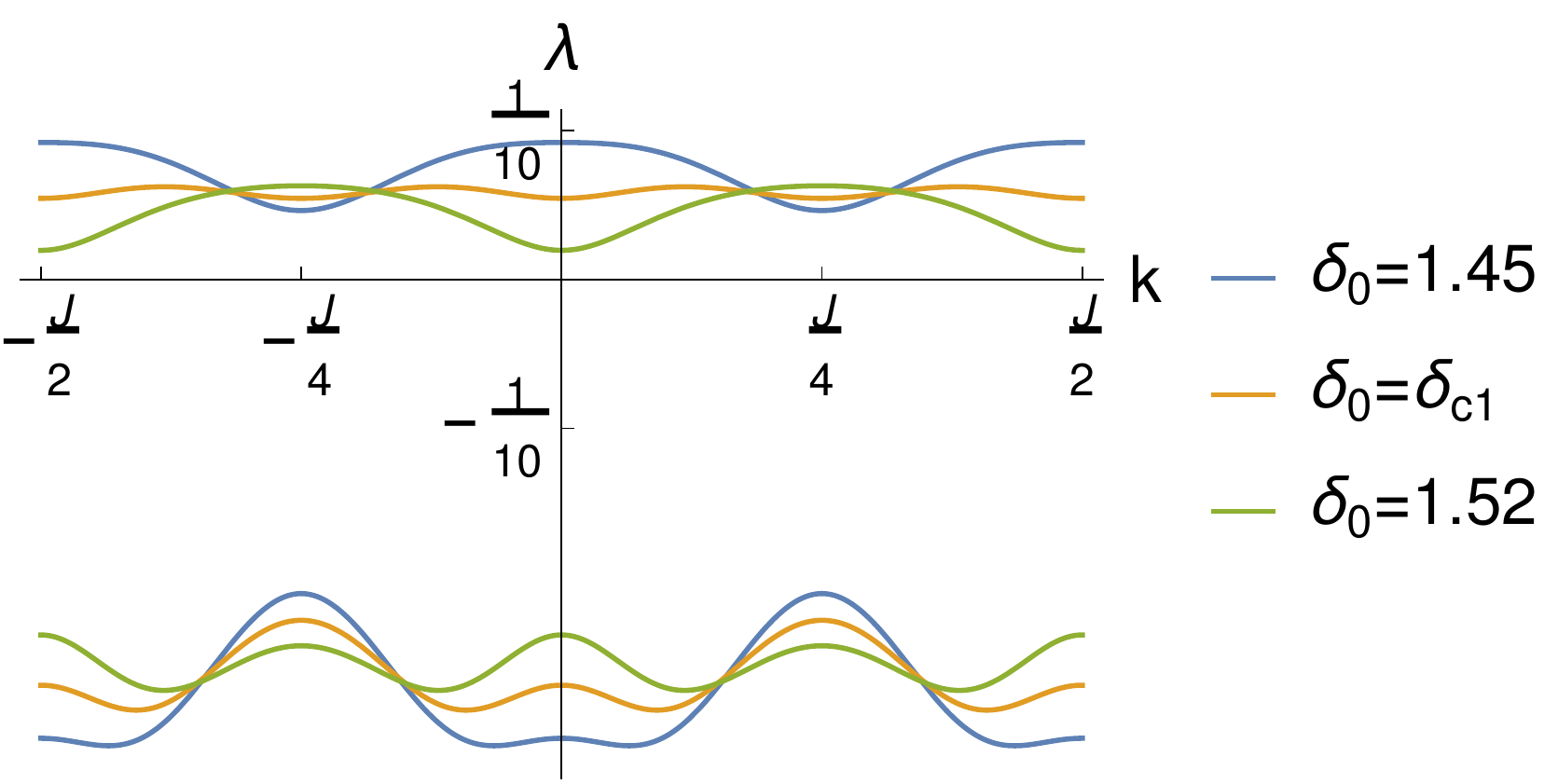}
    \caption{$\gamma=0.5,\alpha=1$.}
    \label{fig:cellepst}
  \end{subfigure}
  \begin{subfigure}{0.49\textwidth}
    \includegraphics[width=\textwidth]{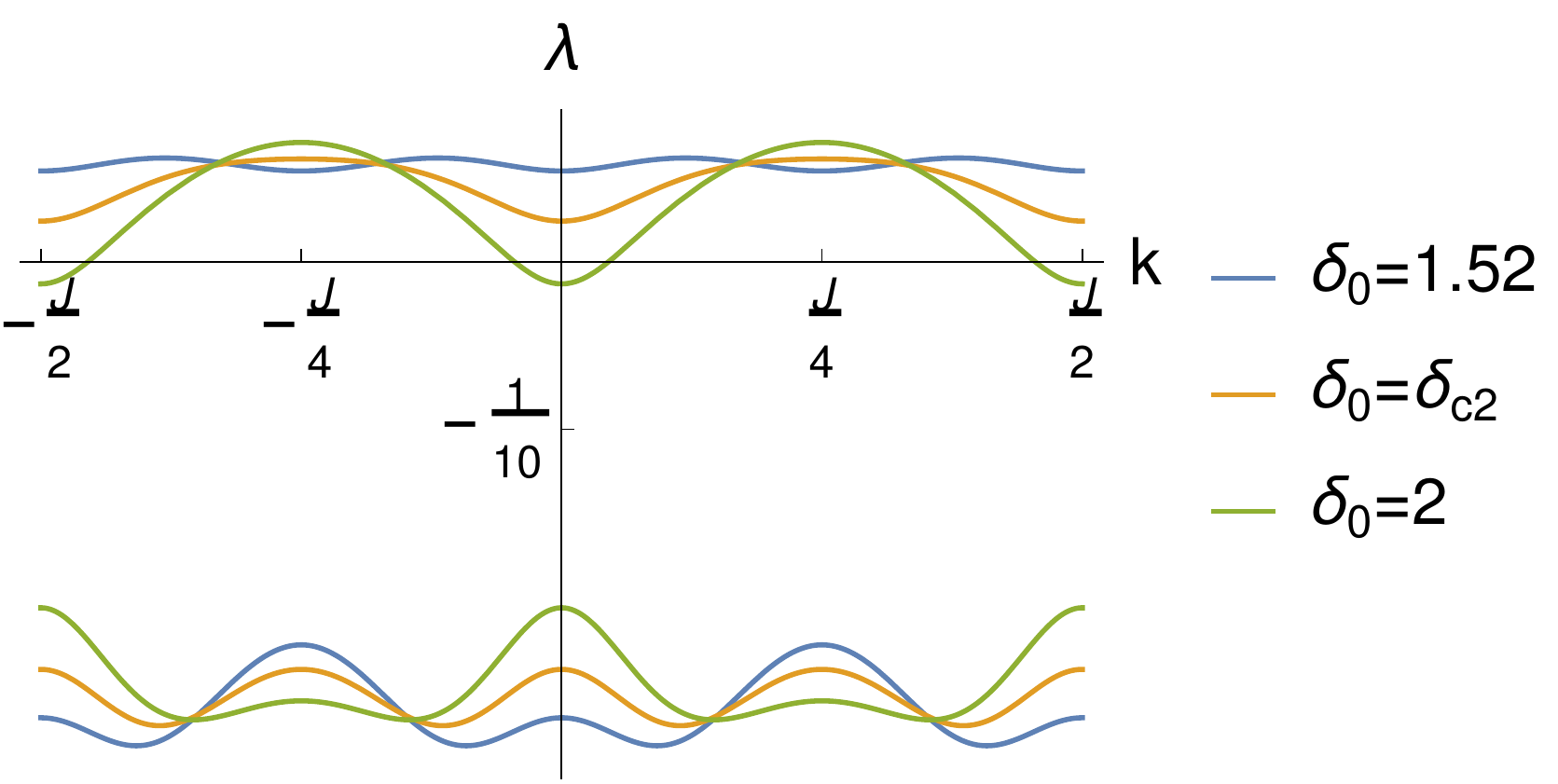}
    \caption{$\gamma=0.5,\alpha=1$.}
    \label{fig:cellepst2}
  \end{subfigure}
  \\
  \begin{subfigure}{0.49\textwidth}
    \includegraphics[width=\textwidth]{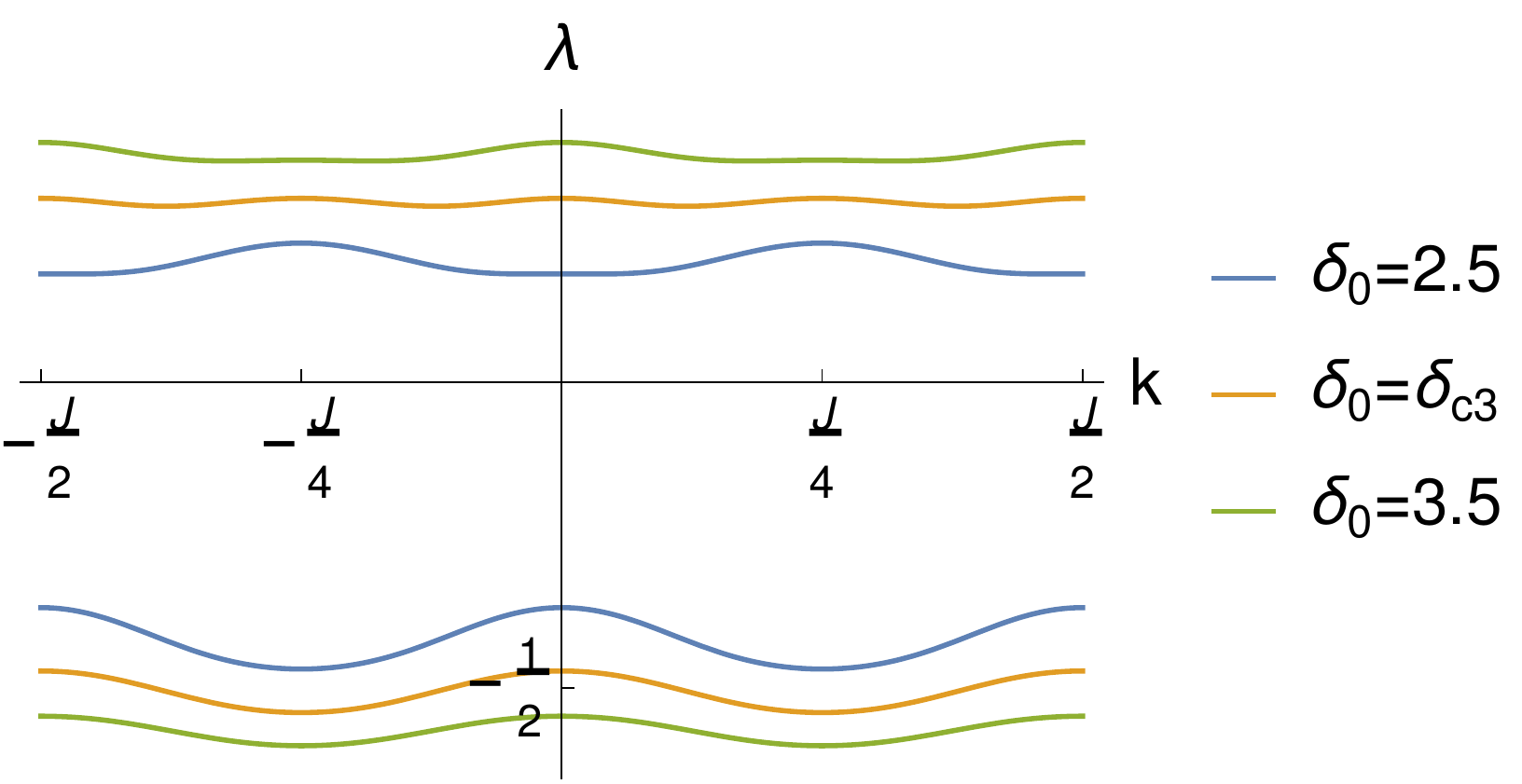}
    \caption{$\gamma=0.5,\alpha=1$.}
    \label{fig:cellepst3}
  \end{subfigure}
  \begin{subfigure}{0.49\textwidth}
    \includegraphics[width=\textwidth]{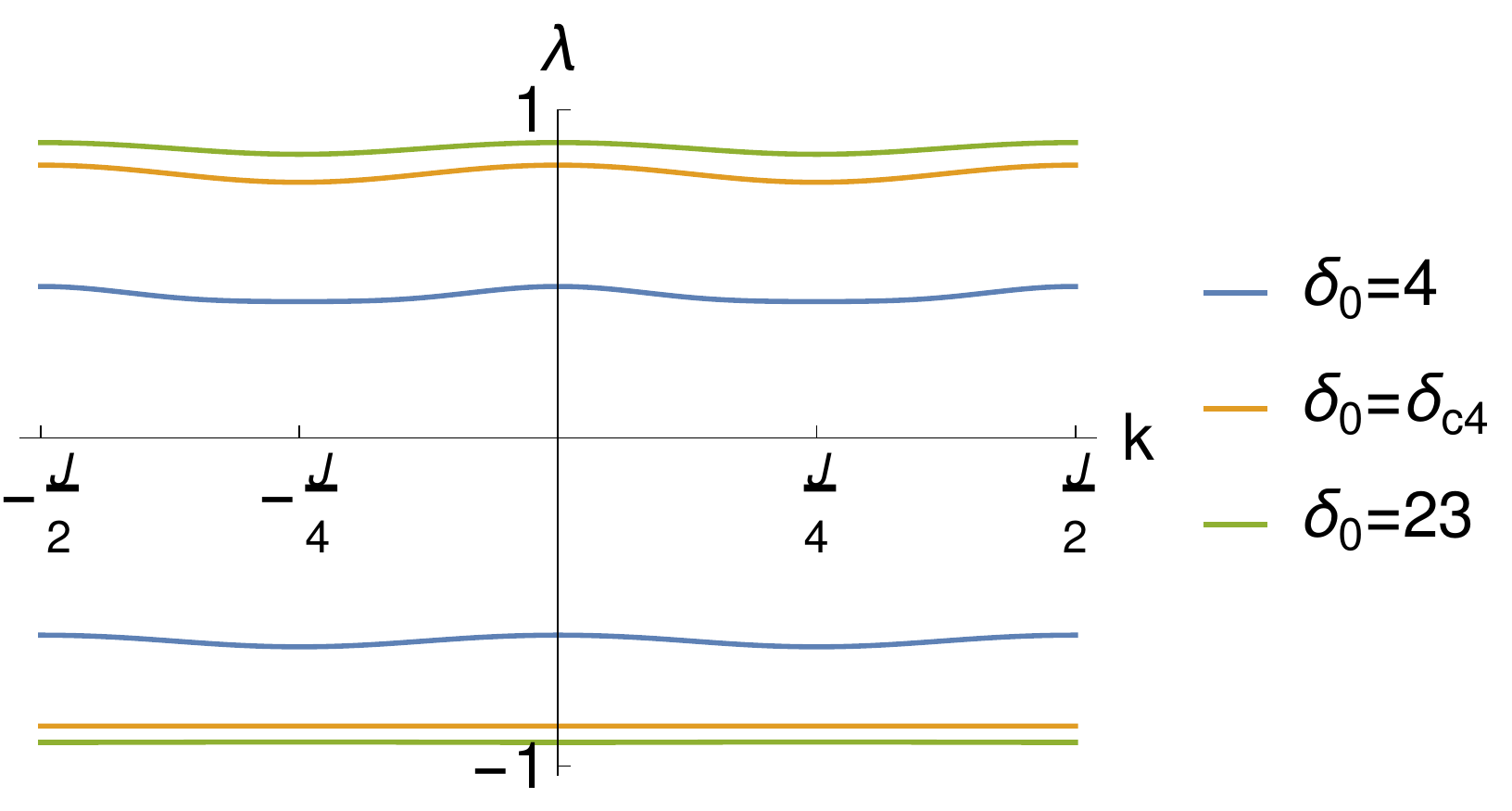}
    \caption{$\gamma=0.5,\alpha=1$.}
    \label{fig:cellepst4}
  \end{subfigure}
  \caption{Spectrum of the iteration operator of algorithm
    \eqref{alg:precit} using a \emph{cell} block-Jacobi smoother for a
    varying stabilization parameter $\dd$ of the SIPG method and reaction
    scaling $\gamma \ge \gamma_c$.}
\end{figure}
show the spectrum of the iteration operator of Algorithm
\ref{alg:precit} for $\gamma=\frac12$. We can see that, in contrast to
the case of the Poisson equation, the maxima and minima are not
located only at $0,J/4,J/2$, however we approximate the behavior
optimizing by considering only the values at $0,J/4,J/2$. Therefore,
in order to equioscillate the spectrum we see that the following
equations need to hold:
\begin{align}\label{eqn:lambdabigeps}
  \begin{cases}
    \lambda_+\big|_{k=\frac{J}2} + \lambda_-\big|_{k=\frac{J}2} = 0, & 
    \text{ for } \dd \le \delta_{c1} 
    \text{ or } \dd \ge \delta_{c4},
    \\
    \lambda_+\big|_{k=\frac{J}4} + \lambda_-\big|_{k=\frac{J}2} = 0, &
    \text{ for } \dd \le \delta_{c2},\\
    \lambda_+\big|_{k=\frac{J}4} + \lambda_-\big|_{k=\frac{J}4} = 0, &
    \text{ for } \dd \le \delta_{c3},\\
    \lambda_+\big|_{k=\frac{J}2} + \lambda_-\big|_{k=\frac{J}4} = 0, &
    \text{ for } \dd \le \delta_{c4},
  \end{cases}
\end{align}
where
\begin{align*}
  \delta_{c1} =&
                 \begin{aligned}
                   &-\frac{1}{36 \gamma^2}\left(4 \gamma \left(1-6
                       \gamma\right)+\xi(\gamma)+\frac{\gamma^2 \left(12 \gamma \left(12
                           \gamma+5\right)+1\right)}{\xi(\gamma)}\right)
                 \end{aligned}, \\
  \delta_{c2} =& \frac{-3 + 36 \gamma^2+2 \gamma+\sqrt{4 \gamma
                 \left(3 \gamma \left(4 \gamma \left(27 \gamma + 35\right) + 65\right)
                 + 37\right)+9}}{16 \gamma \left(3 \gamma + 1\right)}, \\
  \delta_{c3} =& 2 \gamma +2, \\
  \delta_{c4} =& 3 \left(6 \gamma^2+4 \gamma +1\right).
\end{align*}
with $\xi(\gamma)=\scalemath{0.65}{\gamma \sqrt[3]{3 \sqrt{ 3 \left(12
        \gamma \left(27 \gamma \left(8 \gamma \left(\gamma \left(6 \gamma
                \left(33
                  \gamma+46\right)+155\right)+44\right)+51\right)+89\right)+25\right)}-2
    \left(3 \gamma+1\right) \left(12 \gamma \left(57
        \gamma+20\right)+13\right)}}.$

We observe that at $\gamma = \gamma_c = 0.16607\dots$ we have
$\delta_{c1}(\gamma) = \delta_{c2}(\gamma)$. For $\gamma \le
\gamma_c$, we have $\delta_{c2} \le \delta_{c1} \le \delta_{c3} \le
\delta_{c4}$, which means that the distribution of critical values of
$\dd$ changes and we have to perform again the same equioscillation
analysis as we did previously.

Figures \ref{fig:cellepstb}, \ref{fig:cellepstb2},
\ref{fig:cellepstb3} and \ref{fig:cellepstb4}
\begin{figure}
  \centering
  \begin{subfigure}{0.49\textwidth}
    \includegraphics[width=\textwidth]{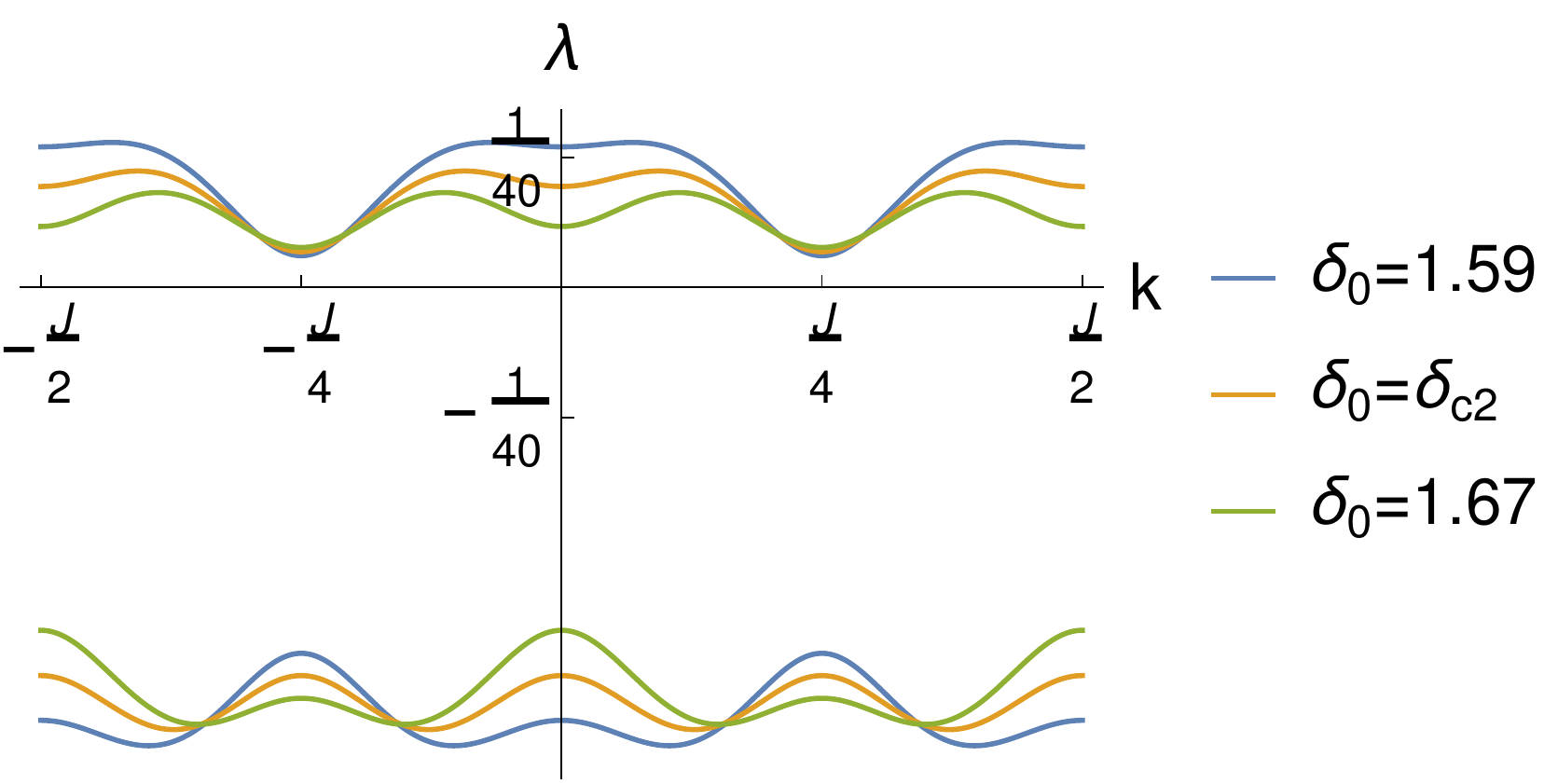}
    \caption{$\gamma=0.05,\alpha=1$.}
    \label{fig:cellepstb}
  \end{subfigure}
  \begin{subfigure}{0.49\textwidth}
    \includegraphics[width=\textwidth]{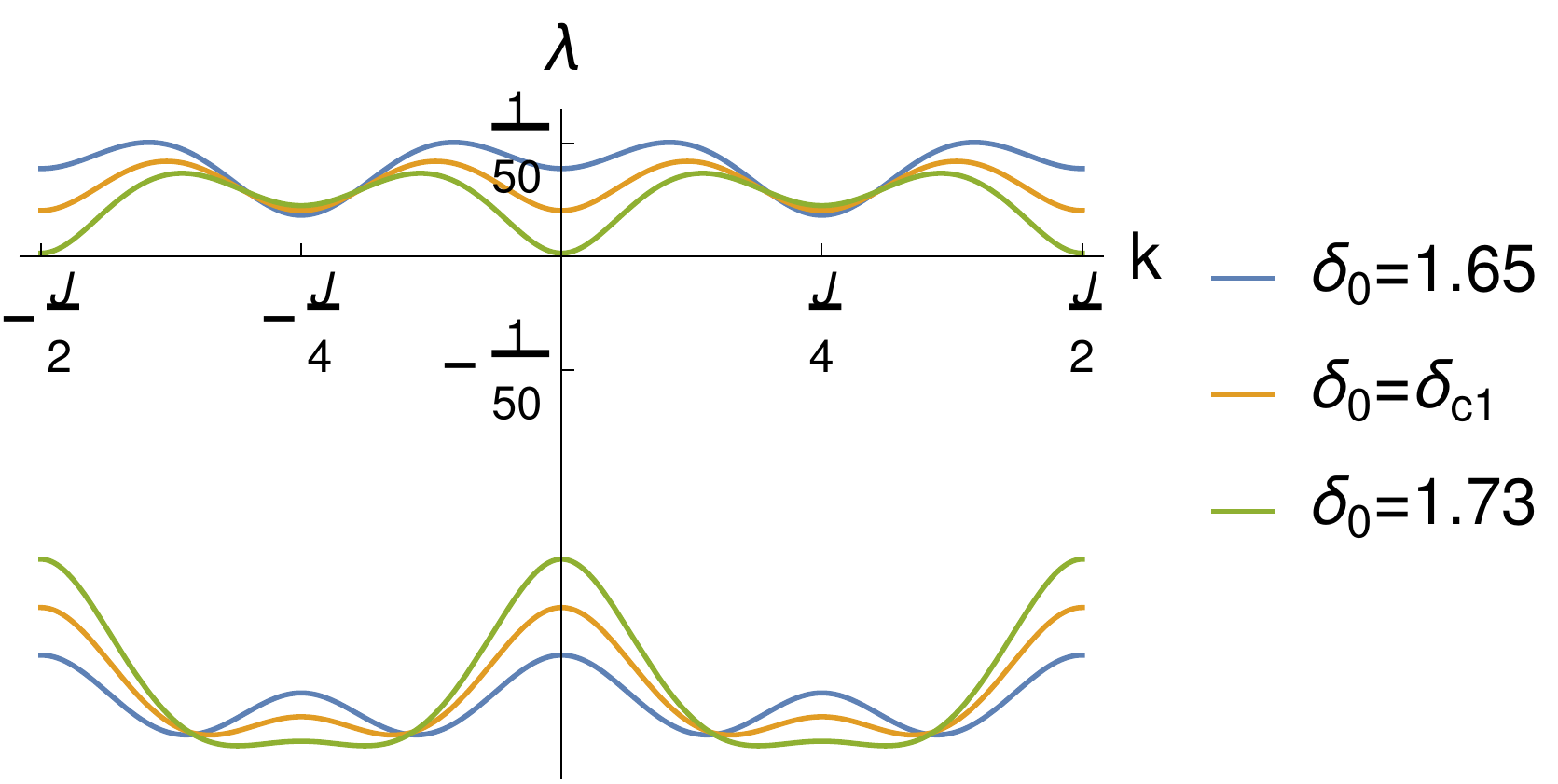}
    \caption{$\gamma=0.05,\alpha=1$.}
    \label{fig:cellepstb2}
  \end{subfigure}
  \\
  \begin{subfigure}{0.49\textwidth}
    \includegraphics[width=\textwidth]{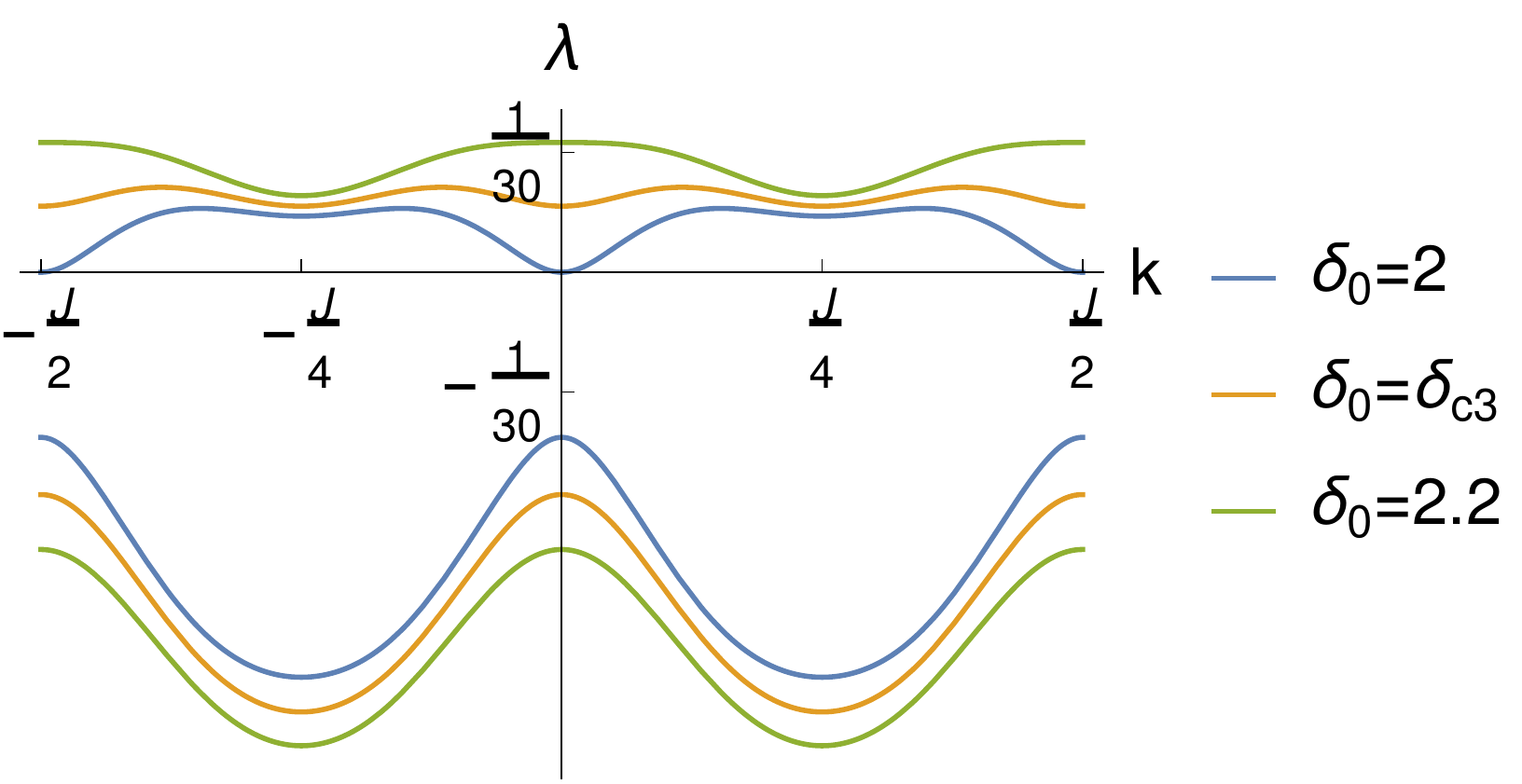}
    \caption{$\gamma=0.05,\alpha=1$.}
    \label{fig:cellepstb3}
  \end{subfigure}
  \begin{subfigure}{0.49\textwidth}
    \includegraphics[width=\textwidth]{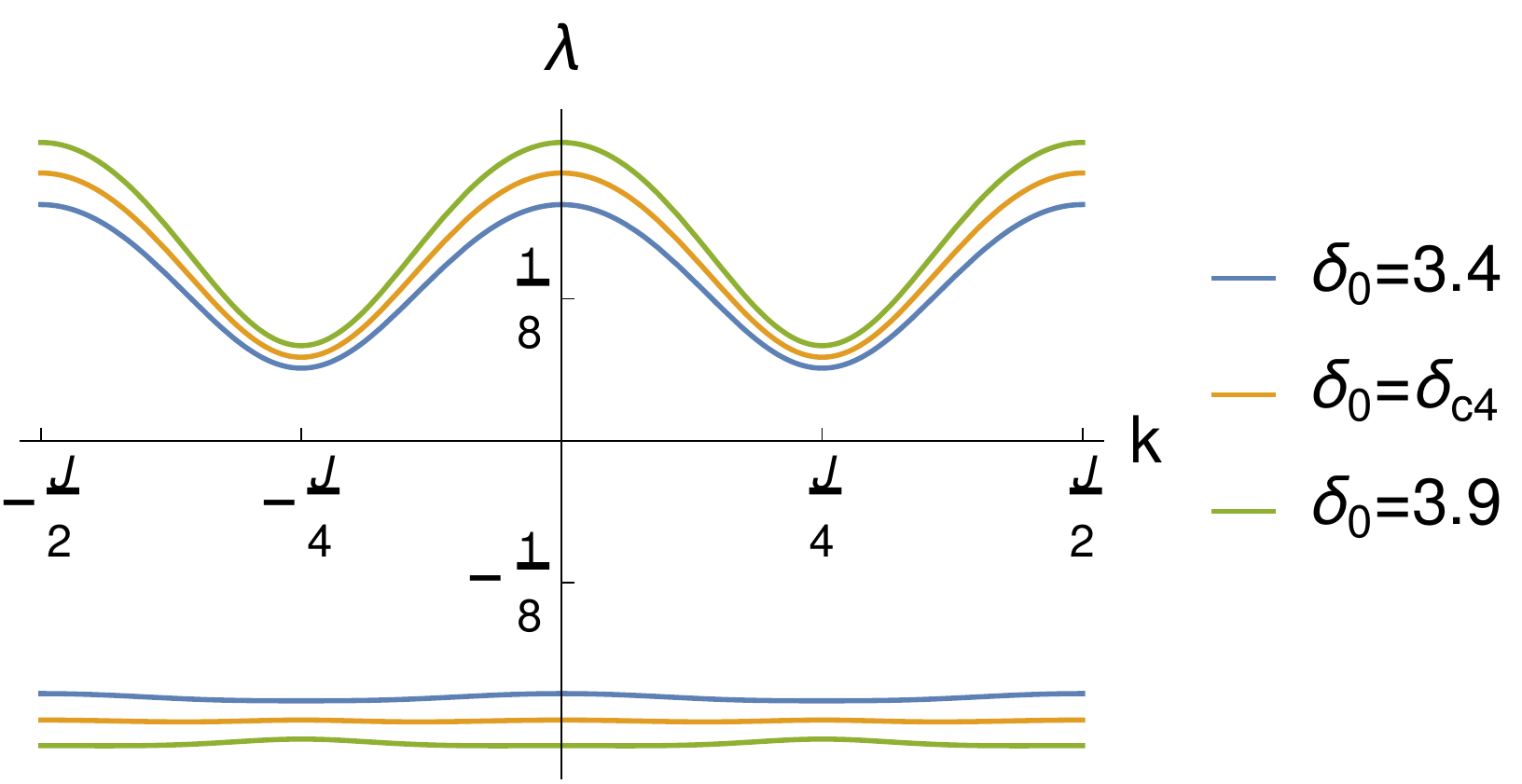}
    \caption{$\gamma=0.05,\alpha=1$.}
    \label{fig:cellepstb4}
  \end{subfigure}
  \caption{Spectrum of the iteration operator of algorithm
    \eqref{alg:precit} using a \emph{cell} block-Jacobi smoother for a
    varying stabilization parameter $\dd$ of the SIPG method and reaction
    scaling $\gamma \le \gamma_c$.}
\end{figure}
show the spectrum of the iteration operator of algorithm
\eqref{alg:precit} for $\gamma=\frac1{20}$. In order to center the
spectrum we see that the following equations need to hold:
\begin{align}\label{eqn:lambdasmalleps}
  \begin{cases}
    \lambda_+\big|_{k=\frac{J}2} + \lambda_-\big|_{k=\frac{J}2} = 0, & 
    \text{ for } \dd \le \delta_{c2} 
    \text{ or } \dd \ge \delta_{c4},\\
    \lambda_+\big|_{k=\frac{J}2} + \lambda_-\big|_{k=\frac{J}4} = 0, &
    \text{ for } \dd \le \delta_{c1},\\
    \lambda_+\big|_{k=\frac{J}4} + \lambda_-\big|_{k=\frac{J}4} = 0, &
    \text{ for } \dd \le \delta_{c3},\\
    \lambda_+\big|_{k=\frac{J}2} + \lambda_-\big|_{k=\frac{J}4} = 0, &
    \text{ for } \dd \le \delta_{c4}.
  \end{cases}
\end{align}
Following equations \eqref{eqn:lambdabigeps} and
\eqref{eqn:lambdasmalleps}, the optimal relaxation parameter is
\begin{align}
  \alpha_\text{opt} = 
  \begin{cases}
    \frac{2 \left(2 \dd \gamma +1\right) \left(6 \dd \gamma+1\right)
      \left(3 (2 \dd-1) \gamma+1\right)}{3 \gamma \left(24 \dd \left(2
          \dd^2-1\right) \gamma^2+2 \left(18 \dd^2+\dd-6\right) \gamma+9
        \dd-1\right)+2},
    \begin{cases}
      \text{ for } \gamma \ge \gamma_c, 1 \le \dd \le \delta_{c1}, \\
      \text{ or } \gamma \ge \gamma_c, \dd \ge \delta_{c4}, \\
      \text{ or } \gamma \le \gamma_c, 1 \le \dd \le \delta_{c2}, \\
      \text{ or } \gamma \le \gamma_c, \dd \ge \delta_{c4},
    \end{cases} \\
    \frac{\left(2 \dd \gamma+1\right) \left(6 \dd
        \gamma+1\right)}{\gamma \left(6 (4 \dd-1) \gamma+5 \dd+6\right)+1},
    \text{ for } \gamma \ge \gamma_c, \delta_{c1} \le \dd \le
    \delta_{c2}.\\
    \frac{\left(3 \gamma+1\right) \left(2 \dd \gamma+1\right) \left(6
        \dd \gamma+1\right) \left(3 (2 \dd-1) \gamma+1\right)}{3 \gamma
      \left(18 \dd (8 (\dd-1) \dd+1) \gamma^3+6 (4 \dd (2 \dd (\dd+1)-3)+1)
        \gamma^2+(\dd (31 \dd-6)-8) \gamma+6 \dd-2\right)+1}, \\
    \text{ for } \gamma \le \gamma_c, \delta_{c2} \le \dd \le
    \delta_{c1},\\
    \frac{2 \left(3 \gamma+1\right) \left(2 \dd \gamma+1\right)
      \left(6 \dd \gamma+1\right)}{\left(3 (\dd+1) \gamma+2\right) \left(12
        (2 \dd-1) \gamma^2+8 \dd \gamma+1\right)},
    \begin{cases}
      \text{ for } \gamma \ge \gamma_c, \delta_{c2} \le \dd \le \delta_{c3}, \\
      \text{ or } \gamma \le \gamma_c, \delta_{c1} \le \dd \le \delta_{c3},
    \end{cases} \\
    \frac{2 \left(3 \gamma+1\right) \left(2 \dd \gamma+1\right)
      \left(6 \dd \gamma + 1 \right)}{\gamma \left(36 \dd (2 \dd+1)
        \gamma^2+6 (\dd (4 \dd+9)+4) \gamma+13 \dd+15\right)+2},
    \scalemath{0.9}{\begin{cases}
        \text{ for } \gamma \ge \gamma_c, \delta_{c3} \le \dd \le \delta_{c4}, \\
        \text{ or } \gamma \le \gamma_c, \delta_{c3} \le \dd \le \delta_{c4}.
      \end{cases}}
  \end{cases}.
\end{align}
Figure \ref{fig:rdaoptcell}
\begin{figure}
  \includegraphics[width=0.49\textwidth]{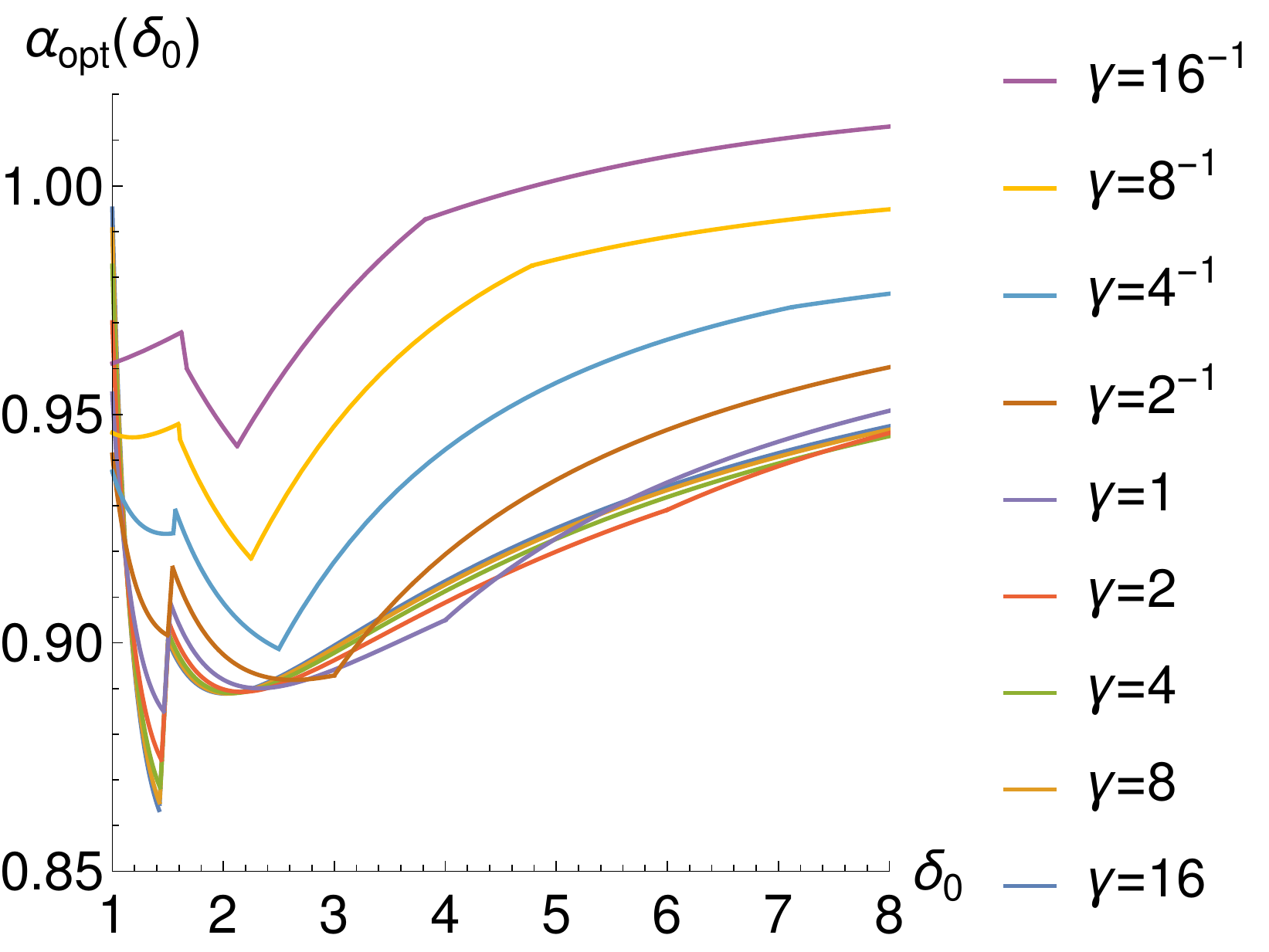}
  \includegraphics[width=0.49\textwidth]{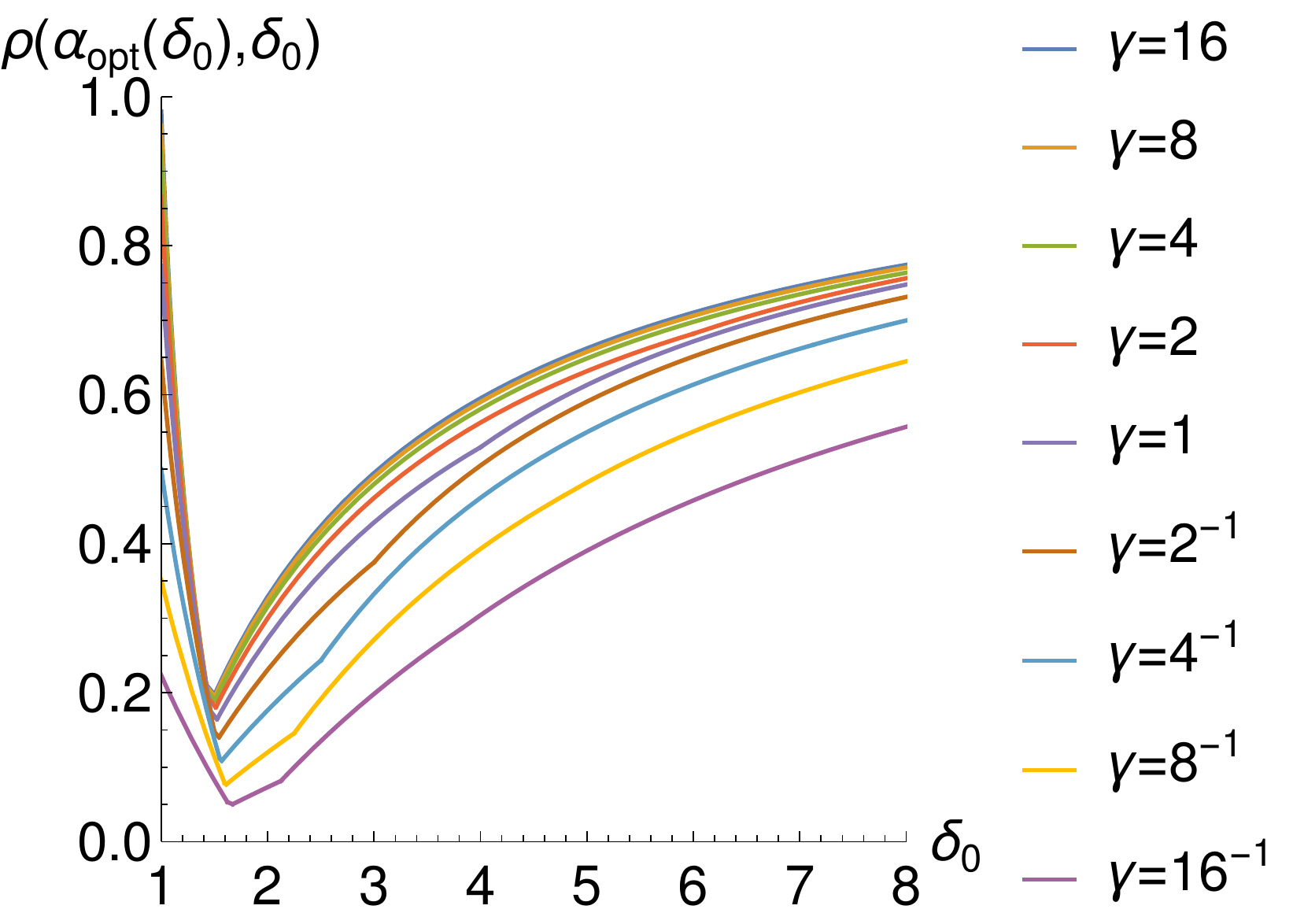}\\
  \caption{Optimized relaxation parameter $\alpha_\text{opt}(\dd)$ and
    corresponding convergence factor of Algorithm \ref{alg:precit} using
    a \emph{cell} block-Jacobi smoother as function of the stabilization
    parameter $\dd$ of the SIPG method for different reaction scalings
    $\gamma=\frac{\eps}{h^2}$.}
  \label{fig:rdaoptcell}
\end{figure}
shows the behavior of $\alpha_\text{opt}$ and the corresponding
convergence factor of the two-level method as a function of $\dd$ for
several values of the reaction scaling $\gamma=\frac{\eps}{h^2}$. From
the left plot in Figure \ref{fig:rdaoptcell}, we see that it would be
quite difficult to guess a good choice of the relaxation parameter
$\alpha$ without analysis. From the right plot in Figure
\ref{fig:rdaoptcell}, we see that the \emph{cell} block-Jacobi two
level method is also convergent for all values of the penalization
parameter $\dd>1$ and reaction scaling $\gamma$ when using the
optimized relaxation parameter $\alpha_\text{opt}$, and it has much
better convergence properties for moderate sizes of the penalization
parameter $\dd$ around 2 than the \emph{point} block-Jacobi two-level
method from Figure \ref{fig:rdaoptpoint}. However convergence is worse
for larger sizes of the penalization parameter $\dd$ than for the
\emph{point} block-Jacobi two-level method. We also see from the left
plot in Figure \ref{fig:rdaoptcell} that overrelaxation can become
necessary when the penalization parameter $\dd$ becomes large,
especially when $\gamma$ is small.

As in the case of Laplace's equation, we see that we obtain the best
performance for $\dd$ around $\frac32$, shown in Figure
\ref{fig:rdaoptcell} as the minimum of the curves on the right, and
this depends only little on the reaction scaling $\gamma$. This shows
that also in the reaction-diffusion case, choosing the penalization
parameter in SIPG wisely can make the associated iterative solver much
faster than just choosing it large enough, even with optimized
relaxation parameter $\alpha$!

\section{Numerical experiments}\label{NumSec}

We now show by numerical experiments that the expressions we obtained,
though quite lengthy in the reaction-diffusion case, are indeed very
good approximations of the optimal relaxation parameters, as a
function of the penalization parameter $\dd$ and in the reaction case
the reaction scaling $\gamma=\frac{\eps}{h^2}$.  To do so, we assemble
the system matrix on a uniform 64-element mesh, with Dirichlet
boundary conditions, and compute numerically the spectral radii of the
two-level operators using the QR method, as implemented in LAPACK
3.6.0, accessed with Python 3.5.2.

\subsection{\emph{Point} block-Jacobi smoother for the Poisson
equation} The dotted lines in Figure \ref{fig:ExpPoint}
\begin{figure}
  \begin{subfigure}[t]{0.49\textwidth}
    \includegraphics[width=\textwidth]{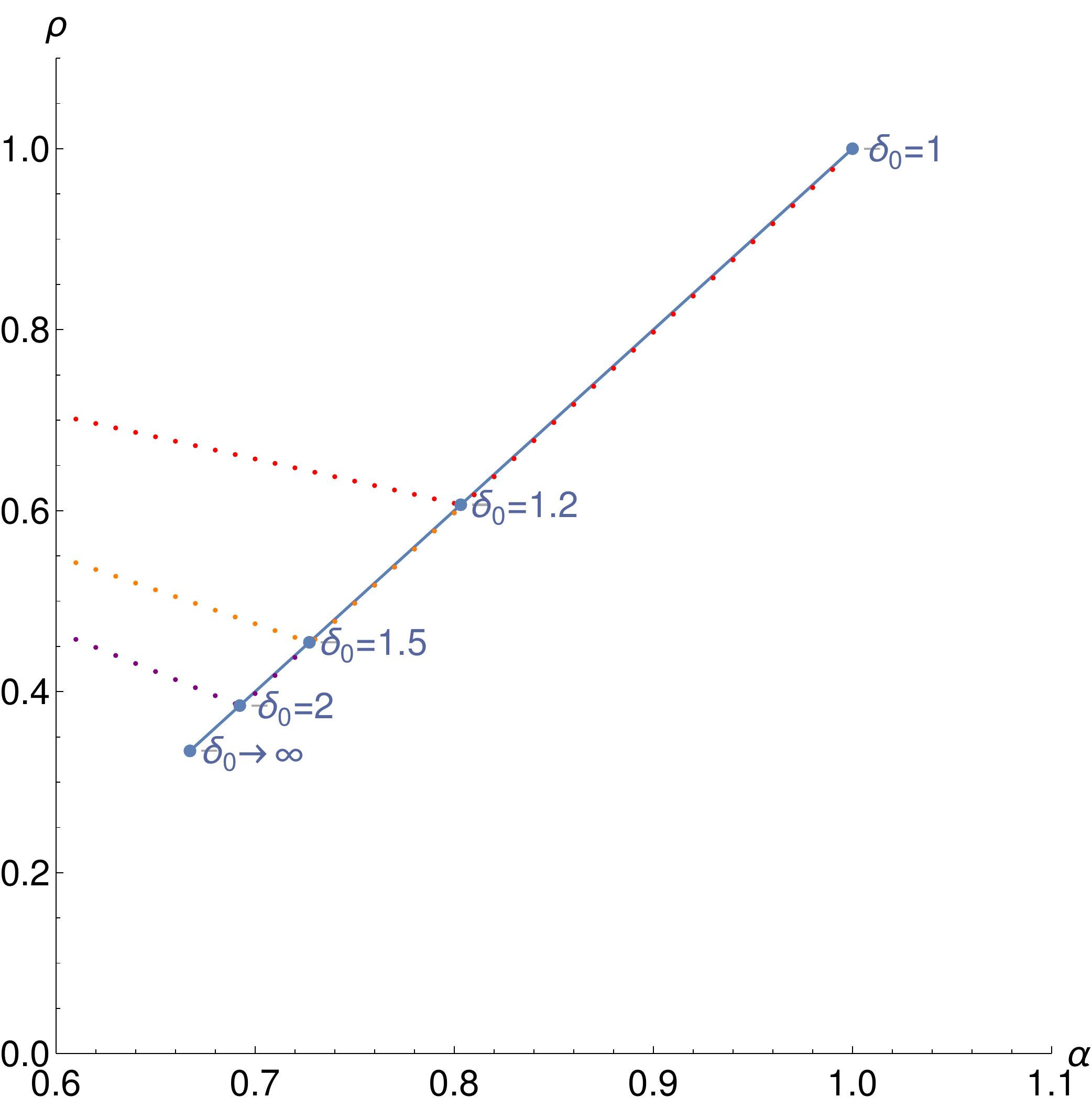}
    \caption{Numerically computed spectral radius using a \emph{point}
      block-Jacobi smoother to solve the Poisson equation. Red points:
      $\dd=1.2$, orange points: $\dd=1.5$, purple points:
      $\dd=2$. Blue points and blue line: predicted theoretically
      optimized spectral radius $\rho(\alpha_\text{opt})$.}
    \label{fig:ExpPoint}
  \end{subfigure}
  \begin{subfigure}[t]{0.49\textwidth}
    \includegraphics[width=\textwidth]{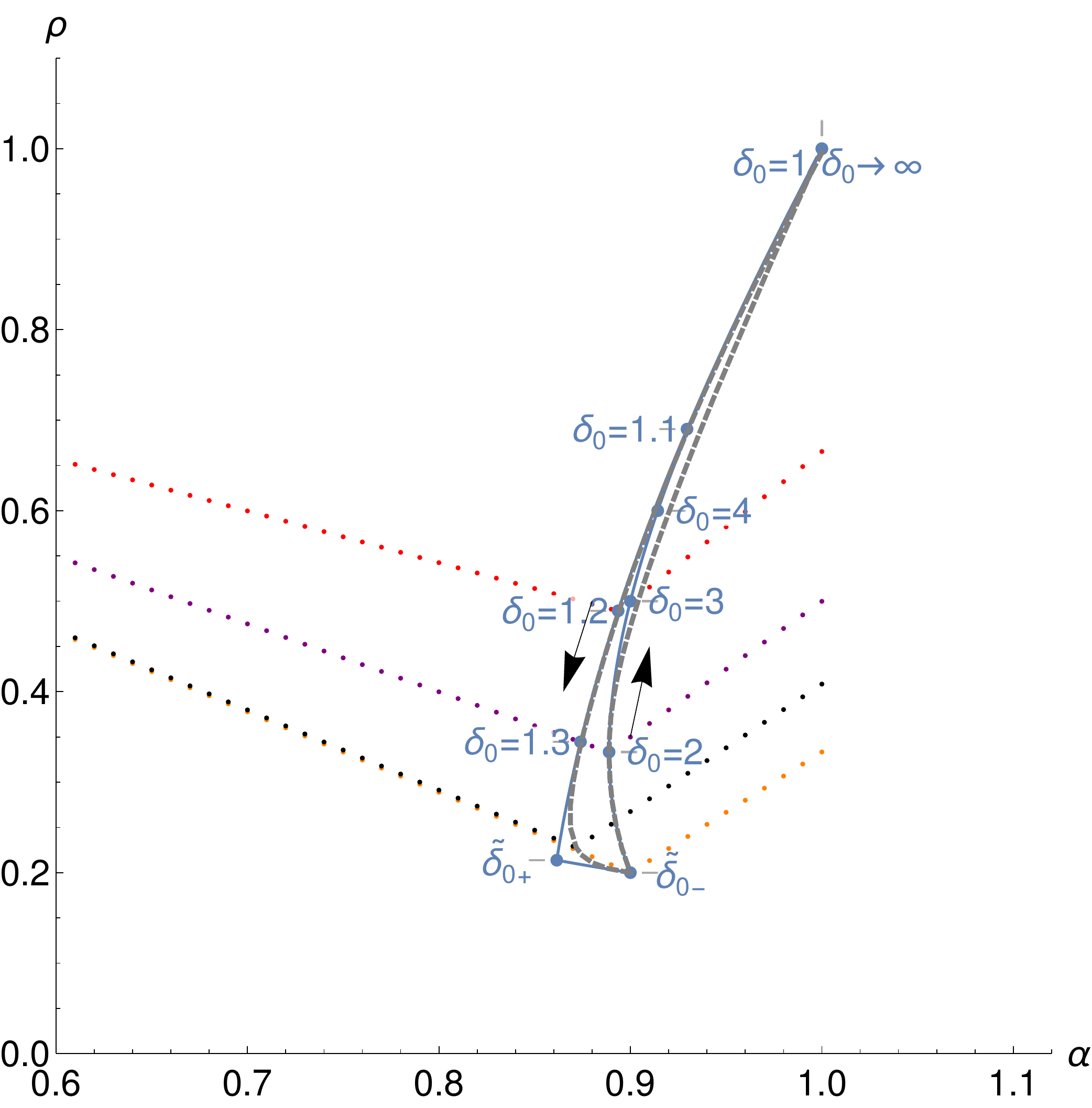}
    \caption{Numerically computed spectral radius using a \emph{cell}
      block-Jacobi smoother to solve the Poisson equation. Red points:
      $\dd=1.2$, orange points: $\dd=\widetilde{\dd}_-$, purple
      points: $\dd=2$, black points: $\dd=\widetilde{\dd}_+$. Dashed
      blue: entire curve of numerically computed optimized spectral
      radii. Solid blue: predicted theoretically optimized spectral
      radii $\rho(\alpha_\text{opt})$.}
    \label{fig:ExpCell}
  \end{subfigure}
  \caption{}
\end{figure}
are numerically computed spectral radii $\rho$ vs. relaxation
parameter $\alpha$ for $\dd = 1.2$ (red), for $\dd = 1.5$ (orange) and
for $\dd = 2$ (purple) for the two-level method with the \emph{point}
block-Jacobi smoother.  We see that they all attain a minimum value
giving fastest convergence, which coincides with the theoretical
prediction of Theorem \ref{thm:PoissonPoint} marked with blue dots and
a label indicating the value of $\dd$ used. We also added a
theoretical blue dot for $\dd = 1$ (top right) and $\dd \rightarrow
\infty$ (bottom left), and the entire theoretically predicted
parametric line $\rho(\alpha_\text{opt}(\dd),\dd)$, also in blue with
$\alpha_\text{opt}(\dd)$ from Theorem \ref{thm:PoissonPoint}. We see
that our theoretical result based on the typical LFA assumption of
periodic boundary conditions predicts the performance with Dirichlet
boundary conditions very well. One might be tempted to use large
values of $\dd$ in order to have as small a spectral radius as
possible, but for large $\dd$, the coarse problem is more difficult to
solve because the $\dd$ is doubled as we showed in
\S\ref{subsec:coarse} and the condition number of the unpreconditioned
coarse operator grows. It would be interesting to investigate if the
capacity of this smoother to deal with large values of $\dd$ can be
used to our advantage in a multigrid setting.

\subsection{\emph{Cell} block-Jacobi smoother for the Poisson
  equation}
The dotted lines in Figure \ref{fig:ExpCell} are numerically computed
spectral radii $\rho$ vs. relaxation parameter $\alpha$ for $\dd =
1.2$ (red), $\dd =\widetilde{\dd}_+ \approx 1.41964$ (black), $\dd
=\widetilde{\dd}_- = 1.5$ (orange) and $\dd = 2$ (purple) for the two
level method with the \emph{cell} block-Jacobi smoother. Like for the
\emph{point} block-Jacobi smoother they all attain a minimum value
which gives fastest convergence. With blue dots, we mark the
theoretical predictions of Theorem \ref{thm:PoissonCell}, also for a
few more values of $\dd\in\{1,1.1,1.3,4,\infty\}$. In contrast to the
\emph{point} block-Jacobi smoother case, the two values $\dd=1$ and
$\dd=\infty$ lead to the same point on the curve at the top right,
which shows that this method also deteriorates when $\dd$ becomes
large.  We also plot the entire theoretically predicted parametric
line $\rho(\alpha_\text{opt}(\dd),\dd)$ in solid blue with
$\alpha_\text{opt}(\dd)$ from Theorem \ref{thm:PoissonCell} and the
corresponding numerically determined one in dashed blue \footnote{We
did not plot this dashed line for the \emph{point} block-Jacobi
smoother case in Figure \ref{fig:ExpPoint}, since it would not have
been visible under the predicted line.}. This shows that the
theoretical prediction is very accurate, except for values around
$\dd\approx\widetilde{\dd}_+$ where there is a small difference.  We
checked that this is due to the Dirichlet boundary conditions, by
performing numerical experiments using periodic boundary conditions
which made the results match the predicted line. We also observed that
the dashed line approaches the predicted line when decreasing the mesh
size. Therefore, even though Theorem \ref{thm:PoissonCell} was
obtained with the typical LFA assumption of periodic boundary
conditions, the predictions are again very good also for the Dirichlet
case. Note that in contrast to the \emph{point} block-Jacobi case,
where best performance is achieved for large $\dd$, for \emph{cell}
block-Jacobi the best performance is achieved for
$\dd=\widetilde{\dd}_-$, and convergence is almost twice as fast as
for \emph{point} block-Jacobi with a similar value for $\dd$. Clearly,
also in practice, the DG penalization parameter influences very much
the performance of the two-level solver, even when using the best
possible relaxation parameter.

\subsection{\emph{Point} block-Jacobi smoother for the reaction-diffusion
  equation}
Results for the solution of a reaction-diffusion equation using a
two-level method with the \emph{point} block-Jacobi smoother are shown
in Figure \ref{fig:ExpRDPoint}.
\begin{figure}
  \centering
  \begin{subfigure}{0.49\textwidth}
    \includegraphics[width=\textwidth]{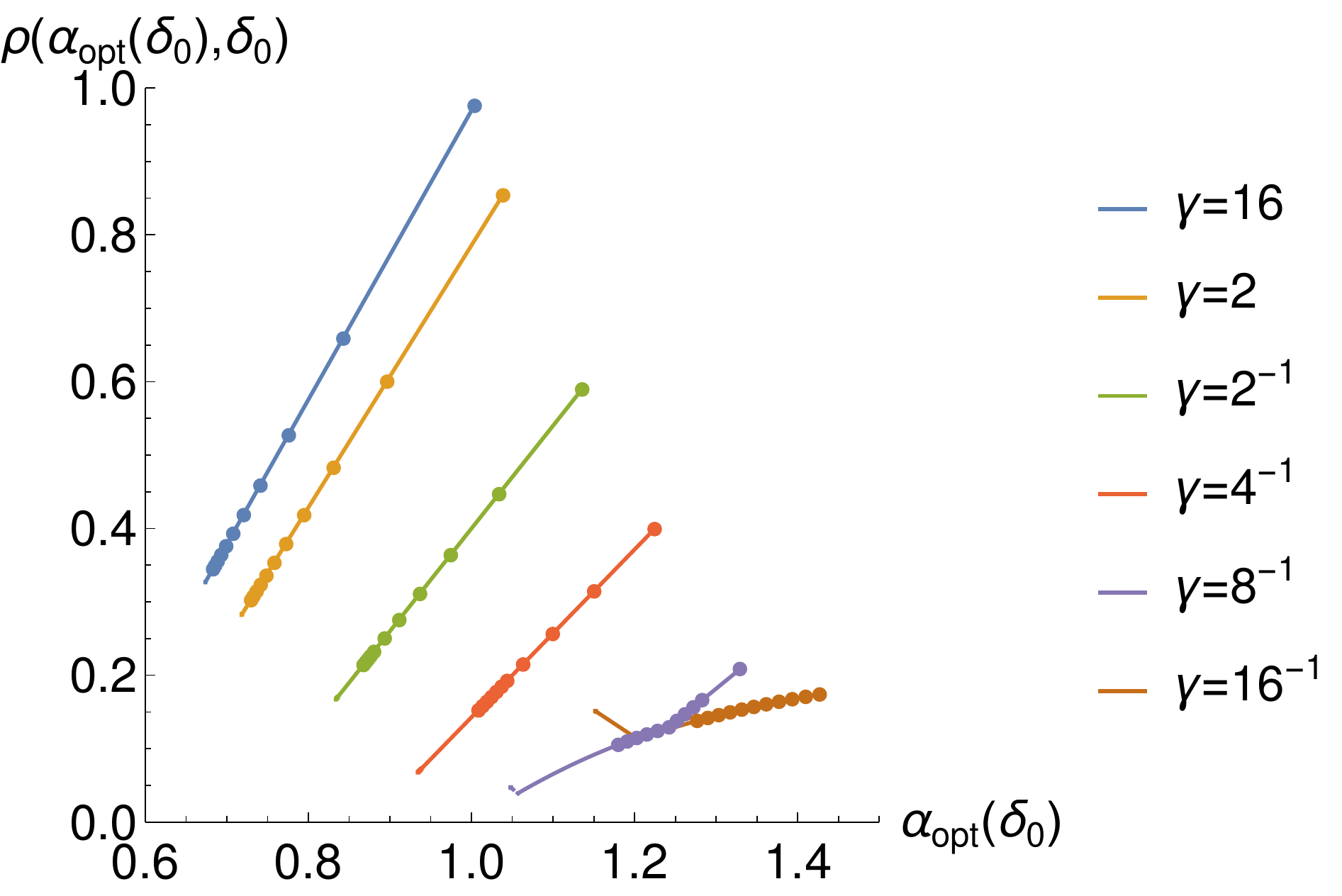}
    \caption{Measured spectral radius using a \emph{point} block-Jacobi
      smoother to solve a reaction-diffusion equation (points) overlayed on
      theoretically predicted values (solid line).}
    \label{fig:ExpRDPoint}
  \end{subfigure}
  \begin{subfigure}{0.49\textwidth}
    \includegraphics[width=\textwidth]{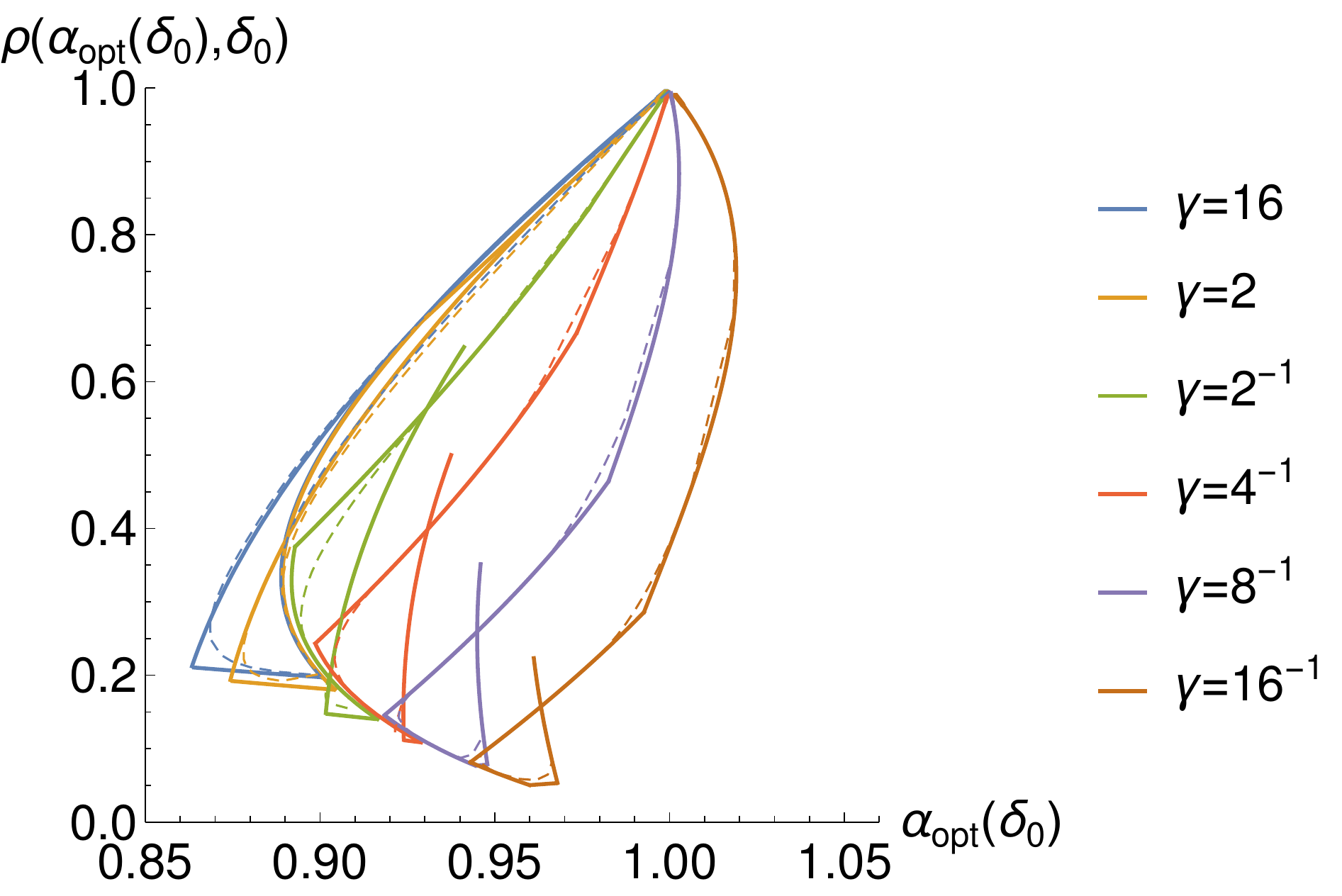}
    \caption{Measured spectral radius using a \emph{cell} block-Jacobi
      smoother to solve a reaction-diffusion equation (dashed line)
      overlayed on theoretically predicted values (solid line).}
    \label{fig:ExpRDCell}
  \end{subfigure}
  \caption{}
\end{figure}
Theoretically predicted parametric curves are shown for $\dd \in
[1,\infty)$, while numerically computed values are shown as points for
$\dd\in[1,50]$. The top right end of the curves corresponds to $\dd =
1$, while the bottom left end corresponds to $\dd \rightarrow \infty$.
In blue, we can see the measured $\rho_\text{opt}$,
$\alpha_\text{opt}$ as dots plotted on top of the predicted parametric
curve of the same color, for $\gamma = 16$. As expected, we see that a
large value of $\gamma$ almost reproduces the predicted curve that we
observed for the Poisson equation (c.f. Figure \ref{fig:ExpPoint}).
As we modify $\gamma$ and make it smaller (in orange, green, red,
violet and brown, for $\gamma=2, 2^{-1}, 4^{-1}, 8^{-1}, 16^{-1}$
respectively), the parametric curve moves towards the bottom right of
the figure, while keeping its shape until $\gamma \approx 7^{-1}$
where it features a point with discontinuous derivative. Keeping in
mind that the rightmost end of each curve corresponds to $\dd = 1$ and
the leftmost end corresponds to $\dd \rightarrow \infty$, we observe
that for any finite value of $\gamma$ the method is robust for any
value of $\dd$, i.e. the convergence factor remains bounded away from
1. Large values of $\gamma$ require underrelaxation, and small values
overrelaxation, and in between there are $\gamma$ values that require
both overrelaxation for small $\dd$ and underrelaxation for large
$\dd$ to be optimal. When $\gamma$ is very small, the regime becomes
insensitive to the values of $\dd$, which is expected since all the
terms in the bilinear form that describe derivatives are negligible in
comparison to the reaction term and even at very large values of
$\dd$, the \emph{point} block-Jacobi smoother can neutralize the
operator's dependency on $\dd$; see also the bottom curve in Figure
\ref{fig:rdaoptpoint} on the right.

\subsection{\emph{Cell} block-Jacobi smoother for the
  reaction-diffusion equation}
Results for the solution of a reaction-diffusion equation using a
two-level method with the \emph{cell} block-Jacobi smoother are shown
in Figure \ref{fig:ExpRDCell}. Theoretically predicted parametric
curves are shown for $\dd \in [1,\infty)$, while numerically computed
values are shown as dashed lines for $\dd\in[1,50]$.  All the curves
end at $\rho_\text{opt}=1$, $\alpha_\text{opt}=1$, while they begin at
smaller values of $\rho_\text{opt}$ for smaller values of $\gamma$.
Once again in blue, we show the measured $\rho_\text{opt}$,
$\alpha_\text{opt}$ with a dashed line, and the predicted value as a
solid line, for $\gamma=16$.  Such a large value of $\gamma$ is almost
equivalent to the Poisson equation and the shapes of the curves of
Figure \ref{fig:ExpCell} are reproduced.  When we set $\gamma$ to
smaller values (in orange, green, red, violet and brown, for
$\gamma=2, 2^{-1}, 4^{-1}, 8^{-1}, 16^{-1}$ respectively), we see that
convergence rapidly improves for values of $\dd$ that are order one,
including $\dd = 1$, represented as the beginning of the curve that
moves down and to the right of the figure. For moderate values of
$\dd$, very small values of $\gamma$ will even result in an exact
solver with the smoother alone. Convergence however still deteriorates
as $\dd \rightarrow \infty$, since, unlike the point block-Jacobi
smoother, the cell block-Jacobi smoother cannot neutralize the
operator’s dependency on $\dd$ for $\dd$ large.  The measured results
(dashed) and theoretically predicted ones (solid) show very good
agreement. Also, we see that small values of $\gamma$ can require
overrelaxation when $\dd$ becomes large.

\subsection{Higher dimensions and different geometries} \label{HDSec}
We now test our closed form optimized relaxation parameters from the
1D analysis in higher dimensions and on geometries and meshes that go
far beyond a simple tensor product generalization. We show in Figure
\ref{fig:higherdimensions}
\begin{figure}
  \tabcolsep0em
  \begin{tabular}{cc}
    \begin{tabular}{l}
      \includegraphics[width=0.2\textwidth]{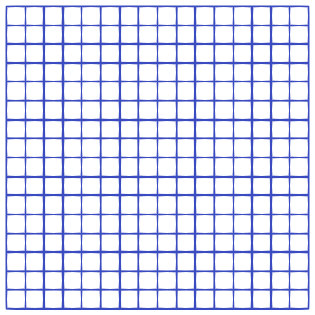}
    \end{tabular}&
    \begin{tabular}{l}
      \includegraphics[width=0.7\textwidth]{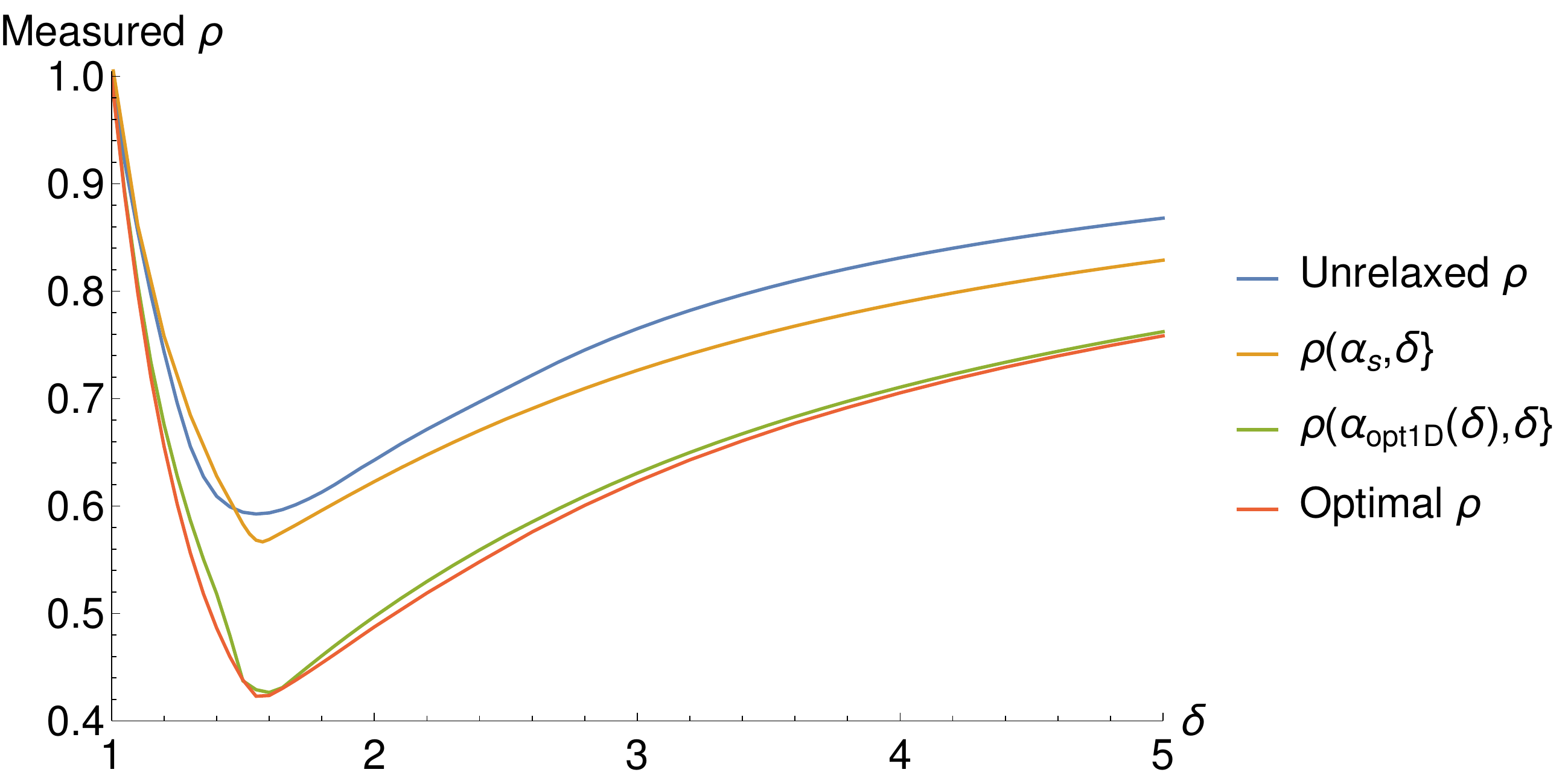}
    \end{tabular}\\
    \begin{tabular}{l}
      \includegraphics[width=0.2\textwidth]{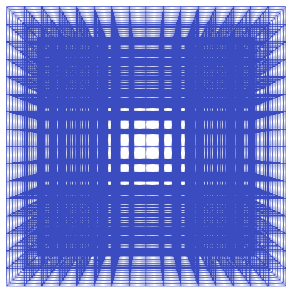}
    \end{tabular}&
    \begin{tabular}{l}
      \includegraphics[width=0.7\textwidth]{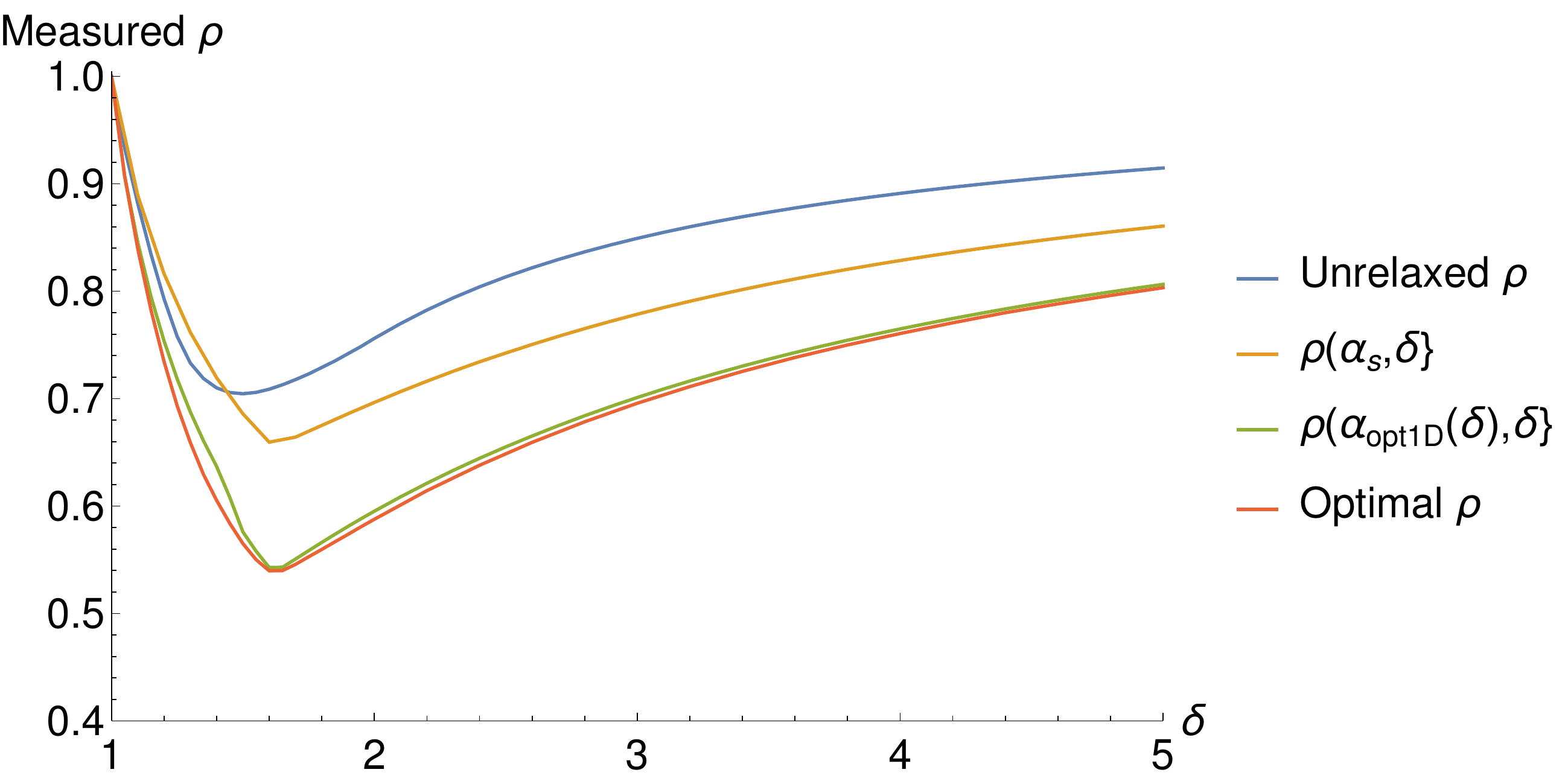}
    \end{tabular}\\
    \begin{tabular}{l}
      \includegraphics[width=0.29\textwidth]{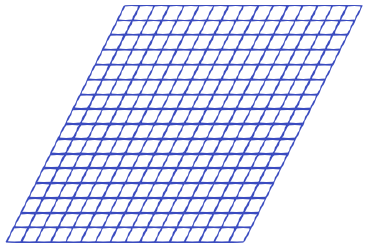}
    \end{tabular}&
    \begin{tabular}{l}
      \includegraphics[width=0.7\textwidth]{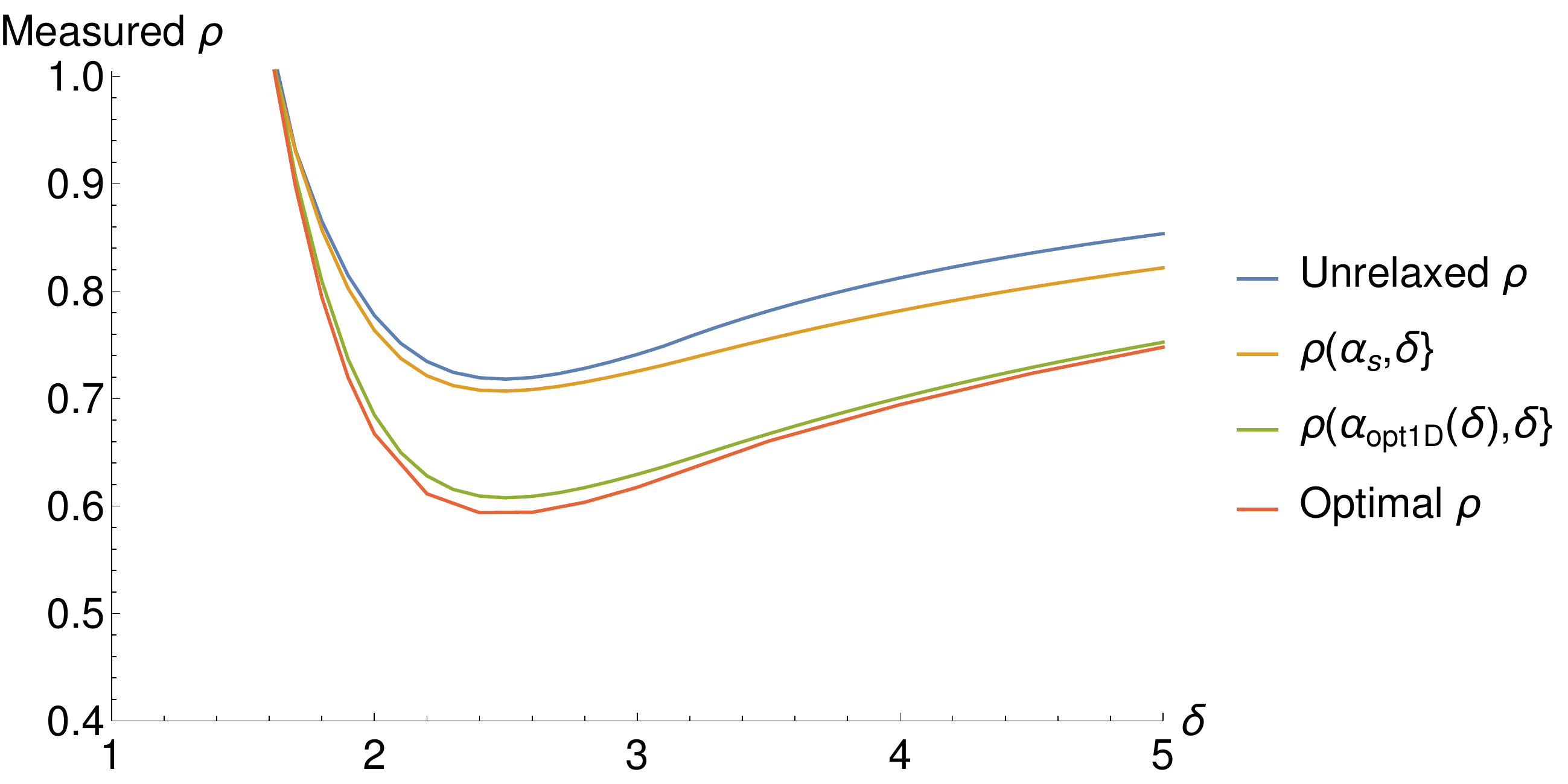}
    \end{tabular}\\
    \begin{tabular}{l}
      \includegraphics[width=0.23\textwidth]{DiskMesh}
    \end{tabular}&
    \begin{tabular}{l}
      \includegraphics[width=0.7\textwidth]{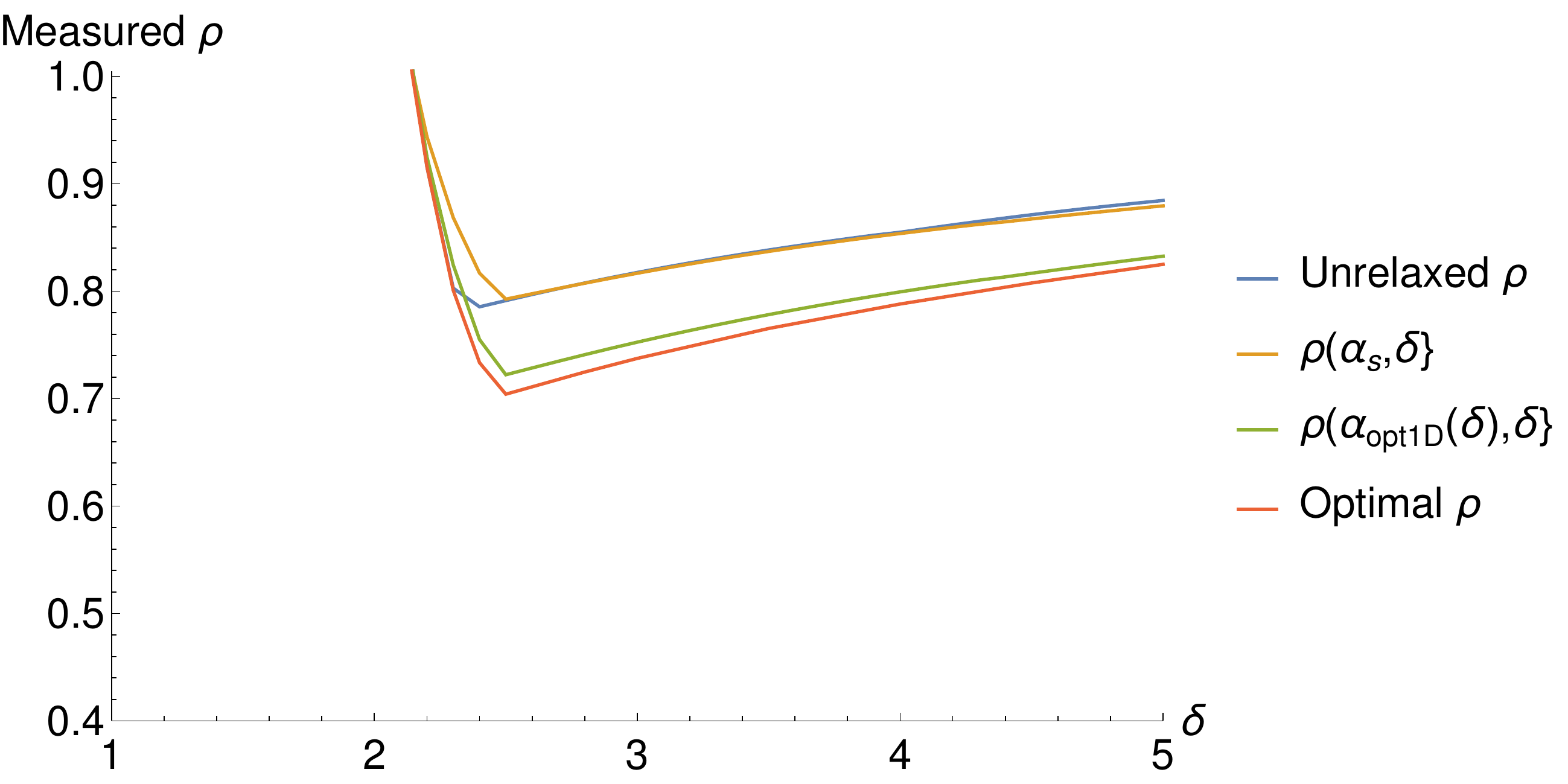}
    \end{tabular}
  \end{tabular}
  \caption{Comparison of the spectral radius of the two level operator
    for the Poisson problem on different geometries and meshes in higher
    dimensions. We compare the unrelaxed method, the relaxation
    $\alpha_s=2/3$ coming from the smoothing analysis alone, the optimized
    $\alpha_\text{opt1D}$ from Theorem \ref{thm:PoissonCell}, and the
    numerically optimal choice.}
  \label{fig:higherdimensions}
\end{figure}
a set of comparisons of the optimality of our closed form optimized
relaxation parameters for the Poisson problem, using \emph{cell}
block-Jacobi smoothers. In each case, we show the mesh used and a
comparison between the unrelaxed method, the relaxation of $2/3$
coming from the smoothing analysis alone, the one predicted by Theorem
\ref{thm:PoissonCell}, and the numerically best performing one. The
closeness between our closed form optimized parameters from the 1D
analysis and the numerically best working one in higher dimensions is
clear evidence that the seminal quote from P. W. Hemker in footnote
\ref{HemkerFootnote} is more than justified.

\section{Conclusion}

We optimized the relaxation parameter in two-level iterative methods
for solving symmetric interior penalty discontinuous Galerkin
discretized Poisson and reaction-diffusion equations using a
\emph{cell} block-Jacobi and a \emph{point} block-Jacobi smoother. Our
optimization for the complete two-level process shows that the
\emph{cell} block-Jacobi smoother leads to a more effective two-level
method for moderate sizes of the penalization parameter, while the
\emph{point} block-Jacobi smoother is superior for large penalization
parameters. Our analysis also reveals that the penalization parameter
in SIPG should not only be chosen large enough such that the DG method
converges, but it can be chosen to optimize the performance of the
associated iterative two-level solver. A good choice can lead to an
iterative solver that converges an order of magnitude faster than
other choices, and this even using the best possible relaxation
parameter in the smoother. While we performed our analysis in 1D, our
numerical experiments in higher dimensions on irregular domains with
irregular meshes clearly show that our closed form optimized
relaxation parameters work very well also in these situations, with
very close to best possible performance of the SIPG two level method.

\bibliographystyle{unsrt}
\bibliography{dgrdlfa}

\appendix

\section{Reaction-diffusion iteration operator eigenvalue coefficients
  using a \emph{point} block-Jacobi smoother}\label{apx:pointc}
\begin{align*}
  c_1 =& -8640 \alpha \dd^2 \gamma^4-14400 \alpha \dd^2 \gamma^3-2544
         \alpha \dd^2 \gamma^2 +6912 \alpha \dd \gamma^4 \\
       &+7776 \alpha \dd \gamma^3 -3744 \alpha \dd \gamma^2-992 \alpha \dd
         \gamma -864\alpha \gamma^4+288 \alpha \gamma^3 +2208 \alpha \gamma^2
  \\
       &-80 \alpha+6912 \dd^2 \gamma^4 +11520 \dd^2\gamma^3+3072 \dd^2
         \gamma^2-5184 \dd \gamma^4 -5184 \dd \gamma^3\\
       &+2688 \dd \gamma^2+1280 \dd \gamma+864 \gamma^4-1248 \gamma^2+128 \\
  c_2 =& -384 \gamma + 240 \alpha \gamma + 256 \dd \gamma - 160 \alpha
         \dd \gamma+1392 \alpha \gamma^2 - 2688 \dd \gamma^2 \\
       &+ 480 \alpha \dd \gamma^2+1536 \dd^2 \gamma^2-960 \alpha \dd^2
         \gamma^2 + 3456 \gamma^3 - 3168 \alpha \gamma^3 - 9216 \dd \gamma^3 \\
       &+12096 \alpha \dd \gamma^3 +2304 \dd^2 \gamma^3 - 5760 \alpha \dd^2
         \gamma^3 + 3456 \dd \gamma^4-3456 \alpha \dd \gamma^4 \\
       &- 6912 \dd^2 \gamma^4 +6912 \alpha \dd^2 \gamma^4 \\
  c_3 =& 96 \gamma^2 + 144 \alpha \gamma^2 - 192 \alpha \dd \gamma^2 +
         48 \alpha \dd^2 \gamma^2-576 \alpha \gamma^3 + 576 \dd \gamma^3 \\
       &+ 864 \alpha \dd \gamma^3 -576 \alpha \dd^2 \gamma^3-864 \gamma^4 +
         864 \alpha \gamma^4 + 1728 \dd \gamma^4 - 3456 \alpha \dd \gamma^4 \\
       &+1728 \alpha \dd^2 \gamma^4 \\
  c_4 =& 2985984 \alpha^2 \dd^4 \gamma^8+9953280 \alpha^2 \dd^4
         \gamma^7+6469632 \alpha^2 \dd^4 \gamma^6-3041280 \alpha^2 \dd^4
         \gamma^5 \\
       &+278784 \alpha^2 \dd^4 \gamma^4-5971968 \alpha^2 \dd^3
         \gamma^8-18911232 \alpha^2 \dd^3 \gamma^7-9123840 \alpha^2 \dd^3
         \gamma^6 \\
       &+8487936 \alpha^2 \dd^3 \gamma^5-2442240 \alpha^2 \dd^3
         \gamma^4+353280 \alpha^2 \dd^3 \gamma^3+5971968 \alpha^2 \dd^2
         \gamma^8 \\
       &+18911232 \alpha^2 \dd^2 \gamma^7+8957952 \alpha^2 \dd^2
         \gamma^6-8543232 \alpha^2 \dd^2 \gamma^5+1833984 \alpha^2 \dd^2
         \gamma^4 \\
       &-1373184 \alpha^2 \dd^2 \gamma^3+100864 \alpha^2 \dd^2
         \gamma^2-746496 \alpha^2 \dd \gamma^8+248832 \alpha^2 \dd \gamma^7 \\
       &+10368000 \alpha^2 \dd \gamma^6+13906944 \alpha^2 \dd
         \gamma^5+2062080 \alpha^2 \dd \gamma^4+856320 \alpha^2 \dd \gamma^3 \\
       &-276480 \alpha^2 \dd \gamma^2+8192 \alpha^2 \dd \gamma-248832
         \alpha^2 \gamma^7-829440 \alpha^2 \gamma^6+359424 \alpha^2 \gamma^5 \\
       &+2062080 \alpha^2 \gamma^4+734976 \alpha^2 \gamma^3+195072 \alpha^2
         \gamma^2-9216 \alpha^2 \gamma+256 \alpha^2 \\
  c_5 =& 11943936 \alpha^2 \dd^4 \gamma^7+17915904 \alpha^2 \dd^4
         \gamma^6-6967296 \alpha^2 \dd^4 \gamma^5+608256 \alpha^2 \dd^4
         \gamma^4 \\
       &-21897216 \alpha^2 \dd^3 \gamma^7-25214976 \alpha^2 \dd^3
         \gamma^6+25712640 \alpha^2 \dd^3 \gamma^5-5981184 \alpha^2 \dd^3
         \gamma^4 \\
       &+457728 \alpha^2 \dd^3 \gamma^3+20901888 \alpha^2 \dd^2
         \gamma^7+25049088 \alpha^2 \dd^2 \gamma^6-20542464 \alpha^2 \dd^2
         \gamma^5 \\
       &+10243584 \alpha^2 \dd^2 \gamma^4-2339328 \alpha^2 \dd^2
         \gamma^3+83968 \alpha^2 \dd^2 \gamma^2-3732480 \alpha^2 \dd \gamma^8
  \\
       &-17169408 \alpha^2 \dd \gamma^7-12773376 \alpha^2 \dd
         \gamma^6+12690432 \alpha^2 \dd \gamma^5-2449152 \alpha^2 \dd \gamma^4
  \\
       &+2492160 \alpha^2 \dd \gamma^3-313344 \alpha^2 \dd \gamma^2+4096
         \alpha^2 \dd \gamma+746496 \alpha^2 \gamma^8 \\
       &+1741824 \alpha^2 \gamma^7-1575936 \alpha^2 \gamma^6-4810752 \alpha^2
         \gamma^5+126720 \alpha^2 \gamma^4 -292608 \alpha^2 \gamma^3\\
       &+201216 \alpha^2 \gamma^2-10752 \alpha^2 \gamma \\
  c_6 =& -5971968 \alpha^2 \dd^4 \gamma^8-7962624 \alpha^2 \dd^4
         \gamma^7+16920576 \alpha^2 \dd^4 \gamma^6-4866048 \alpha^2 \dd^4
         \gamma^5 \\
       &+382464 \alpha^2 \dd^4 \gamma^4+11943936 \alpha^2 \dd^3
         \gamma^8+11943936 \alpha^2 \dd^3 \gamma^7-36163584 \alpha^2 \dd^3
         \gamma^6 \\
       &+23003136 \alpha^2 \dd^3 \gamma^5-4174848 \alpha^2 \dd^3
         \gamma^4+89088 \alpha^2 \dd^3 \gamma^3- 11943936 \alpha^2 \dd^2
         \gamma^8\\
       &-10948608 \alpha^2 \dd^2 \gamma^7+35997696 \alpha^2 \dd^2 \gamma^6-
         19491840 \alpha^2 \dd^2 \gamma^5+11828736 \alpha^2 \dd^2 \gamma^4\\
       &-685056 \alpha^2 \dd^2 \gamma^3- 512 \alpha^2 \dd^2 \gamma^2+4478976
         \alpha^2 \dd \gamma^8-7464960 \alpha^2 \dd \gamma^7\\
       &- 35168256 \alpha^2 \dd \gamma^6-1852416 \alpha^2 \dd
         \gamma^5-8808192 \alpha^2 \dd \gamma^4+ 1552128 \alpha^2 \dd
         \gamma^3\\
       &+4976640 \alpha^2 \gamma^7+8792064 \alpha^2 \gamma^6- 663552 \alpha^2
         \gamma^5+105984 \alpha^2 \gamma^4-988416 \alpha^2 \gamma^3\\
       &+2304 \alpha^2 \gamma^2 \\
  c_7 =& -11943936 \alpha^2 \dd^4 \gamma^7+5971968 \alpha^2 \dd^4
         \gamma^6- 995328 \alpha^2 \dd^4 \gamma^5+55296 \alpha^2 \dd^4
         \gamma^4\\
       &+21897216 \alpha^2 \dd^3 \gamma^7- 22560768 \alpha^2 \dd^3
         \gamma^6+6137856 \alpha^2 \dd^3 \gamma^5-654336 \alpha^2 \dd^3
         \gamma^4\\
       &- 15360 \alpha^2 \dd^3 \gamma^3-20901888 \alpha^2 \dd^2
         \gamma^7+22063104 \alpha^2 \dd^2 \gamma^6- 10202112 \alpha^2 \dd^2
         \gamma^5\\
       &+2585088 \alpha^2 \dd^2 \gamma^4+84480 \alpha^2 \dd^2 \gamma^3+
         4478976 \alpha^2 \dd \gamma^8+17418240 \alpha^2 \dd \gamma^7\\
       &-10450944 \alpha^2 \dd \gamma^6+ 1907712 \alpha^2 \dd
         \gamma^5-4020480 \alpha^2 \dd \gamma^4-145152 \alpha^2 \dd \gamma^3\\
       &- 1492992 \alpha^2 \gamma^8-1990656 \alpha^2 \gamma^7+7382016
         \alpha^2 \gamma^6+ 3704832 \alpha^2 \gamma^5\\
       &+2080512 \alpha^2 \gamma^4+76032 \alpha^2 \gamma^3 \\
  c_8 =& 2985984 \alpha^2 \dd^4 \gamma^8-1990656 \alpha^2 \dd^4
         \gamma^7+ 497664 \alpha^2 \dd^4 \gamma^6-55296 \alpha^2 \dd^4
         \gamma^5\\
       &+2304 \alpha^2 \dd^4 \gamma^4- 5971968 \alpha^2 \dd^3
         \gamma^8+6967296 \alpha^2 \dd^3 \gamma^7-2488320 \alpha^2 \dd^3
         \gamma^6\\
       &+ 359424 \alpha^2 \dd^3 \gamma^5-18432 \alpha^2 \dd^3
         \gamma^4+5971968 \alpha^2 \dd^2 \gamma^8- 7962624 \alpha^2 \dd^2
         \gamma^7\\
       &+3483648 \alpha^2 \dd^2 \gamma^6-940032 \alpha^2 \dd^2 \gamma^5+
         50688 \alpha^2 \dd^2 \gamma^4-3732480 \alpha^2 \dd \gamma^8\\
       &+7216128 \alpha^2 \dd \gamma^7+ 248832 \alpha^2 \dd \gamma^6+1216512
         \alpha^2 \dd \gamma^5-55296 \alpha^2 \dd \gamma^4 \\
       &-4727808 \alpha^2 \gamma^7-1824768 \alpha^2 \gamma^6-580608 \alpha^2
         \gamma^5+20736 \alpha^2 \gamma^4 \\
  c_9 =& -746496 \alpha^2 \dd \gamma^8-248832 \alpha^2 \dd
         \gamma^7+746496 \alpha^2 \gamma^8+ 248832 \alpha^2 \gamma^7 \\
  c_{10} =& 6912 \dd^2 \gamma^4+11520 \dd^2 \gamma^3+3072 \dd^2
            \gamma^2-5184 \dd \gamma^4-5184 \dd \gamma^3\\
       &+ 2688 \dd \gamma^2+1280 \dd \gamma+864 \gamma^4-1248 \gamma^2+128 \\
  c_{11} =& -6912 \dd^2 \gamma^4+2304 \dd^2 \gamma^3+1536 \dd^2
            \gamma^2+3456 \dd \gamma^4-9216 \dd \gamma^3- 2688 \dd \gamma^2\\
       &+256 \dd \gamma+3456 \gamma^3-384 \gamma \\
  c_{12} =& 1728 \dd \gamma^4+576 \dd \gamma^3-864 \gamma^4+96 \gamma^2
\end{align*}

\section{Reaction-diffusion iteration operator eigenvalue coefficients
  using a \emph{cell} block-Jacobi smoother}\label{apx:cellc}
\begin{align*} 
  c_1 =& 16 (-144 \alpha \dd^3 \gamma^4-192 \alpha \dd^3
         \gamma^3-36 \alpha \dd^2 \gamma^4- 216 \alpha \dd^2 \gamma^3\\
       &-170 \alpha \dd^2 \gamma^2+72 \alpha \dd \gamma^4+ 96\alpha \dd
         \gamma^3-84 \alpha \dd \gamma^2-50 \alpha \dd \gamma+36 \alpha
         \gamma^3\\
       &+ 60 \alpha \gamma^2-4\alpha+144 \dd^3 \gamma^4+192 \dd^3 \gamma^3-36
         \dd^2 \gamma^4+ 96 \dd^2 \gamma^3+176 \dd^2 \gamma^2-24 \dd \gamma^3\\
       &+12 \dd \gamma^2+48 \dd \gamma- 3 \gamma^2+4 ) \\
  c_2 =& 16 (144 \alpha \dd^3 \gamma^4-96 \alpha \dd^3 \gamma^3+72
         \alpha \dd^2 \gamma^4+ 216 \alpha \dd^2 \gamma^3-46 \alpha \dd^2
         \gamma^2\\
       &-72 \alpha \dd \gamma^4+ 60 \alpha \dd \gamma^3+72 \alpha \dd
         \gamma^2-4 \alpha \dd \gamma-36 \alpha \gamma^3+ 12 \alpha \gamma^2\\
       &+6 \alpha \gamma-144 \dd^3 \gamma^4+96 \dd^3 \gamma^3- 240 \dd^2
         \gamma^3+64 \dd^2 \gamma^2-108 \dd \gamma^2+8 \dd \gamma-12\gamma )\\
  c_3 =& 16 (-36 \alpha \dd^2 \gamma^4-12 \alpha \dd \gamma^3+36 \dd^2
         \gamma^4 +24 \dd \gamma^3+3 \gamma^2 ) \\
  c_4 =& 1024 \alpha^2 \gamma^2 (5184 \dd^6 \gamma^6+13824 \dd^6
         \gamma^5 +9216 \dd^6 \gamma^4-18144 \dd^5 \gamma^6-43200 \dd^5
         \gamma^5 \\
       &-17136 \dd^5 \gamma^4+10944 \dd^5 \gamma^3+21060 \dd^4 \gamma^6+36720
         \dd^4 \gamma^5 -16236 \dd^4 \gamma^4\\
       &-33624 \dd^4 \gamma^3+4665 \dd^4 \gamma^2-9072 \dd^3 \gamma^6 -864
         \dd^3 \gamma^5+41760 \dd^3 \gamma^4+23292 \dd^3 \gamma^3\\
       &-16140 \dd^3 \gamma^2+858 \dd^3 \gamma+1944 \dd^2 \gamma^6-3672 \dd^2
         \gamma^5-12384 \dd^2 \gamma^4 +7524 \dd^2 \gamma^3\\
       &+15018 \dd^2 \gamma^2-3096 \dd^2 \gamma+61 \dd^2+1836 \dd \gamma^5
         +2592 \dd \gamma^4-2340 \dd \gamma^3\\
       &-1116 \dd \gamma^2+2931 \dd \gamma-228 \dd+432 \gamma^4 +1080
         \gamma^3+504 \gamma^2-72 \gamma+219) \\
  c_5 =& 1024 \alpha^2 \gamma^2 (-10368 \dd^6 \gamma^6-6912 \dd^6
         \gamma^5 +9216 \dd^6 \gamma^4+33696 \dd^5 \gamma^6+5184 \dd^5
         \gamma^5\\
       &-46944 \dd^5 \gamma^4 +10656 \dd^5 \gamma^3-37584 \dd^4
         \gamma^6+23760 \dd^4 \gamma^5+65988 \dd^4 \gamma^4 \\
       &-48960 \dd^4 \gamma^3+4518 \dd^4 \gamma^2+16848 \dd^3 \gamma^6-34776
         \dd^3 \gamma^5 -23616 \dd^3 \gamma^4+65916 \dd^3 \gamma^3\\
       &-19836 \dd^3 \gamma^2+834 \dd^3 \gamma -3888 \dd^2 \gamma^6+14040
         \dd^2 \gamma^5-4752 \dd^2 \gamma^4-26532 \dd^2 \gamma^3 \\
       &+24900 \dd^2 \gamma^2-3498 \dd^2 \gamma+56 \dd^2-3672 \dd
         \gamma^5+1944 \dd \gamma^4 +2772 \dd \gamma^3\\
       &-8028 \dd \gamma^2+3960 \dd \gamma-222 \dd-864 \gamma^4-432 \gamma^3
         +576 \gamma^2-756 \gamma+216) \\
  c_6 =& 1024 \alpha^2 \gamma^2 (5184 \dd^6 \gamma^6-6912 \dd^6 \gamma^5
         +2304 \dd^6 \gamma^4-12960 \dd^5 \gamma^6+36288 \dd^5 \gamma^5\\
       &-18864 \dd^5 \gamma^4 +2592 \dd^5 \gamma^3+12312 \dd^4 \gamma^6-54000
         \dd^4 \gamma^5+52380 \dd^4 \gamma^4 -15912 \dd^4 \gamma^3\\
       &+1041 \dd^4 \gamma^2-6480 \dd^3 \gamma^6+30888 \dd^3 \gamma^5 -54720
         \dd^3 \gamma^4+33012 \dd^3 \gamma^3-5640 \dd^3 \gamma^2\\
       &+156 \dd^3 \gamma +1296 \dd^2 \gamma^6-11880 \dd^2 \gamma^5+19476
         \dd^2 \gamma^4-25236 \dd^2 \gamma^3 +10038 \dd^2 \gamma^2\\
       &-762 \dd^2 \gamma+4 \dd^2+1296 \dd \gamma^5-6480 \dd \gamma^4 +3636
         \dd \gamma^3-6228 \dd \gamma^2+1209 \dd \gamma\\
       &-12 \dd+324 \gamma^4-1080 \gamma^3 +36 \gamma^2-684 \gamma+6) \\
  c_7 =& 1024 \alpha^2 \gamma^2 (-2592 \dd^5 \gamma^6+1728 \dd^5
         \gamma^5 +3888 \dd^4 \gamma^6-6480 \dd^4 \gamma^5+1548 \dd^4
         \gamma^4\\
       &-1296 \dd^3 \gamma^6 +4536 \dd^3 \gamma^5-4896 \dd^3 \gamma^4+468
         \dd^3 \gamma^3+1296 \dd^2 \gamma^6 +1512 \dd^2 \gamma^5\\
       &+2808 \dd^2 \gamma^4-1548 \dd^2 \gamma^3+48 \dd^2 \gamma^2 +1080 \dd
         \gamma^5+1944 \dd \gamma^4+1116 \dd \gamma^3\\
       &-180 \dd \gamma^2+216 \gamma^4 +432 \gamma^3+180 \gamma^2) \\
  c_8 =& 1024 \alpha^2 \gamma^2 (324 \dd^4 \gamma^6+216 \dd^3 \gamma^5
         -648 \dd^2 \gamma^6+36 \dd^2 \gamma^4-540 \dd \gamma^5-108 \gamma^4)
  \\
  c_9 =& 8(2 \dd \gamma+1)(6 \dd \gamma+1)(3 (8 \dd-1) \gamma^2+32 \dd
         \gamma-3 \gamma^2+8) \\
  c_{10} =& -64 \gamma (2 \dd \gamma+1) (6 \dd \gamma+1) (\dd (3
            \gamma-2)+3) \\
  c_{11} =& 48 \gamma^2 (2 \dd \gamma+1)(6 \dd \gamma+1)
\end{align*}

\end{document}